\UseRawInputEncoding  

\documentclass[11pt,a4paper,reqno]{amsart}

\usepackage[margin=2.5cm]{geometry} 
\usepackage{setspace}
\usepackage{datetime}
\usepackage{xcolor}
\usepackage{array}
\usepackage{enumitem}
\usepackage{indentfirst}
\usepackage[normalem]{ulem}

\usepackage{amsmath,amsfonts,amssymb,mathrsfs}

\usepackage{graphicx}
\usepackage{float}      
\usepackage{makecell}

\usepackage[numbers,sort&compress]{natbib}
\usepackage{subcaption}  
\usepackage{tikz}
\usetikzlibrary{%
	arrows,arrows.meta,positioning,calc,fit,patterns,plotmarks,
	shapes.geometric,shapes.misc,shapes.symbols,shapes.arrows,shapes.callouts,
	backgrounds,chains,topaths,trees,petri,mindmap,matrix,folding,fadings,through,
	scopes,decorations.fractals,decorations.shapes,decorations.text,
	decorations.pathmorphing,decorations.pathreplacing,decorations.footprints,
	decorations.markings,shadows
}

\tikzset{middlearrow/.style={
		decoration={markings, mark= at position 0.5 with {\arrow[very thick]{#1}} },
		postaction={decorate}
}}
\tikzset{lefttlearrow/.style={
		decoration={markings, mark= at position 0.7 with {\arrow[very thick]{#1}} },
		postaction={decorate}
}}
\tikzset{leftlearrow/.style={
		decoration={markings, mark= at position 0.55 with {\arrow[very thick]{#1}} },
		postaction={decorate}
}}

\usepackage{lipsum}
\makeatletter
\renewcommand\section{\@startsection{section}{1}%
	\z@{-1.0\linespacing\@plus-.7\linespacing}{1\linespacing}%
	{\normalfont\bfseries\large\centering}}
\renewcommand\subsection{\@startsection{subsection}{2}%
	\z@{-1\linespacing\@plus-.7\linespacing}{1\linespacing}%
	{\normalfont\large\bfseries\raggedright\hangindent=2em  \hangafter=1}}
\makeatother

\newtheorem{theorem}{Theorem}[section]

\newtheorem{definition}{Definition}
\newtheorem{lemma}[theorem]{Lemma}
\newtheorem{claim}{Claim}
\newtheorem{proposition}[theorem]{Proposition}
\newtheorem{observation}{Observation}

\newtheorem{problem}{Problem}

\usepackage{hyperref}
\hypersetup{
	colorlinks=true,
	linkcolor=black,
	urlcolor=black,
	citecolor=black,
	pdfborder={0 0 0}
}

\title{Arbitrary orientations of cycles in oriented graphs}

\author{Guanghui Wang}
\thanks{GHW: School of Mathematics,  Shandong University, Jinan 250100, China.   \texttt{ghwang@sdu.edu.cn}. The author was supported by National Natural Science Foundation of China (No.12231018) and State Key Laboratory of Cryptography and Digital Economy Security.}

\author{Yun Wang}
\thanks{YW: Data Science Institute, Shandong University, Jinan 250100, China. \texttt{yunwang@sdu.edu.cn}. The author was supported by China Postdoctoral Science Foundation (No. 2025M773101) and National Natural Science Foundation of China (No.12501489).}

\author{Zhiwei Zhang}
\thanks{ZWZ: Interdisciplinary Center, Shandong University, Jinan 250100, China. \texttt{zhiweizh@mail.sdu.edu.cn}.}   

\begin{document} 
	\maketitle
	\begin{abstract}
		{We show that every sufficiently large oriented graph $G$ with minimum indegree and outdegree both at least $(3|V(G)|-1)/8$ contains every orientation of a Hamilton cycle. This result improves the approximate bound established by Kelly and resolves a long-standing problem posed by H\"aggkvist and Thomason in 1995. The degree condition is tight and it can be improved to $(3|V(G)|-4)/8$ for Hamilton cycles that are nearly directed, generalizing a classic result by Keevash, K\"uhn and Osthus. Additionally, we derive a pancyclicity result for arbitrary orientations. More precisely, the above  degree condition suffices to guarantee the existence of cycles of every possible orientation and every possible length unless $G$ is isomorphic to one of the exceptional  oriented graphs. 
			
		}
	\end{abstract}
	
	\maketitle
	
	\noindent{\bf Keywords:}  \texttt{Orientation,  Hamilton cycle, pancyclicity, randomised embedding, robust outexpander}

	\section{Introduction}\label{SEC-introduction}
	
	Notation follows \cite{bang2009} so we only repeat a few definitions here (also see Section \ref{SEC-notation}). A \emph{digraph} is not allowed to have parallel edges or loops and an \emph{oriented graph} is a digraph with no cycle of length 2. The \emph{minimum semidegree} $\delta^0(G)$ of a digraph $G$ is the minimum of all the indegrees and outdegrees of the vertices in $G$. For a path or cycle in digraphs, we say that it is \emph{directed} if all its edges are oriented in the same direction and, it is \emph{antidirected} if it contains no directed path of length 2. Throughout the paper, we will write $n$ for the number of vertices of the (di)graph $G$ we are considering.

	A fundamental result of Dirac \cite{diracPLMS2} states that every graph $G$ with minimum degree $n/2$ contains a Hamilton cycle. A directed version of this result was obtained by Ghouila-Houri \cite{ghouilaCRASP25}, who showed that any digraph $G$ with minimum semidegree $\delta^0(G)\geqslant  n/2$ contains a directed  Hamilton cycle. In \cite{keevashJLMS79} Keevash, K{\"u}hn and Osthus proved that the minimum semidegree threshold for an oriented graph $G$ containing a directed Hamilton cycle turns out to be $(3n-4)/8$.
	
	Instead of asking for a directed Hamilton cycle in digraphs, one may ask under certain conditions, if it is possible to show a digraph containing every possible  orientation of a Hamilton cycle. For tournaments, Rosenfeld \cite{rosenfeldJCTB16} conjectured that the directed Hamilton cycle is the only orientation of a Hamilton cycle that can be avoided by tournaments on arbitrarily many vertices. This conjecture was settled by Thomason \cite{thomasonTAMS296} who proved that every tournament on $n\geqslant  2^{128}$ vertices contains every oriented  Hamilton cycle except possibly the directed Hamilton cycle. In \cite{havetJCTB80} Havet showed that the number of vertices above can be lowered to 68. The question for general digraphs and oriented graphs was considered by H\"aggkvist and Thomason in 1995. They   \cite{haggkvistJGT19} proved that a minimum semidegree $\delta^0(G)\geqslant  n/2+n^{5/6}$ ensures the existence of every possible orientation of a Hamilton cycle in a digraph $G$ and $\delta^0(G)\geqslant (5/12+o(1))n$  suffices if  $G$ is oriented \cite{haggkvist1997}. In \cite{debiasioEJC22},  Debiasio and Molla  got the exact minimum semidegree threshold for an antidirected Hamilton cycle in digraphs,  they showed that such a cycle exists if $\delta^0(G)\geqslant n/2+1$. Later, Debiasio, K\"uhn, Molla, Osthus and Taylor \cite{debiasioSJDM29} settled the problem by completely determining the exact threshold for every possible orientation of a Hamilton cycle in digraphs. They showed that $\delta^0(G)\geqslant  n/2$ will suffice if the orientation is not antidirected and they believe that it would be interesting to obtain an exact version of this result in oriented graphs. Kelly \cite{kellyEJC18} proved that  an oriented graph $G$ admits every possible orientation of a Hamilton cycle if $\delta^0(G)\geqslant  (3/8+o(1))n$. In this paper, we improve this approximate result by proving the following, which solves a problem of H\"aggkvist and Thomason in \cite{haggkvist1997}  (also see page 2 in \cite{haggkvistJGT19}). 
	
	\begin{theorem}\label{THM-arbitraryori}
		There exists an integer $n_0$ such that every oriented graph $G$ on $n\geqslant  n_0$ vertices with minimum semidegree  $\delta^0(G)\geqslant  (3n-1)/8$ contains every possible orientation of a Hamilton cycle.
	\end{theorem}

	The degree condition can be improved to $(3n-4)/8$ when the Hamilton cycle is close to being directed,  which generalizes the classic result by Keevash, K\"uhn and Osthus \cite{keevashJLMS79}. The following proposition shows that if the oriented Hamilton cycle has a long antidirected subpath, then the degree bound in Theorem \ref{THM-arbitraryori} is the best possible.

	\begin{proposition}\label{PROP-degreesharp}
		
		For any $n\geqslant  3$, there is an $n$-vertex oriented graph $G$  with minimum semidegree $\delta^0(G)= \lceil  (3n-1)/8\rceil-1$ which does not contain a cycle having an antidirected subpath of order larger than $3(n/4+1)$. In particular, it has no antidirected Hamilton cycle.
	\end{proposition}

	Let $G$ be the oriented graph depicted in Figure \ref{FIG-degreesharp}. We now  briefly prove Proposition \ref{PROP-degreesharp} by showing that every antidirected path of $G$ only uses the vertices in $W\cup X\cup Z$ or the vertices in $Y\cup X\cup Z$. Suppose $P=v_1v_2\cdots v_l$ is an antidirected path with $v_1v_2\in E(P)$, that is, the edge $v_1v_2$ is oriented from $v_1$ to $v_2$. If $v_1$ is embedded in $W\cup Z$, then $v_2$ must belong to $W\cup X$ as  $v_1v_2\in E(P)$ and then  $v_3\in W\cup Z$ as  $v_3v_2\in E(P)$. Continue this procedure until all vertices are embedded, we get that all vertices with odd indices must belong to $W\cup Z$ and vertices with even indices are in $W\cup X$. The case that $v_1$ is embedded in $X \cup Y$ can be verified similarly and this proves the proposition.

	\begin{figure}
		\centering
		\begin{tikzpicture}[black,line width=1pt,scale=1.3]
			\draw (0,1.6) circle (0.4);
			\draw (0,1.6) ellipse (1 and 0.6);
			\coordinate [label=center:$W$] (A) at (0,1.6);
			\draw (-2.3,0) ellipse (0.8 and 0.6);
			\coordinate [label=center:$Z$] (D) at (-2.3,0);
			\draw (0,-1.6) circle (0.4);
			\draw (0,-1.6) ellipse (1 and 0.6);
			\coordinate [label=center:$Y$] (C) at (0,-1.6);
			\draw (2.3,0) ellipse (0.8 and 0.6);
			\coordinate [label=center:$X$] (B) at (2.3,0);
			\draw [-stealth, line width=1.2pt] (0.05,1.2) -- (-0.05,1.2);
			\draw [-stealth, line width=1.2pt] (0.05,-1.2) -- (-0.05,-1.2);
			\draw[-stealth] [line width=2.5pt](1,1.3) -- (2,0.65);
			\draw[-stealth] [line width=2.5pt](-1,-1.3) -- (-2,-0.65);
			\draw[-stealth] [line width=2.5pt](2,-0.65) -- (1,-1.3);
			\draw[-stealth] [line width=2.5pt](-2,0.65) -- (-1,1.3);
			\filldraw[white](-1.45,0.1) circle (1pt)node[](u){};
			\path[draw, -stealth, line width=1.2pt] (u) edge[bend left=15] (1.45,0.1);      
			\filldraw[white](1.45,-0.1) circle (1pt)node[](v){};
			\path[draw, -stealth, line width=1.2pt] (v) edge[bend left=15] (-1.45,-0.1);
		\end{tikzpicture}
		\caption{\small The oriented graph $G$ in Proposition \ref{PROP-degreesharp}. The bold edges indicate that all possible edges are present and have the directed shown. The set $W$ has size $\lfloor n/4\rfloor$ and $|X|=|Z|=\lceil n/4\rceil$. Each of $W$ and $Y$ spans an almost regular tournament, that is, the indegree and outdegree of every vertex differ by at most one. The oriented graph induced by $X$ and $Z$ is an almost regular bipartite tournament. Table \ref{TAB-degreesharp} in Appendix \ref{APPSEC-table1} will help readers to better check every vertex of $G$ has the correct indegree and outdegree.} 
		\label{FIG-degreesharp}
	\end{figure}

	\medskip
	
	In this paper we extend	Theorem \ref{THM-arbitraryori} further to a pancyclicity result for arbitrary orientations. We first introduce some results on this topic. In \cite{kellyJCTB100}, Kelly, K\"uhn and Osthus showed that in oriented graphs the minimum semidegree condition $\delta^0(G)\geqslant (3n-4)/8$ gives not only a directed Hamilton cycle but also a directed cycle of every possible length.
	
	\begin{theorem}[\cite{kellyJCTB100}]\label{THM-directedpancyclic}
		There exists an integer $n_0$ such that every oriented graph $G$ on $n\geqslant n_0$ vertices with $\delta^0(G)\geqslant (3n-4)/8$ contains a directed cycle of length $t$ for all $3\leqslant t\leqslant n$.
	\end{theorem}

	Later, Kelly \cite{kellyEJC18} proved that the above result 
	with the addition of an error term in the degree condition can be extended to arbitrary orientations of cycles.  
	
	\begin{theorem}[\cite{kellyEJC18}]\label{THM-kellyanyorianyl}
		Given $\varepsilon>0$,  there exists $n_0$ such that if $G$ is an oriented graph on $n\geqslant n_0$ vertices with $\delta^0(G)\geqslant (3/8+\varepsilon)n$, then $G$ contains a cycle of every possible orientation and of every possible length.
	\end{theorem}
	
	For more results on pancyclicity and vertex pancyclicity, we refer the reader to Section \ref{SEC-pancyclicity} and \cite{kellyEJC18,kellyJCTB100}. In this paper, we show that the minimum semidegree condition in Theorem \ref{THM-arbitraryori}  implies pancyclicity unless $G$ is isomorphic to an exception. More precisely, the following result says that the error term in the degree condition of Theorem \ref{THM-kellyanyorianyl} can be omitted. 
	
	\begin{theorem}\label{THM-anyorianylength}
		There exists an integer $n_0$ such that if $G$ is an oriented graph on $n\geqslant  n_0$ vertices with  $\delta^0(G)\geqslant  (3n-1)/8$, then one of the following holds.
		
		(i) $G$ contains a cycle of every possible orientation and of every possible length.
		
		(ii) There exists $v\in V(G)$ such that $G-v$ is isomorphic to the oriented graph of order $8s+2$ depicted in Figure \ref{FIG-degreesharp}.
		
		In particular, if  $\delta^0(G)\geqslant  3n/8$, then $G$ contains a cycle of every possible orientation and of every possible length.
	\end{theorem}
	The oriented graph in Theorem \ref{THM-anyorianylength} (ii) has no antidirected cycle of order larger than $3(n/4+1)$. If not, let $L$ be such a long cycle and let $u$ be a source  of $L$ which is embedded into $W\cup Z$. Suppose $w\in V(L)$ is a vertex embedded into the set $Y$.  Clearly, the cycle $L$ has two antidirected paths $L[u,w]$ and $L[w,u]$ between $u$ and $w$ and one may assume w.l.o.g that $v\notin L[u,w]$. By the proof of Proposition \ref{PROP-degreesharp}, all vertices of $L[u,w]$ must belong to $W\cup X\cup Z$, contradicting the fact that $w$ is embedded in $Y$.

	For powers of a directed Hamilton cycle in digraphs, please see \cite{draganicJCTB158, draganicCPC30,bollobasJCTB50,debiasio2024} for more information.  For Hamilton cycles with prescribed  orientations in random digraphs, the reader is referred to \cite{friezeRSA54,friezeEJC27,frieze2019}. Another problem related to the topic of this paper is an oriented version of graph discrepancy, which 
	has received renewed attention in recent years. In general, given an oriented graph $G$, one is interested in finding a Hamilton cycle with
	as many edges as possible which are oriented in the direction of the cycle. For this topic, we refer the reader to \cite{freschiJCTB169,gishbolinerJGT103,guoarXiv2025}. For other topics of Hamilton cycles in digraphs, please see the survey paper  \cite{kuhnejc33}.

	\medskip
	
	{\bf The rest of the paper is organized as follows:} We start out with Section \ref{SEC-notation} which contains some extra definitions that will be used in the paper. In Section \ref{SEC-sketch}, we outline the proofs of Theorem \ref {THM-arbitraryori} and Theorem \ref{THM-anyorianylength}. The important tools and some extra results are presented in Section \ref{SEC-pre} and we describe the structure of the given oriented graph in Section \ref{SEC-structureG}. Then Sections \ref{SEC-fewsink} and \ref{SEC-manysink} prove Theorem \ref {THM-arbitraryori} in two cases depending on the number of sinks of the given Hamilton cycle. Finally, we prove Theorem \ref{THM-anyorianylength} in Section \ref{SEC-pancyclicity} and  Section \ref{SEC-remark} contains additional remarks and open problems.
	\section{Notation}\label{SEC-notation}
	
	In this section, we provide some basic terminology and notation.  For integers $a<b$, $[a]$ will denote the set $\{1,2,\ldots, a\}$ and $[a,b]=\{a,a+1,\ldots,b\}$. We will often write $t= a\pm b$ which means that $t$ is an integer in the interval $[a-b,a+b]$. Let $f$ and $g$ be two mappings from $\mathbb{N}$ to $\mathbb{R}^+$. We write $f(n)=O(g(n))$ if there exists some constant $K>0$ such that $f(n) \leqslant Kg(n)$ for all sufficiently large $n$.

	Let $G=(V,E)$ be a digraph.  The order of $G$ is the number of vertices in $G$,  denoted by $|V(G)|$ or $|G|$.  We will denote an edge oriented from $u$ to $v$ by $uv$. We use $G[X]$ to denote the subdigraph induced by a vertex set $X$. Let $G-X = G[V\backslash X]$.  
	For $v\in V(G)$, we denote the set of outneighbors and inneighbors of $v$ in $G$ by $N^+(v)$ and $N^-(v)$ respectively, and we write $d^+(v)=|N^+(v)|$ and $d^-(v)=|N^-(v)|$ for the outdegree and indegree of $v$, respectively. For $S\subseteq V(G)$, we write $N^-(v,S)=N^-(v)\cap S$, $N^+(v,S)=N^+(v)\cap S$. Set $d^-(v,S)=|N^-(v,S)|$, $d^+(v,S)=|N^+(v, S)|$ and $d(v,S)=d^+(v,S)+d^-(v,S)$. For two subsets $A,B$ of $V(G)$, we use $E(A,B)$ to denote the set of edges from $A$ to $B$ and write $E(A)$ for $E(A,A)$. Set $e(A,B) = |E(A,B)|$ and $e(A) = |E(A)|$. 
	
	We normally describe an oriented path or cycle $R$ by listing its vertices $v_1, v_2,\ldots, v_l$ in the clockwise order and write $R = v_1v_2\cdots v_l(v_1)$. The predecessor of $v_i$ on $R$ is the vertex $v_{i-1}$ and is denoted by $v_i^-$. Similarly, the successor of $v_i$,  denoted by $v_i^+$, is the vertex $v_{i+1}$. The length of $R$ is the number of its edges. A $k$-cycle (reps., $k$-path) is a cycle (reps., path) of order $k$. For path $R$, the vertices $v_1$ and  $v_l$ are  called the initial and  terminal of  $R$, respectively.  An $(A,B)$-path is a path with the initial vertex belonging to $A$ and  the terminal vertex  belonging to $B$.  We use $R[v_i, v_j] = v_iv_{i+1}\cdots v_j$ to denote the subpath of $R$ from $v_i$ to $v_j$.    Given sets $X_1,X_2,\ldots,X_l\subseteq V(G)$, we say that $R= v_1v_2\cdots v_l$ has form $X_1X_2\cdots X_l$ if $v_i\in X_i$ for all $i\in[l]$. For simplicity, we will say that $R$ has form $X^{l}$ if every vertex of $R$ belongs to $X$.  A vertex $v\in R$ is called a sink of $R$ if $d^+(v,R)=0$ and a source if $d^-(v,R)=0$. We say that a vertex $v\in R$ is a normal vertex of $R$ if it is neither a sink nor a source  of $R$. Let $\sigma (R)$ denote the number of sinks of  $R$.   We say that $R=v_1v_2\cdots v_l$ is an antidirected out-path if it is an antidirected path with $v_1v_2\in E(R)$.
	

	Let $C$ be an oriented cycle.   We define $dist_C(v_i,v_j)$ to be the distance between $v_i$ and $v_j$ on the cycle, that is, $dist_C(v_i,v_j)$ equals the length of the path $C[v_i,v_j]$.  For two subpaths $P_1$ and $P_2$ of $C$, the distance of them on $C$, denoted by $dist_C(P_1,P_2)$, is the distance between the terminal vertex of $P_1$ and the initial vertex of $P_2$. A sequence $(P_1,P_2,\ldots, P_k)$ of disjoint subpaths of $C$ is a path partition of $C$ if $dist_C(P_i,P_{i+1}) = 1$ for all $i\in[k]$. 
	
	\section{Proof sketch} \label{SEC-sketch}
	
	We first observe that any oriented graph satisfying the degree condition in Theorems \ref{THM-arbitraryori} and \ref{THM-anyorianylength} must be a robust outexpander or it is ``close to'' the oriented graph in Figure \ref{FIG-degreesharp}, see Theorem \ref{LEM:non-expander-case}. Roughly speaking, a digraph is a robust outexpander if every vertex set $S$ of reasonable size has an outneighborhood which is at least a little larger than $S$ itself. Taylor \cite{taylor2013} proved that every sufficiently large robust outexpander of linear minimum semidegree contains every possible orientation of a Hamilton cycle. This allows us to restrict our attention to the ``close to'' case. 
	
	In the paper, we say that $G$ is ``close to'' the extremal oriented graph if there is a partition $(W,X,Y,Z)$ of $V(G)$ satisfying conditions \ref{EP1}-\ref{EP7}  as stated in Section \ref{SEC-structureG}. Roughly speaking, \ref{EP1} and \ref{EP3} imply that $WXYZW$ is close to an $n/4$-blowup of the directed 4-cycle.  Furthermore, $G$ has $O(\delta n)$ bad vertices (which have small degrees) by \ref{EP3}-\ref{EP5}  and the degrees of those bad vertices satisfy an acceptable lower bound due to \ref{EP6} and \ref{EP7}. The condition \ref{EP4} ensures that each of $G[W]$ and $G[Y]$ spans an almost regular tournament.  Finally, for every good vertex in $Z$, \ref{EP5} dictates that the number of its inneighbors and outneighbors within $X$ is approximately equal to $|X|/2$.

	We will next divide the proof of Theorem \ref{THM-arbitraryori} into two cases depending on the number of sinks of the Hamilton cycle $C$ and the proofs are provided in Sections \ref{SEC-fewsink} and \ref{SEC-manysink} respectively.

	\subsection{$C$ has few sinks}

	In this case there are many directed 13-paths on  $C$ which can be used to balance the sizes of $W,X,Y,Z$ and to cover bad vertices. Indeed, Lemma \ref{LEM-cyclepartationfew} shows that there exist disjoint short subpaths $L_1,L_2,\ldots,L_l$ of $C$ which together contain all sinks and sources  of $C$ and  $dist_C(L_i,L_{i+1})\equiv 0\pmod 4$ for all $i$.  Moreover, there are sufficiently many directed 13-paths  among these $l$ paths.

	We first embed several directed 13-paths into $G$ to balance the sizes of $X,Z$ and cover all bad vertices of $G$, see Lemma \ref{LEM-XZbalbance}, which is the main difficulty in this case. Then we move on to greedily embed the paths containing sinks or sources into $G[W]$. Note that this is possible as $G[W]$ is close to a regular tournament and the number of vertices on those paths are quite small compared to the size of $W$. After this we embed several directed 13-paths into $G$ again to further balance the sizes of $W,X$ and of $W,Y$ respectively.  Finally, embed the remaining paths in $\{L_1,L_2,\ldots, L_l\}$ and remove all internal vertices of the paths $L_i$ from $G$. Suppose  $(W^{\prime},X^{\prime},Y^{\prime},Z^{\prime})$ is obtained from $(W,X,Y,Z)$ by  performing the removing operation. Then we embed the remaining subpaths  of $C$  by applying the Blow-up Lemma. It should be noted that every remaining path is a directed path of order 3 modulo 4 as $dist_C(L_i,L_{i+1})\equiv 0\pmod 4$. Further, in the above process we can embed all paths $L_i$ such that their endvertices are good vertices in $W$.  Thus the endvertices of remaining subpaths can be embedded properly in $X^{\prime}\cup Z^{\prime}$ when we apply the Blow-up Lemma.  We illustrate the process described above in Figure \ref{FIG-embedfewsink}. 
	
\begin{figure}[H]
	\centering

	\begin{subfigure}[t]{1\linewidth}
		\centering
		\begin{tikzpicture}[black,line width=1pt,scale=0.7]
			\coordinate [label=center:$L_1$] () at (0,1);
			\coordinate [label=center:$L_2$] () at (0,-3);
			\coordinate [label=center:$R_1$] () at (7.7,-1);
			\coordinate [label=center:$R_2$] () at (-7.7,-1);
			\filldraw[black](-5,0) circle (1.5pt)node[label=above: $a_1$](a1){};
			\node at (5,0) [diamond,fill=black,draw,inner sep=0.5mm][label=above: $b_1$](b1){};
			\filldraw[black](5,-2) circle (1.5pt)node[label=below: $a_2$](a2){};
			\node at (-5,-2) [rectangle,draw,inner sep=0.7mm][label=below:$b_2$](b2){};
			\node at (0,0) [rectangle,draw,inner sep=0.7mm][](p1){};
			\filldraw[black](2.5,0) circle (1.5pt)node[](p2){};
			\node at (-2.5,0) [diamond,fill=black,draw,inner sep=0.5mm][](p3){};
			\filldraw[black](2.5,-2) circle (1.5pt)node[](p4){};
			\filldraw[black, fill=black](0,-2) circle (2pt)node[](p5){};
			\filldraw[black](-2.5,-2) circle (1.5pt)node[](p6){};
			\foreach \i/\j in {b1/p2,p2/p1,p3/p1,p3/a1,a2/p4,p4/p5,p5/p6,p6/b2}{\path[draw, -stealth,line width=1pt] (\i)--(\j);}
			\draw [-stealth,blue,line width=2pt] (b1) to [out=0,in=90]  (7,-1) to  [out=-90,in=0] (a2);
			
			\draw [-stealth,red,dashed,line width=1pt] (a1) to [out=180,in=90]  (-7,-1) to  [out=-90,in=180] (b2);
		\end{tikzpicture}
		\caption*{\small (a)}
\end{subfigure}
\begin{subfigure}[t]{1\linewidth}
	\centering
	\begin{tikzpicture}[black,line width=1pt,scale=0.8]
		\draw (0,3.2) ellipse (2.7 and 0.8);
		\draw (-4,0) ellipse (0.6 and 2);
		\draw (4,0) ellipse (0.6 and 2);
		\draw (0,-3.2) ellipse (2.7 and 0.8);
		\coordinate [label=center:$X^{\prime}$] (X) at (5,0);
		\coordinate [label=center:$Z^{\prime}$] (Z) at (-5,0);
		\coordinate [label=center:$W^{\prime}$] (W) at (3.2,3.2);
		\coordinate [label=center:$Y^{\prime}$] (Y) at (3
		.2,-3.2);

		\coordinate (x1) at (4,0.35);
		\coordinate (x2) at (4,1.05);
		\coordinate (x3) at (4,-0.35);
		\coordinate (x4) at (4,-1.05);
		\coordinate (y1) at (0.4,-3.2);
		\coordinate (y2) at (1.2,-3.2);
		\coordinate (y3) at (-0.4,-3.2);
		\coordinate (y4) at (-1.2,-3.2);
		\coordinate (z1) at (-4,0.35);
		\coordinate (z2) at (-4,1.05);
		\coordinate (z3) at (-4,-0.35);
		\coordinate (z4) at (-4,-1.05);
		\coordinate (a3) at (-0.25,3.2);\coordinate (b3) at (-0.95,3.2);
		
		\coordinate (x11) at (3.9,0.41);
		\coordinate (x21) at (3.9,1.15);
		\coordinate (x31) at (3.9,-0.29);
		\coordinate (x41) at (3.9,-0.99);
		\coordinate (y11) at (0.5,-3.1);
		\coordinate (y21) at (1.3,-3.1);
		\coordinate (y31) at (-0.3,-3.1);
		\coordinate (y41) at (-1.1,-3.1);
		\coordinate (z11) at (-3.9,0.25);
		\coordinate (z21) at (-3.9,0.95);
		\coordinate (z31) at (-3.9,-0.45);
		\coordinate (z41) at (-3.9,-1.15);
		\coordinate (a31) at (-0.3,3.15);
		\coordinate (b31) at (-1.05,3.15);
		\node at (-1.7,3.2) [diamond,fill=black,draw,inner sep=0.5mm][label=left: $b_1$](b1){};
		
		\node at (1.25,3.2) [rectangle,draw,inner sep=0.7mm][label=above:$b_2$](b2){};
		
		\coordinate (a1) at (1.75,3.2);
		\coordinate (a2) at (0.6,3.2);
		\coordinate (a11) at (1.8,3.25);
		\coordinate (a21) at (0.55,3.15);
		\foreach \i/\j in {a1/x21,x2/y41,y4/z41,z4/b2}{\path[draw, -stealth, red, dashed,  line width=1pt] (\i)--(\j);}
		\foreach \i/\j in {b1/x41,x4/y21,y2/z21,z2/b31,b3/x31,x3/y11,y1/z11,z1/a31,a3/x11,x1/y31,y3/z31,z3/a21}{\path[draw, -stealth, blue, line width=2pt] (\i)--(\j);}
		\node at (0,5.5) [rectangle,draw,inner sep=0.7mm][](p1){};
		\filldraw[black](-2.2,5) circle (1.5pt)node[](p2){};
		\node at (2.2,5) [diamond,fill=black,draw,inner sep=0.5mm][](p3){};
		\filldraw[black](0.3,4.4) circle (1.5pt)node[](p4){};
		\filldraw[black, fill=white](0.9,4.9) circle (2pt)node[](p5){};
		\filldraw[black](1.5,4.4) circle (1.5pt)node[](p6){};
		\foreach \i/\j/\k/\l in {b1/160/-120/p2,p2/50/180/p1,p3/130/0/p1,p3/-50/40/a11,a2/160/-120/p4,p4/70/-160/p5,p5/-20/130/p6,p6/-50/40/b2}{\draw [-stealth] (\i) to [out=\j,in=\k]  (\l);}
		\foreach \i/\n in {(-0.25,3.2)/a3,(-0.95,3.2)/b3,(4,0.35)/x1,(4,1.05)/x2,(4,-0.35)/x3,(4,-1.05)/x4,(-4,0.35)/z1,(-4,1.05)/z2,(-4,-0.35)/z3,(-4,-1.05)/z4,(0.4,-3.2)/y1,(1.2,-3.2)/y2,(-0.4,-3.2)/y3,(-1.2,-3.2)/y4}{\filldraw[black]\i circle (1.5pt)node[](\n){};}
		\filldraw[black](1.75,3.2) circle (1.5pt)node[label=right: $a_1$](a1){};
		\filldraw[black](0.6,3.2) circle (1.5pt)node[label=above: $a_2$](a2){};
	\end{tikzpicture}
	\caption*{(b)}
\end{subfigure}
	\caption{\small An illustration of how to embed $C$ when it contains few sinks. The black diamonds,  white squares and black circles indicate the sources, sinks  and normal vertices of $C$, respectively. The white circle on the path $L_2$ in (b) indicates that it is a bad vertex of $G$. In the proof we use  directed 13-paths to cover bad vertices but here we use a 5-path $L_2$ for an illustration. Note that the blue fat directed path $R_1$ has the same direction as the cycle $C$ but the red dashed directed path $R_2$ has the opposite direction. } 
		\label{FIG-embedfewsink}
\end{figure}

\subsection{$C$ has many sinks}

In this case, the cycle $C$ may contain a long antidirected subpath. However, as mentioned in Proposition \ref{PROP-degreesharp}, it is highly non-trivial to embed an antidirected path of order more than $3(n/4+1)$. In the proof, special edges are useful to overcome this difficulty. 
Here, an edge $e$ is \emph{special} for the partition $(W,X,Y,Z)$ of $V(G)$ if it belongs to $E(W\cup Z,Y\cup Z)\cup E(X\cup Y,W\cup X)$. It is not difficult to check that if there is a special edge from $Y$ to $X$, then one may embed a long antidirected path in $G[W\cup X\cup Z]$ first and then the special edge can help us to embed the remaining part of the path in $G[Y]$. In Section \ref{SEC:2matching}, we first claim that the degree condition $\delta^0(G)\geqslant (3n-1)/8$ guarantees the existence of special edges and each special edge can be extended to a connection path between $W$ and $Y$. It is worth noting that the directed 3-paths of $C$ in appropriate positions also play a crucial role in embedding. Indeed, Theorem \ref{THM-specialmanysink} shows that there is a copy of $C$ in $G$ if one of the following holds: 

(i) there are two disjoint special edges for the partition $(W,X,Y,Z)$; 

(ii) the partition has a special edge and $C$ contains a directed 3-path.


\medskip
\begin{figure}[H]
\centering
\begin{tikzpicture}[black,line width=1pt,scale=0.8]
	\draw[line cap=round, draw = gray!30,line width =0.45cm] (-5,0.9) -- (5,0.9);
	\draw[line cap=round, draw = gray!30,line width =0.45cm] (-5.1,0) -- (5.1,0);
	\draw[line cap=round, draw = gray!30,line width =0.45cm] (-5.1,-2) -- (5.1,-2);
	\draw[line cap=round, draw = gray!30,line width =0.45cm] (-5.1,-2) -- (0,-3.8);   
	\draw[line cap=round, draw = gray!30,line width =0.45cm] (5.1,-2) -- (0,-3.8);
	\draw[line cap=round, draw = gray!30,line width =0.45cm] (-5,1.1) -- (0,3.8);
	\draw[line cap=round, draw = gray!30,line width =0.45cm] (5,1.1) -- (0,3.8);
	\draw[line cap=round, draw = gray!30,line width =0.45cm] (-5,2) -- (-3,4);
	\draw[line cap=round, draw = gray!30,line width =0.45cm] (5,2) -- (2,4);
	
	\draw[draw=gray!30,fill =gray!30] (-5,1.35) arc (90:270:0.35);
	\draw[draw=gray!30,fill =gray!30] (5,0.65) arc (270:450:0.35);
	
	\draw [line cap=round,gray!40,line width=0.2cm]  (4.2,3.01) -- (4.4,2.4) -- (3.77,2.33); 
	\draw [line cap=round,gray!40,line width=0.2cm]  (-4.12,3.42) -- (-3.5,3.5) -- (-3.58,2.88);         
	\draw [line cap=round,gray!40,line width=0.2cm]  (4,0.4) -- (4.5,0) -- (4,-0.4); 
	\draw [line cap=round,gray!40,line width=0.2cm]  (0.25,1.3) -- (-0.25,0.9) -- (0.25,0.5);       
	\draw [line cap=round,gray!40,line width=0.2cm]  (-0.25,-1.6) -- (0.25,-2) -- (-0.25,-2.4); 
	\draw [line cap=round,gray!40,line width=0.2cm]  (-2.19,-2.61) -- (-2.77,-2.82) -- (-2.46,-3.36); 
	\draw [line cap=round,gray!40,line width=0.2cm]  (2.65,-2.475) -- (2.33,-2.98) -- (2.91,-3.2);
	\draw [line cap=round,gray!40,line width=0.2cm]  (-2.64,2.82) -- (-2.06,2.7) -- (-2.3,2.13);  
	\draw [line cap=round,gray!40,line width=0.2cm]  (2.24,3.02) -- (2.5,2.44) -- (1.88,2.33);   
	
	\draw[draw=black] (-3.5,4) rectangle (3.5,5); 
	\draw[ draw=black] (-3.5,-4) rectangle (3.5,-5);
	\draw[ draw=black] (-5,-3.5) rectangle (-6,3.5);
	\draw[draw=black] (5,-3.5) rectangle (6,3.5);
	
	\foreach \i/\j in {1,...,7}
	{  \draw [black] (-4.5+\i,5) -- (-4.5+\i,4);
		\draw [black] (-3.5+\i,-5) -- (-3.5+\i,-4);
		\draw [black] (-6,-3.5+\i) -- (-5,-3.5+\i);
		\draw [black] (6,-3.5+\i) -- (5,-3.5+\i);}

	\coordinate (W11) at (-3.15,4.15);
	\coordinate (W12) at (-3,4); 
	\coordinate (W13) at (-2.85,3.85); 
	\coordinate (Z11) at (-5.15,2.15);
	\coordinate (Z12) at (-5,2);
	\coordinate (Z13) at (-4.85,1.85);
	
	\coordinate (W71) at (2.15,4.15);
	\coordinate (W72) at (2,4); 
	\coordinate (W73) at (1.85,3.85); 
	\coordinate (X11) at (5.15,2.15);
	\coordinate (X12) at (5,2);
	\coordinate (X13) at (4.85,1.85);
	
	\coordinate (W31) at (0,3.9);
	\coordinate (W32) at (0,3.6);
	\coordinate (X31) at (5.1,1.1);
	\coordinate (X32) at (5.1,0.9);
	\coordinate (Z31) at (-5.25,1.15);
	\coordinate (Z32) at (-5.1,1.05);
	\coordinate (Z33) at (-5.1,0.75); 
	
	\coordinate (X41) at (5.2,0.2);
	\coordinate (X42) at (5.2,0.1); 
	\coordinate (X43) at (5.2,-0.2); 
	\coordinate (Z41) at (-5.2,0.2);
	\coordinate (Z42) at (-5.2,-0.05);
	
	\coordinate (Y41) at (0,-3.6);
	\coordinate (Y42) at (0,-3.9);
	\coordinate (X61) at (5.1,-1.8);
	\coordinate (X62) at (5.1,-2.0); 
	\coordinate (X63) at (5.25,-2.15); 
	\coordinate (Z61) at (-5.2,-1.8);
	\coordinate (Z62) at (-5.2,-2); 
	\foreach \i/\j in {W11/Z11,Z11/W12,W12/Z12,Z12/W13,W13/Z13,W71/X11,X11/W72,W72/X12,W73/X13,X41/Z41,X42/Z42,Z42/X43,W32/Z32,Z32/X32,X32/Z33,X31/W31,W31/Z31,Z62/Y42,Y42/X63,X63/Y41,X62/Z61,Z61/X61} {
		\path[line cap=round,draw = blue,line width=1pt] (\i)--(\j);
	}

	\node at (4,4.5) [black] {$W$};
	\node at (4,-4.5) [black] {$Y$};
	\node at (-6.4,3) [black] {$Z_k$};
	\node at (6.45,3) [black] {$X_k$};
	\node at (-6.45,-2) [black] {$Z_1$};
	\node at (6.45,-2) [black] {$X_1$};
	\node at (-6.45,-3.1) [black] {$Z_0$};
	\node at (6.45,-3.1) [black] {$X_0$};
	\node at (-6.4,0) [black] {$Z_{k^\ast}$};
	\node at (6.45,0) [black] {$X_{k^\ast}$};
	
	\foreach \i in {0,1,2} {
		\filldraw[black](6.4,1.35+0.15*\i) circle (0.4pt); 
		\filldraw[black](6.4,-0.85-0.15*\i) circle (0.4pt); 
		\filldraw[black](-6.4,1.35+0.15*\i) circle (0.4pt); 
		\filldraw[black](-6.4,-0.85-0.15*\i) circle (0.4pt); 
	}
\end{tikzpicture}
\caption{\small An illustration of how randomised embedding method may be applied to $G$.}
\label{FIG-wind}
\end{figure}

A further difficulty arises from the asymmetry of the sets $X$ and $Z$. To see this, \ref{EP5} shows that  almost all vertices of $Z$ have many in- and outneighbors in $X$. However, possibly every vertex of $X$ has indegree zero or outdegree zero in $Z$. For example, let $X_1,X_2$ be a partition of $X$ such that $ZX_1X_2Z$ is a blowup of the directed 3-cycle.  Then every vertex of $X_1$ (resp., $X_2$) has outdegree (resp., indegree) zero in $Z$. To overcome this main difficulty,  we first establish further partitions of $G$ and $C$
in Sections \ref{SEC:lemmaG} and \ref{SEC:lemmaC}, respectively, these results are referred to as ``Lemma for $G$'' and ``Lemma for $C$''.   More precisely, we apply the Diregularity Lemma to  $G[X\cup Z]$. Using the degree condition in \ref{EP5}, we show that there is an almost perfect matching between the clusters of $X$ and $Z$, say $Z_1X_1,Z_2X_2,\ldots,Z_{k^\ast}X_{k^\ast}$ and $X_{k^\ast+1}Z_{k^\ast+1},X_{k^\ast+2}Z_{k^\ast+2},\ldots,X_kZ_k$ are edges in the matching from $Z$ to $X$ and from $X$ to $Z$, respectively. We further randomly choose disjoint sets $W_1,W_2,\ldots, W_k$ and $Y_1,Y_2,\ldots, Y_k$ of $W$ and $Y$, respectively, such that the sets $W_i,X_i,Y_i,Z_i$ with $i\in[k]$ have the same size. We also construct a path partition $(Q_0P_1Q_1P_2Q_2\cdots P_tQ_t\cdots)$  of $C$  such that every tuple $(Q_{i-1},P_i,Q_i)$ can be divided into three types, see Definition \ref{DEF-Ptype}. Roughly speaking, $(Q_{i-1},P_i,Q_i)$ is of type I if we can not describe the exact orientation of $P_i$; it is of type II if $P_i$ is of constant order which can be embedded into $G[W]$ greedily and it is of type III if $P_i$ is an antidirected path.

Based on the above partitions, now we show how to  cover almost all vertices in $X\cup Z$, for further details, see  Claim \ref{CLM-embedpath} in Section \ref{SEC:embedC}. Inspired by the techniques used in \cite{kellyEJC18, taylor2013}, the paths of type I will be randomly embedded around a triangle $Z_iX_iY_i$ or $X_iZ_iW_i$ and the paths of type III will be embedded around an edge $Z_iX_i, Z_iW_{i}$ or $W_{i}X_i$. Here, ``a path $P$ is randomly embedded around $F$'' means that we start from a random vertex $V_i\in V(F)$  and assign the next vertex of $P$ to either the successor or the predecessor of $V_i$ in $F$ according to the orientation of the edge. It is worthwhile mentioning that the way of splitting  $C$  is totally different from the way used in \cite{kellyEJC18, taylor2013}. Moreover, in this paper, we embed paths not only around a triangle but also around an edge.

Then we use two short subpaths of $C$ which contain many sinks or many directed subpaths to cover bad vertices in $G$ and the remaining vertices in $X\cup Z$, see Steps 1 and 3 in the proof of Claim \ref{CLM-connectcycle}. After this,  the semidegree of the remaining $G[W]$ (resp., $G[Y]$) is still large enough. Then Lemmas \ref{LEM-semitoexpander} and \ref{LEM-Connect} imply that the remaining vertices  in $W\cup Y$ can be covered by two oriented Hamilton paths with given endvertices in the remaining $G[W]$ and $G[Y]$, respectively.  Therefore, one may obtain an embedding of $C$ in $G$.

\medskip

Now we mention the differences between the two cases. When $C$ has few sinks, the cycle $C$ has many directed 13-paths which can be used to cover bad vertices. However, if $C$ has many sinks, then it is no longer guaranteed that $C$ has many directed or antidirected paths. To overcome this, we will use both sinks and short directed paths to cover bad vertices.   Another difference is that when $C$ has few sinks we need to balance the sizes of $X$ and $Z$ first, but we do not proceed this when $C$ has many sinks. This is because when $C$ has few sinks the cycle is close to being directed and every directed path uses the same number of vertices in $X$ and $Z$. However, for the case that $C$ has many sinks, we have enough sinks and sources to cover the remaining vertices in $X$ and $Z$ no matter whether they have the same size.

\medskip

\textbf{Some ideas in the proof of Theorem \ref{THM-anyorianylength}.} Similar with the proof of Theorem \ref{THM-arbitraryori}, we prove Theorem \ref{THM-anyorianylength} by using the stability technique where we split the proof into two cases depending on whether $G$ has expansion property. For the case that  $G$ has no expansion property, Theorem \ref{THM-extremal} implies that $V(G)$ has a partition $(W,X,Y,Z)$ satisfying \ref{EP1}-\ref{EP7}. Thus if $|C|<n/5$, then it seems to be not difficult to embed $C$ into $G[W]$ as $G[W]$ is close to a regular tournament of order $|W|\approx n/4$. When $|C|\geqslant n/5$, we can randomly choose a  subgraph $G^\ast$ in $G$ with $|G^\ast|=|C|$ which inherits the properties. Then the problem of embedding $C$ into $G$ is equivalent to the problem of finding a Hamilton cycle $C$ in $G^\ast$, which can be done by Theorem \ref{THM-arbitraryori}. If the given oriented graph $G$ has expansion property, by applying several simple results on expansion we will show that a minimum semidegree $\delta^0(G)\geqslant 
0.346n$ is sufficient to guarantee the existence of any oriented cycle, see Theorem \ref{THM-expanderanylength}.

\section{Preparation}\label{SEC-pre}

\subsection{Probabilistic tools}

\begin{theorem}[{Chernoff Bound 1}]  \label{THM-chernoff1}
Suppose that $X_1,X_2, \ldots, X_s$ are independent random variables with $\mathbb{P}[X_i=1]=p$ and $\mathbb{P}[X_i=0]=1-p$. Let $X=\sum_{i=1}^sX_i$. Then for any $0\leqslant t\leqslant sp$,
$$\mathbb{P}[|X-\mathbb{E}[X]|>t] <  2e^{-t^2/(3\mathbb{E}[X])}.$$
\end{theorem}

The following lemma is an immediate consequence of  Chernoff Bound 1.
\begin{lemma}[\cite{debiasioEJC22}]\label{LEM-randompartition}
For any $\varepsilon > 0$, there exists $n_0$ such that if $G$ is a digraph on $n \geqslant  n_0$ vertices, then every set $S$ of $V(G)$ with $|S|\geqslant t$ contains a subset $T$ of order $t$  satisfying the following: 
\begin{align*}|d^{\pm}(v,T) - \frac{|T|}{|S|}d^{\pm}(v,S)| \leqslant  \varepsilon n \mbox{ and }|d^{\pm}(v,S\backslash T)- \frac{|S\backslash T|}{|S|}d^{\pm}(v,S)| \leqslant  \varepsilon n \mbox{ for each } v\in V(G).\end{align*}
\end{lemma}

\begin{theorem} [{Chernoff Bound 2}]\label{THM-chernoff2}
Let $X$ be a random variable determined by $s$ independent  trials $X_1,X_2,\ldots, X_s$ such that changing the outcome of any one trial can affect $X$ by at most $c$. Then for any $\lambda\geqslant  0$, $$\mathbb{P}[|X-\mathbb{E}(X)|>\lambda]<2e^{-\lambda^2/2c^2s}.$$
\end{theorem}

We use Chernoff Bound 2 to prove the following lemma, which allows us to embed a collection of short subpaths of the given Hamilton cycle into  the reduced oriented graph $R$ in such a way that each of the clusters receives a similar number of vertices. When we talk about ``a path $P$ is  randomly embedded around $F$'' means that we start from the specified initial vertex $V_i\in V(F)$ and assign the next vertex of $P$ to either the successor or the predecessor of $V_i$ in $F$ according to the orientation of the edge. This randomised embedding method is also used in \cite{kellyEJC18,taylor2013}. 

\begin{lemma}\label{LEM-randomwind}
Let $R$ be a digraph and let $F = V_1V_2\cdots V_kV_1$ be a cycle in $R$. For any $\varepsilon>0$, there exists $n_0$ such that  for any $n\geqslant n_0$ and a collection of oriented paths $\{P_1,P_2,\ldots,P_s\}$ with $\sum_{i=1}^s |P_i|\leqslant n$  and $|P_i|\leqslant  n^{1/3}$ for each $i\in[s]$, there is a randomised embedding around $F$ of these $s$ paths satisfying the following. Define $a(i)$ to be the number of vertices in $\bigcup_{i=1}^s V(P_i)$ assigned to $V_i$ by this embedding. Then for all $V_i \in V(R)$, we have 
$$\left|a(i) - \frac{\sum_{i=1}^s |P_i|}{k}\right| \leqslant  \varepsilon n.$$
\end{lemma}
\begin{proof} We assign the initial vertex of each $P_i$ to one of $V_1,V_2,\ldots, V_k$ independently and uniformly at random. Observe that the assignment of each $P_i$ can change the number of vertices assigned to any vertex of $R$ by at most $|P_i|\leqslant n^{1/3}$. Clearly  $\mathbb{E}[a(i)] = \sum_{i=1}^s |P_i| / k$  and $s\leqslant n$. 
By Theorem \ref{THM-chernoff2}, we have 
\begin{align*}
	\mathbb{P}[|a(i) - \sum_{i=1}^s |P_i| / k| > \varepsilon n]  <  2 e ^ {-(\varepsilon n)^2/2 (n^{1/3})^2  n} <  1/k,
\end{align*}
as $n$ is sufficiently large. Thus the probability that there exists some $V_i$ with $i\in[k]$ which does not have almost the expected number of vertices assigned to it is less than 1. Thus with positive probability the above randomised embedding   satisfies the conclusion of the lemma.
\end{proof}

\subsection{The Diregularity Lemma and Blow-up Lemma}
The \emph{density} of a bipartite graph $G=(X,Y)$ with vertex classes $X,Y$ is defined to be $$\rho(X,Y)=\frac{e(X,Y)}{|X||Y|},$$ where $e(X,Y)$ is the number of edges between $X$ and $Y$ in the graph $G$. 

Given $\varepsilon>0, d\in[0,1]$, we say that a bipartite graph $G=(X,Y)$ is

$\bullet$ \emph{$\varepsilon$-regular} if for all sets $A\subseteq X$ and $B\subseteq Y$ with $|A|\geqslant  \varepsilon|X|, |B|\geqslant  \varepsilon|Y|$ we have $$|\rho(A,B)-\rho(X,Y)|<\varepsilon;$$

$\bullet$ \emph{$(\varepsilon,d)$-regular} if it is $\varepsilon$-regular with density $\rho(X,Y)\geqslant  d$;

$\bullet$ \emph{$(\varepsilon,d)$-superregular} if it is $(\varepsilon,d)$-regular and additionally the degree $d_G(x,Y)\geqslant  d|Y|$ for every $x\in X$ and $d_G(y,X)\geqslant  d|X|$ for every $y\in Y$.

\medskip

Let $G=(X,Y)$ be a bipartite digraph where all edges are oriented from $X$ to $Y$. We define the pair $(X,Y)$ as (super)regular if it satisfies the (super)regularity condition in the underlying graph of  $G$.

The following simple observation is well-known, which shows that removing a small proportion of vertices from a regular pair $(A,B)$ can make it  become superregular.

\begin{lemma}\label{LEM-regularsuper}
Let $(X,Y)$ be an $(\varepsilon,d)$-regular pair. Then there exists $A\subseteq X, B\subseteq Y$ of size $|A|=(1-\varepsilon)|X|$ and $|B|= (1-\varepsilon)|Y|$ such that $(A,B)$ is $(2\varepsilon,d-2\varepsilon)$-superregular.
\end{lemma}

We will use the following degree form of the Diregularity Lemma with ``refinement property''.

\begin{lemma}[Diregularity Lemma, \cite{keevashJLMS79}]
\label{LEM-reguler}
For every $\varepsilon \in (0,1)$ and all numbers $M^{\prime}$, $M^{\prime\prime}$ there are numbers $M$ and $n_0$ such that if 

$\bullet$ $G$ is a digraph on $n \geqslant  n_0$ vertices,

$\bullet$ $U_0,U_1,\ldots,U_{M^{\prime\prime}}$ is a partition of $V(G)$, and 

$\bullet$ $d \in [0,1]$ is any real number,

then there is a partition $V_0,V_1,\ldots,V_k$ 
of $V(G)$ and a spanning subdigraph $G^\prime$ of $G$ such that the following statements hold:

$\bullet$ $M^{\prime}\leqslant  k \leqslant  M$,

$\bullet$ $|V_0| \leqslant  \varepsilon n$ and $|V_1|=\dots=|V_k|=m$,

$\bullet$ for each $V_i$ with $i\in[k]$ there exists some $U_j$ containing $V_i$,

$\bullet$  for each $x \in V(G)$, $d^\pm(x,G^{\prime}) > d^\pm(x)  - (d+\varepsilon)n$,

$\bullet$ for all $i\in[k]$ the digraph $G^{\prime}[V_i]$ is empty,

$\bullet$ for all $1 \leqslant  i\neq j \leqslant  k$ the bipartite graph $G^{\prime}[V_i,V_j]$ whose vertex classes are $V_i$ and $V_j$ and whose edges are all from $V_i$ to $V_j$ in $G^{\prime}$ is $\varepsilon$-regular and has density either $0$ or at least $d$.
\end{lemma}

The last condition of Lemma \ref{LEM-reguler} shows that all pairs of clusters are $\varepsilon$-regular in both directions, but possibly with different densities. For given clusters $V_1,V_2,\ldots,V_k$ and the  digraph $G^\prime$, the reduced digraph $R^\prime$ with parameters $(\varepsilon, d)$ is the digraph  whose vertex set is $\{V_1,V_2,\ldots,V_k\}$ and in which $V_iV_j$ is an edge of $R^\prime$ if and only if the bipartite graph $G^{\prime}[V_i,V_j]$ is $(\varepsilon,d)$-regular. Note that when $G$ is an oriented graph its reduced digraph $R^\prime$ is not necessarily oriented. 

The next lemma (which is essentially from \cite{kellyCPC17,keevashJLMS79}) shows that if we retain precisely one direction with suitable probability for each 2-cycle in $R^\prime$, then we obtain a spanning \emph{oriented} subgraph $R\subseteq R^\prime$ possessing desirable properties. Its  proof is provided in Appendix \ref{APPSEC-proforien}.

\begin{lemma}\label{LEM-orientedreduced}
For each $0<\varepsilon <1$, there exists  $n_0$ such that the following holds. Suppose  $0\leqslant d\leqslant 1$ is a constant with $\varepsilon \leqslant d/2$. Let $G$ be an oriented graph on $n\geqslant n_0$ vertices and let $R^\prime$ be the reduced digraph with parameters $(\varepsilon,d)$ obtained by applying the Diregularity Lemma to $G$. Then $R^\prime$ has a spanning oriented subgraph $R$ such that  for every vertex $V_i\in R$, 
\begin{equation}\label{EQU-orientreduced}
	d_R^+(V_i) \geqslant \left(\frac{e_G(V_i,V(G))}{|G||V_i|}-(d+3\varepsilon)\right)|R|\mbox{ and } d_R^-(V_i) \geqslant\left(\frac{e_G(V(G),V_i)}{|G||V_i|}- (d+3\varepsilon)\right)|R|.
\end{equation}
In particular, $\delta^0(R) \geqslant  (\delta^0(G)/|G| - (d+3\varepsilon))|R|$.
\end{lemma}

To prove Theorem \ref{THM-arbitraryori}, we will need the Blow-up Lemma of Koml\'os, S\'arkoz\"y and  Szemer\'edi \cite{komlosC17}. It allows that there are certain vertices $y$ to be embedded into $V_i$ whose images are \textit{a priori} restricted to certain sets $S_y \subseteq V_i$ for every $i$. Here, we use $\Delta(H)$ to denote the maximum degree of $H$.

\begin{lemma}[Blow-up Lemma, \cite{komlosC17}]
\label{LEM-blowup}
For every $d, \Delta, c > 0$ and $k\in \mathbb{N}$, there exist constants $\varepsilon_0=\varepsilon_0(d,\Delta,c,k)$  and $\alpha=\alpha(d,\Delta,c,k)$ such that the following holds. Let $n_1,n_2,\dots,n_k$ be positive integers, $0<\varepsilon<\varepsilon_0$, and let $G$ be a $k$-partite graph with vertex classes $V_1,V_2,\dots,V_k$ where $|V_i|=n_i$ for $i\in[k]$. Let $J$ be a graph on vertex set $\{V_1,V_2,\dots,V_k\}$ such that $G[V_i,V_j]$ is $(\varepsilon,d)$-superregular whenever $V_iV_j\in E(J)$. Suppose that $H$ is a $k$-partite graph with vertex classes $W_1,W_2,\dots,W_k$ of size at most $n_1,n_2,\dots,n_k$, respectively, with  $\Delta(H) \leqslant  \Delta$. Suppose further that there exists a graph homomorphism $\phi: V(H) \rightarrow V(J)$ such that $|\phi^{-1}(i)| \leqslant  n_i$ for every $i \in [k]$. Moreover, suppose that in each class $W_i$, there is a set of at most $\alpha n_i$ special vertices $x$, each equipped with a set $S_x \subseteq V_i$ with $|S_x| \geqslant  c n_i$. Then there is an embedding of $H$ into $G$ such that every special vertex $x$ is mapped to a vertex in $S_x$.
\end{lemma}

\subsection{Robust expander}

Robust expanders, introduced by K\"uhn, Osthus and Treglown \cite{kuhnJCTB100}, have formed a key feature in many results involving Hamilton cycles. Below we give the definition of the robust expander and some brief background.
\begin{definition}
Let $\nu\in(0,1)$. For an $n$-vertex digraph $G$ and $S\subseteq V(G)$, the $\nu$-outneighborhood of $S$ in $G$, denoted by $RN_{\nu,G}^+(S)$, is the set of vertices that have at least $\nu n$ inneighbors in $S$. Given $0<\nu\leqslant  \tau <1$, we say that $G$ is a robust $(\nu,\tau)$-outexpander if $$|RN_{\nu,G}^+(S)|\geqslant  |S|+\nu n$$ for every $S$ satisfying $\tau n< |S|< (1-\tau)n$. 
\end{definition}




\begin{lemma}[\cite{kuhnAM237}]\label{LEM-semitoexpander}
Let $0<1/n\ll\nu\ll\tau\leqslant \varepsilon/2\leqslant  1$. Every $n$-vertex oriented graph $G$  with $\delta^0(G)\geqslant  (3/8+\varepsilon)n$  is a robust $(\nu,\tau)$-outexpander.
\end{lemma}

Taylor \cite{taylor2013} proved that every sufficiently large robust outexpander of linear minimum semidegree contains any oriented  Hamilton cycle.

\begin{theorem}[\cite{taylor2013}]\label{THM-expanderanyori}
Let $1/n\ll \nu\leqslant  \tau\ll \gamma <1$ and let $G$ be an $n$-vertex digraph with $\delta^0(G)\geqslant  \gamma n$. If  $G$ is a robust $(\nu,\tau)$-outexpander, then it contains every possible orientation of a Hamilton cycle.
\end{theorem}

The following lemma of DeBiaso and Treglown~\cite{debiasioarXiv2025} allows us to find any not too short oriented path between any pair of vertices in a robust outexpander.

\begin{lemma}[\cite{debiasioarXiv2025}]\label{LEM-Connect}
Let $1/n \ll\varepsilon\ll  \nu \ll \tau \ll \gamma \ll 1$ and let $G$ be an $n$-vertex digraph  with $\delta^0(G)\geqslant \gamma n$. If $G$ is a robust $(\nu,\tau)$-outexpander, then for any distinct $x,y \in V(G)$, and any $k$-vertex oriented path $P$ with $1/\varepsilon\leqslant k \leqslant n$, there is a copy of $P$ in $G$ that starts from $x$ and ends at $y$.
\end{lemma}



\section{The structure of $G$}\label{SEC-structureG}

We now come to the first ingredient in the proof of Theorem \ref{THM-arbitraryori}, which is Theorem \ref{THM-extremal}, saying that every sufficiently large oriented graph $G$ with high minimum semidegree that fails to be a robust outexpander must structurally ``close to''  a well-behaved configuration.  To formally say what we mean by ``close to'' a structure  we need the following definition. 

Let $(W,X,Y,Z)$ be a partition of $V(G)$.  For simplicity, we  refer to $Z$ and $X$ as the ``predecessor'' and ``successor'' of $W$, respectively, and write $W^-=Z, W^+=X, W^{++}=Y$.  Analogous definitions hold for $X,Y$ and $Z$.

\begin{definition}[$\delta$-extremal]\label{DEF-extremal}
Let $G$ be an $n$-vertex oriented graph. We say that a partition  $(W,X,Y,Z)$ of $V(G)$ is $\delta$-extremal if the following statements hold:
\begin{enumerate}[label =\upshape \textbf{(EP\arabic{enumi})}, ref =\upshape (EP\arabic{enumi})]
	\setlength{\itemindent}{1.5em}
	\item $|W|,|X|,|Y|,|Z|=n/4\pm O(\delta n)$;\label{EP1}
	
	\item $e(W\cup Z,Y\cup Z)= O(\delta n^2)$;\label{EP2}

	\item For each $J\in \{W,X,Y,Z\}$ and $\sigma \in \{+,-\}$, all but $O(\delta n)$ vertices $v\in J$  satisfy $d^\sigma(v, J^\sigma)\geqslant |J^\sigma|-O(\delta n)$;\label{EP3}
	
	\item For $J\in\{W,Y\}$, all but $O(\delta n)$  vertices $v\in J$ satisfy $d^\pm(v,J)\geqslant |J|/2-O(\delta n)$;\label{EP4}
	
	\item $d^\pm(v,X)\geqslant |X|/2-O(\delta n)$ for all but  $O(\delta n)$  vertices $v\in Z$; \label{EP5}

	\item  For $J\in\{W,Y\}$ and $\sigma\in\{+,-\}$, every $v\in J$ satisfies $d^\sigma(v,J\cup J^\sigma)\geqslant n/50-O(\delta n)$;\label{EP6}
	
	\item For $J\in\{X,Z\}$ and $\sigma\in\{+,-\}$, every $v\in J$ satisfies  $d^\sigma(v,J^\sigma\cup J^{\sigma\sigma})\geqslant n/50-O(\delta n)$.\label{EP7}
\end{enumerate}
\end{definition}

We refer to a vertex $u\in J$ as a \emph{good} vertex  if it fulfills all degree conditions for $J$ in \ref{EP3}-\ref{EP5}. Otherwise, it is a \emph{bad} vertex of $J$. Whether a vertex is a bad vertex depends on the partition of $V(G)$, however, for the sake of convenience, we sometimes refer to it as a bad vertex of $G$ when no ambiguity arises.

\medskip

The following theorem is the main result in this section. 
\begin{theorem}\label{THM-extremal}
Let $0<1/n\ll\nu \leqslant \tau \ll\varepsilon \ll \delta\ll 1$ and let $G$ be an $n$-vertex oriented graph. If  $\delta^0(G)\geqslant 3n/8-O(\varepsilon n)$, then either $G$ is a robust $(\nu,\tau)$-outexpander or $V(G)$ has a $\delta$-extremal partition $(W,X,Y,Z)$. 
\end{theorem}

Theorem \ref{THM-extremal} establishes that if $G$ has no expansion property,  then $V(G)$ admits a partition $(W,X,Y,Z)$ such that  $WXYZW$ is close to an $n/4$-blowup of the directed 4-cycle, as guaranteed by \ref{EP1} and \ref{EP3}.  Furthermore,   \ref{EP3}-\ref{EP5} imply that $G$ has only $O(\delta n)$ bad vertices and additionally,  the degrees of those bad vertices satisfy an acceptable lower bound by \ref{EP6} and \ref{EP7}. The condition\ref{EP4} ensures that each of $G[W]$ and $G[Y]$ spans an ``almost regular tournament''.  Finally, for every good vertex in $Z$, \ref{EP5} dictates that the number of both its inneighbors and outneighbors within $X$ is approximately equal to $|X|/2$.

\medskip

To prove Theorem \ref{THM-extremal}, it suffices to prove the following three lemmas. 

\begin{lemma}\label{LEM:non-expander-case}
Let $0<1/n\ll\nu \leqslant \tau \ll\varepsilon\ll 1$ and let $G$ be an $n$-vertex oriented graph with  $\delta^0(G)\geqslant 3n/8-O(\varepsilon n)$. If $G$ is not a robust $(\nu,\tau)$-outexpander, then there is a partition $(W,X,Y,Z)$ of $V(G)$ satisfying $|W|,|X|,|Y|,|Z|=n/4\pm O(\varepsilon n)$ and $e(W\cup Z,Y\cup Z)\leqslant \varepsilon^2 n^2$.
\end{lemma}

\begin{lemma}\label{LEM-findbad}
Let $0<1/n\ll\varepsilon\ll \delta \ll 1$ and let $G$ be an $n$-vertex oriented graph with $\delta^0(G)\geqslant 3n/8-O(\varepsilon n)$. If $(W,X,Y,Z)$ is a partition of $V(G)$ satisfying $|W|,|X|,|Y|,|Z|=n/4\pm O(\varepsilon n)$ and $e(W\cup Z,Y\cup Z)\leqslant \varepsilon^2 n^2$, then \ref{EP1}-\ref{EP5}  hold for the partition.
\end{lemma}

\begin{lemma}\label{LEM-distributebad}
Let $0<1/n\ll\varepsilon\ll \delta \ll 1$ and let $G$ be an $n$-vertex oriented graph with $\delta^0(G)\geqslant 3n/8-O(\varepsilon n)$. Any partition of $V(G)$ satisfying \ref{EP1}-\ref{EP5} can be converted to a new partition that meets \ref{EP1}-\ref{EP7} by reassigning at most $O(\delta n)$ vertices.
\end{lemma}

We next end this section by giving the proofs of Lemmas \ref{LEM:non-expander-case}-\ref{LEM-distributebad}.

\begin{proof}[\textbf{Proof of Lemma \ref{LEM:non-expander-case}}]
Since $G$ is not a robust $(\nu,\tau)$-outexpander,  there exists a set $S\subseteq V(G)$ with $\tau n< |S|< (1-\tau)n$ such that $|RN_{\nu,G}^+(S)|<|S|+\nu n$. Set $W=RN_{\nu,G}^+(S)\cap S$, $X=RN_{\nu,G}^+(S)\backslash  S$, $Y=V(G)\backslash  (RN_{\nu,G}^+(S)\cup S)$ and $Z=S\backslash  RN_{\nu,G}^+(S)$. Next we claim that  $(W,X,Y,Z)$ is the desired partition.  

Since every vertex not in $RN_{\nu,G}^+(S)$ has indegree less than $\nu n$ in $S$, we have $e(W\cup Z,Y\cup Z)\leqslant \nu n^2$ and thus the second statement of the lemma holds as $\nu\ll \varepsilon$. Moreover, the inequality $|RN_{\nu,G}^+(S)|<|S|+\nu n$ implies that 
\begin{align}\label{EQ-BD}
	|X|<|Z|+\nu n.
\end{align} 

Now we show that  $W,X,Y$ and $Z$ are almost equal in size by proving the following claims.
\begin{claim}\label{CLM-bcd}
	$|X|+|Y|+|Z| \geqslant  3n/4 - O(\varepsilon n)$ and thus $|W|\leqslant  n/4+O(\varepsilon n)$.
\end{claim}
\begin{proof}
	First observe that $Y\cup Z \neq \emptyset$ since otherwise $n =|RN_{\nu, G}^+(S)|<|S|+ \nu n < (1-\tau)n + \nu n$, contradicting the fact $\nu\leqslant \tau$.  Let $u$ be a vertex in $Y\cup Z$ with $d^-(u,Y)\leqslant  |Y|/2$. More precisely, let $u$ be any vertex of $Z$ if $Y=\emptyset$ and otherwise let $u$ be a vertex of $Y$ with minimum indegree in $G[Y]$. It follows by $u\notin RN_{\nu,G}^+(S)$ that $d^-(u,S)=d^-(u,W\cup Z)<\nu n$. Then the degree condition  and  (\ref{EQ-BD}) imply that  
	$$3n/4-O(\varepsilon n) \leqslant  2 d^-(v) < |Y|+2|X| + 2\nu n< |X|+|Y|+|Z|+3\nu n,$$ 
	which yields the desired conclusion as  $\nu \ll \varepsilon$.
\end{proof}

\begin{claim}\label{CLM-abc}
	$|W| + |X| + |Y| \geqslant  3n/4 - O(\varepsilon n)$ and thus $|Z|\leqslant  n/4+O(\varepsilon n)$.
\end{claim}

\begin{proof}
	We may assume  $Z\neq\emptyset$ since otherwise there is nothing to prove. It follows by $Z=S\backslash RN_{\nu,G}^+(S)$ that  $d^-(v,Z)\leqslant  d^-(v,S)<\nu n$ for each $v\in Z$. Then $e(Z)<|Z|\nu n$ and there exists $u\in Z$ such that $d(u,Z)< 2\nu n$. Since $G$ is oriented and $\delta^0(G)\geqslant 3n/8-O(\varepsilon n)$, we have  $3n/4-O(\varepsilon n)\leqslant d(u)< 2 \nu n + |W| + |X| + |Y|$ and this proves the claim.
\end{proof}

\begin{claim}\label{CLM-abd}
	$|W| + |X| + |Z| \geqslant  3n/4 - O(\varepsilon n)$ and thus $|Y|\leqslant  n/4+O(\varepsilon n)$.
\end{claim}

\begin{proof}
	Recall that $d^-(v,S)<\nu n$ for each $v\notin RN_{\nu, G}^+(S)$. Calculating the outdegrees of  vertices in $S$, we have 
	$$|S|(3n/8-O(\varepsilon n)) < |S||RN_{\nu, G}^+(S)| + \nu n |V(G)\backslash RN_{\nu, G}^+(S)|<|S|(|S|+\nu n) + \nu n|S|/\tau,$$  the last inequality holds by  $|RN_{\nu, G}^+(S)|<|S|+\nu n$ and $|V(G)\backslash RN_{\nu, G}^+(S)|<  n <|S|/\tau$. Then by Claim \ref{CLM-abc}, we have $|W|=|S|-|Z| > 3n/8-O(\varepsilon n) - \nu n (1 + 1/ \tau)-|Z|\geqslant n/9$. This implies that   $e(W,Y\cup Z)\leqslant \nu n^2\leqslant 9\nu n |W|$. Computing all outdegrees of vertices in $W$, we have  $|W|(3n/8-O(\varepsilon n))   \leqslant|W|^2/2 + 9\nu n |W| + |W||X|$. Hence there exists $u\in W$ with $3n/4-O(\varepsilon n)\leqslant  2d^+(u) \leqslant  |W| + 18\nu n + 2|X|$, which implies the claim by (\ref{EQ-BD}).
\end{proof}

It follows by (\ref{EQ-BD})  and Claim \ref{CLM-abc} that $|X|\leqslant  n/4+ O(\varepsilon n)$. Moreover, we have $|X|\geqslant n/4- O(\varepsilon n)$ by Claims  \ref{CLM-bcd}-\ref{CLM-abd} and $n=|W|+|X|+|Y|+|Z|$. In the same way, we have $|W|,|Y|,|Z|= n/4 \pm O(\varepsilon n)$, which completes the proof.
\end{proof}

Before proving Lemma \ref{LEM-findbad}, we need the following result. 

\begin{lemma}\label{LEM:edges-to-degree}
Let $\xi $ be a real with $\xi  \ll 1$. Suppose that  $G$ is a digraph and  $A,B$ are two disjoint subsets of $V(G)$. 

\emph{(i)} If $e(A,B)\geqslant|A||B|-O(\xi n^2)$, then there are at most $\xi ^{1/3} n$ vertices $a$ in $A$ with $d^+(a,B)\leqslant |B|-\xi^{1/2} n$  and at most $\xi ^{1/3} n$  vertices $b$ in  $B$ with $d^-(b,A)\leqslant |A|-\xi^{1/2}  n$.

\emph{(ii)} If $e(A,B)\leqslant O(\xi n^2)$, then there are at most $\xi ^{1/3} n$ vertices $a$ in $A$ with $d^+(a,B)\geqslant \xi^{1/2}  n$ and at most $\xi ^{1/3} n$  vertices $b$ in  $B$ with $d^-(b,A)\geqslant\xi^{1/2} n$.
\end{lemma}
\begin{proof}
Observe that if $e(A,B)\leqslant O(\xi n^2)$, then in the complement digraph $\overline{G}$  of $G$ we have $e_{\overline{G}}(A,B)\geqslant |A||B|-O(\xi n^2)$. Thus (ii) follows directly from the first statement. Moreover, by symmetry, it suffices to prove the former statement of (i). Suppose to the contrary that there are at least  $\xi ^{1/3} n$ vertices $a \in A$ with $d^+(a, B) \leqslant |B| - \xi^{1/2}  n$. Since each such $a$ has at least $\xi^{1/2}  n$ missing edges to $B$,  the total number of missing edges from $A$ to $B$ is at least $\xi ^{5/6}n^2$. This contradicts the assumption that $e(A, B) \geqslant |A||B| - O(\xi  n^2)$, and then the claim holds.
\end{proof}

\begin{proof}[\textbf{Proof of Lemma \ref{LEM-findbad}}]
Recall that  $(W,X,Y,Z)$ is a partition of $V(G)$ with $|W|,|X|,|Y|,|Z|=n/4\pm O(\varepsilon n)$ and $e(W\cup Z,Y\cup Z)\leqslant \varepsilon^2 n^2$.
Applying Lemma \ref{LEM:edges-to-degree} (ii) with $A=W\cup Z$ and $B=Y\cup Z$, there are three sets $W^L\subseteq W,Y^L\subseteq Y,Z^L\subseteq Z$ satisfying

\begin{enumerate}[label =\upshape \textbf{(\roman{enumi})}, ref=\upshape {(\roman{enumi})}]
	\setlength{\itemindent}{1.5em}
	\item $|W^L|,|Y^L|,|Z^L|\leqslant 2 \varepsilon^{2/3} n$; \label{L1}
	\item $d^+(u,Y\cup Z) \leqslant \varepsilon n$ for each $u\in (W\backslash W^L)\cup (Z\backslash Z^L)$; \label{L2}
	\item $d^-(v,W\cup Z) \leqslant  \varepsilon  n$ for each $v\in (Y\backslash Y^L)\cup (Z\backslash Z^L)$. \label{L3}
\end{enumerate}
Next we finish the proof by proving the following two claims and some additional arguments.

\begin{claim}\label{CLM-dgood}
	All vertices of $Z\backslash Z^L$ are good.
\end{claim}
\begin{proof}
	Let $z$ be any vertex in $Z\backslash Z^L$. It follows by \ref{L2}-\ref{L3} that  $d(z) \leqslant d^+(z,W)+|X|+d^-(z,Y)+2\varepsilon n$. Since  $d(z)\geqslant 2\delta^0(G)$  and  $|W|,|X|,|Y|=n/4\pm  O(\varepsilon n)$, we have $d^+(z,W)$, $d^-(z,Y)\geqslant n/4 -O(\varepsilon n)$. Moreover, the degree condition and \ref{L3} imply that $3n/8-O(\varepsilon n)\leqslant d^-(z)\leqslant  \varepsilon n + |Y|+ d^-(z,X)$ and thus $d^-(z,X)\geqslant n/8-O(\varepsilon n)$. By symmetry,  we have $d^+(z,X) \geqslant n/8 -O(\varepsilon n)$. Thus $z$ is a good vertex since $|W|,|X|,|Y|,|Z|=n/4\pm O(\varepsilon n)$  and $\varepsilon \ll\delta$, which   completes the proof.
\end{proof}

\begin{claim}\label{CLM-agood}
	There are at most $\delta n$ bad vertices in $W$ and $e(W,X)\geqslant |W||X|-O(\varepsilon^{1/4}n^2)$.
\end{claim}
\begin{proof}
	Let $w$ be any vertex in $W\backslash W^L$. By \ref{L2} and the degree condition, we have $3n/8 -O(\varepsilon n) \leqslant d^+(w)\leqslant  d^+(w,W) + |X|+\varepsilon n$. Then  $|X| =n/4\pm O(\varepsilon n)$ implies that $d^+(w,W)\geqslant n/8 -O(\varepsilon n)$ and thus $e(W)\geqslant n^2/32-O(\varepsilon^{2/3} n^2)$ by \ref{L1}. Moreover, we have $d^-(w,W)\leqslant n/8 +O(\varepsilon n)$ as $G$ is oriented.  Therefore, there are at most $\varepsilon^{1/4} n$ vertices $u\in W$ with $d^-(u,W)<n/8 - \varepsilon^{1/3} n$  since otherwise
	\begin{align*}
		e(W)&\leqslant \varepsilon^{1/4}n(n/8-\varepsilon^{1/3} n)+ |W^L||W|+(|W\backslash W^L|-\varepsilon^{1/4}n)(n/8+O(\varepsilon n))\\
		&\leqslant n^2/32-\varepsilon^{7/12}n^2+2\varepsilon^{2/3}n^2+O(\varepsilon n^2)\\
		&< n^2/32-O(\varepsilon^{2/3} n^2),
	\end{align*}
	a contradiction. Let $W^{\prime}$ be the set of vertex in $W$ with $d^+(u,W)<n/8-O(\varepsilon n)$ or $d^-(u,W)<n/8-\varepsilon^{1/3} n$. By the above arguments, we have $|W^{\prime}|\leqslant |W^L|+\varepsilon^{1/4}n\leqslant 2\varepsilon^{1/4}n$.  Moreover,  for any $v \in W\backslash (W^{\prime}\cup W^L)$,  \ref{L2} gives that  $d^+(v) \leqslant \varepsilon n + d^+(v,X)+(|W|-d^-(v,W))$ and thus $d^+(v,X)\geqslant n/4 -2\varepsilon^{1/3}n$ by the degree condition. Therefore, by $|W|,|X|= n/4\pm O(\varepsilon n)$ and $|W^{\prime}|\leqslant 2\varepsilon^{1/4}n$, we have $$e(W,X)\geqslant n^2/16-O(\varepsilon^{1/4}n^2)\geqslant |W||X|-O(\varepsilon^{1/4}n^2).$$ Furthermore,  for each $z\in Z\backslash Z^L$ we have $d^+(z,W)\geqslant n/4-O(\varepsilon n)$ by the proof of Claim \ref{CLM-dgood}. Then \ref{L1}  implies that $e(Z,W) \geqslant |Z||W| -O(\varepsilon^{2/3} n^2)$. By Lemma \ref{LEM:edges-to-degree} (i), there are at most $\varepsilon^{2/9} n$ vertices of $W$ have indegree less than $n/4-O(\varepsilon^{1/3}n)$ in $Z$. Then $W$ has at most $|W^{\prime}\cup W^L|+\varepsilon^{2/9} n<\delta n$ bad vertices by the fact  $|W|,|X|,|Y|,|Z|=n/4\pm O(\varepsilon n)$.
\end{proof}
Swapping the roles of $W$ and $Y$, $X$ and $Z$, and  interchanging ``$+$'' and ``$-$'' in the proof of Claim \ref{CLM-agood}, $Y$ has at most $\delta n$ bad vertices and $e(X,Y)\geqslant |X||Y|-O(\varepsilon^{1/4}n^2)$. Applying Lemma \ref{LEM:edges-to-degree} (i) with $e(W,X)$ and $e(X,Y)$ respectively, there are at most $\delta n$ vertices $x\in X$ with $d^-(x,W)<|W|-\delta n$ or $d^+(x,Y)<|Y|-\delta n$, which implies that $X$ has at most $\delta n$ bad vertices. This completes the proof.
\end{proof}

\begin{proof}[\textbf{Proof of Lemma \ref{LEM-distributebad}}]

Let $(W,X,Y,Z)$ be a partition of $V(G)$ satisfying \ref{EP1}-\ref{EP5} and let  $B$ be the set of bad vertices in the partition. First observe that  
$|B|= O(\delta n)$ due to \ref{EP3}-\ref{EP5}, thus any partition that arises from  $(W,X,Y,Z)$ by reassigning bad vertices will still satisfy \ref{EP1}-\ref{EP5}.  It therefore suffices to establish that the new partition satisfies \ref{EP6}-\ref{EP7} after the reassignment of every bad vertex. 

Let $v$ be any vertex of $B$. By the degree condition and the fact that $V(G)=W\cup X\cup Y\cup Z$, we get that either $d^+(v,W\cup X)\geqslant n/50-O(\delta n)$ or $d^+(v,Y\cup Z)\geqslant n/50-O(\delta n)$ and, either $d^-(v,W\cup Z)\geqslant n/50-O(\delta n)$ or $d^-(v,X\cup Y)\geqslant n/50-O(\delta n)$. Reassign the vertex  $v$ according to the following rules:

\begin{itemize}
	\item Assign $v$ to $W$ if $d^+(v,W\cup X), d^-(v,W\cup Z)\geqslant n/50-O(\delta n)$;
	
	\item  Assign $v$ to $X$ if $d^+(v,Y\cup Z), d^-(v,W\cup Z)\geqslant n/50-O(\delta n)$;
	
	\item Assign $v$ to $Y$ if $d^+(v,Y\cup Z), d^-(v,X\cup Y)\geqslant n/50-O(\delta n)$;
	
	\item Assign $v$ to $Z$ if $d^+(v,W\cup X), d^-(v,X\cup Y)\geqslant n/50-O(\delta n)$.
\end{itemize}

It is not difficult to check that the partition after the reassignment satisfies \ref{EP6}-\ref{EP7}, which completes the proof.
\end{proof}

\section{$C$ has few sinks}\label{SEC-fewsink}

Keevash, K\"uhn and Osthus \cite{keevashJLMS79} showed that every sufficiently large oriented graph $G$ on $n$ vertices whose minimum semidegree is at least $(3n-4)/8$ contains a directed Hamilton cycle. In this section, we extend this result to all oriented Hamilton cycles with few sinks.

\begin{theorem}\label{THM-fewsink}
There exists a constant $\xi\ll 1$ such that if $G$ is a sufficiently large $n$-vertex oriented graph with $\delta^0(G)\geqslant (3n-4)/8$, then $G$ contains every $n$-vertex oriented cycle $C$ with $\sigma(C)\leqslant \xi n$.  In particular, $G$ has a directed Hamilton cycle. 
\end{theorem}

Indeed, in the proof we need $\xi \ll \varepsilon,\alpha$, where $\varepsilon$ and $\alpha$ are given in the Blow-up Lemma. Before proving this theorem, we show two key lemmas to the proof. A collection of disjoint paths $\mathcal{P}=\{P_1,P_2\ldots, P_p\}$ is said to be a \emph{$\delta$-balanced path system}  for a partition $(W,X,Y,Z)$ if  
\begin{enumerate}[label =\upshape \textbf{(PS\arabic{enumi})}, ref=\upshape (PS\arabic{enumi})]

\setlength{\itemindent}{1.5em}
\item  $p=O(\delta n)$ and  $|X\backslash V(\mathcal{P})| = |Z\backslash V(\mathcal{P})|$; \label{PS1} 

\item  Each $P_i$ is a directed 13-path with its endvertices being good vertices of  $W$. \label{PS2}
\end{enumerate}

\begin{lemma}\label{LEM-XZbalbance}
Let $0<1/n\ll\nu\ll \tau\ll \delta\ll 1$ and let $G$ be an $n$-vertex oriented graph with   $\delta^0(G)\geqslant (3n-4)/8$. If $G$ is not a robust $(\nu,\tau)$-outexpander, then $V(G)$ has $\delta$-extremal partition that admits  a $\delta$-balanced path system.
\end{lemma}

\begin{lemma}\label{LEM-balbancecycle}
There exists a constant $\xi\ll1$ such that the following holds. Let $0<1/n\ll \delta\ll \xi\ll1$ and let $G$ be an $n$-vertex oriented graph.  If $V(G)$ has a $\delta$-extremal partition with a $\delta$-balanced path system, then $G$ contains any $n$-vertex oriented cycle $C$ with $\sigma(C)\leqslant \xi n$.
\end{lemma}

Equipped with Lemmas \ref{LEM-XZbalbance} and \ref{LEM-balbancecycle}, we are ready to prove Theorem \ref{THM-fewsink}.

\begin{proof}[\textbf{Proof of Theorem \ref{THM-fewsink}}]
Let $\xi$ be the constant given in Lemma \ref{LEM-balbancecycle} and let $0<1/n \ll \nu \ll \tau \ll \varepsilon \ll \delta \ll \xi \ll 1$.  Suppose $C$ is any $n$-vertex oriented cycle with $\sigma(C)\leqslant \xi n$. By Theorem \ref{THM-expanderanyori}, it suffices  to consider the case that $G$ has no expansion property. It follows by   Lemma \ref{LEM-XZbalbance} that $V(G)$ has a $\delta$-extremal partition   with a $\delta$-balanced path system. Then we are done  by Lemma \ref{LEM-balbancecycle}.
\end{proof}

The proofs of Lemmas \ref{LEM-XZbalbance} and \ref{LEM-balbancecycle}  are provided in the subsequent two subsections, respectively.

\subsection{Proof of Lemma \ref{LEM-XZbalbance}}

Clearly, Lemma \ref{LEM-XZbalbance} follows immediately by the following two results.

\begin{lemma}\label{LEM-matchlarge}
Let $0<1/n \ll \delta \ll 1$ and let $G$ be an $n$-veretx oriented graph. Suppose $(W,X,Y,Z)$ is a $\delta$-extremal partition  of $V(G)$ with $|X|\geqslant |Z|$. If $E(X\cup Y,W \cup X)$ contains a matching of size $|X|-|Z|$, then $G$ contains a $\delta$-balanced path system for $(W,X,Y,Z)$. 
\end{lemma}

\begin{lemma}\label{LEM-assignmatch}
Let $0<1/n\ll \nu\ll\tau\ll \delta \ll 1$ and let $G$ be an $n$-vertex oriented graph with $\delta^0(G)\geqslant (3n-4)/8$. If $G$ is not a robust $(\nu,\tau)$-outexpander, then $V(G)$ has a $\delta$-extremal partition $(W,X,Y,Z)$ with $|X|\geqslant |Z|$ such that $E(X\cup Y,W\cup X)$ contains a matching of size $|X|-|Z|$.
\end{lemma}

First we give a proof of Lemma \ref{LEM-matchlarge}.

\begin{proof}[\textbf{Proof of Lemma \ref{LEM-matchlarge}}]
Set $l=|X|-|Z|$. Clearly, $l\geqslant 0$ and $l=O(\delta n)$ by \ref{EP1}. If $l=0$, then $\mathcal{P}=\emptyset$ is a trivial $\delta$-balanced path system. So we may assume $l>0$ in the following. By the condition, there is a matching $M$ of size $l$ in $E(X\cup Y,W \cup X)$. It then suffices to show  that for each edge of $M$, there is a directed 13-path $P_i$ that contains the edge and both endvertices of $P_i$ as good vertices of $W$. Moreover, the 13-paths are pairwise disjoint and each $P_i$ reduces the difference between $|X|$ and $|Z|$ by exactly one, that is, $|V(P_i)\cap X|=|V(P_i)\cap Z|+1$.

Next we construct $P_1,P_2,\ldots, P_l$ sequentially.  Let $ U $ be the set of bad vertices in  $G$ and all used vertices in 13-paths. Then $| U |\leqslant O(\delta n)+13l=O(\delta n)$ due to \ref{EP3}-\ref{EP5}. Let $uv$ be any edge of $M$. Since $M$ is a matching in $E(X\cup Y,W\cup X)$, it follows that $u\in X\cup Y$ and $v\in W\cup X$.  Consider first the case that $v\in W$. By \ref{EP6} and  $| U |=O(\delta n)$,  one can select a vertex  $v_1$ in $N^+(v,W\cup X)\backslash  U $. Note that $v_1$ is a good vertex as $v_1 \notin  U $. Then by \ref{EP3}, if $v_1\in W$, we can further pick an unused vertex $v_2$ in $N^+(v_1,X)\backslash U $ and, if $v_1 \in X$, we can instead choose $v_2$ in $N^+(v_1,Y)\backslash U $.  We iterate this process to construct a directed path $P_v$ starting at $v$,  which takes one of two forms: $W^2XYZW$ or $WXYZW^2$. Similarly, for the case $v\in X$, there is a path $P_v$ with form $XYZW^3$ or $XZW^4$.

By symmetry, by considering inneighbors of vertices, there is a directed path $P_u$ ending at $u$ with form $W^2XYZWX$ or $W^3XYZX$ if  $u\in X$ and, with form $W^4XY^2$ or $WXYZWXY$ if $u\in Y$. Set $P_i=P_uP_v$.  Clearly, $P_i$ is a directed 13-path with endvertices in $W$ and $|V(P_i)\cap X|=|V(P_i)\cap Z|+1$.  Moreover, all vertices in $V(P_i)\backslash \{u,v\}$ are good as they do not belong to $ U $. In particular, the endvertices of $P_i$ are good vertices of $W$. By the construction of $P_i$, we always select vertices while avoiding those used previously, thus, each 
$P_i$ is disjoint from all 13-paths constructed earlier. Therefore, $\mathcal{P}=\{P_1,P_2,\ldots,P_l\}$ is a path system satifying  \ref{PS1} and \ref{PS2}.
\end{proof}


For a partition $(W,X,Y,Z)$ of $V(G)$, we define $L_{WX}$ and $L_{XY}$ to be two sets of vertices with ``large'' degrees. More precisely, we have 
\begin{equation*}
\begin{split}
	L_{WX}= \{x \in W\cup X : d^-(x,X\cup Y) \geqslant n/50-O(\delta n)\},\\
	L_{XY} = \{x \in X\cup Y : d^+(x,W\cup X)\geqslant n/50-O(\delta n)\}.
\end{split}
\end{equation*}  
Let $K$ be a subset of $L_{WX}\cup L_{XY}$ with $|K|=\min\{|L_{WX}\cup L_{XY}|,n/200\}$. 
Let $M$ be a maximum matching in $E(X\cup Y,W\cup X)$ and let $e(M)$ be the size of $M$. Then 
\begin{equation}\label{EQ:eM-L}
e(M)\geqslant \min\{|L_{WX}\cup L_{XY}|,n/200\}.
\end{equation}
Indeed, since every $x\in K$ has both outdegree and indegree at most $|K|\leqslant n/200$ in $G[K]$, one may greedily find a matching of $E(X\cup Y,W\cup X)$ with size $|K|$ such that every edge in the matching covers exactly one vertex in $K$ the definition of $L_{WX}$ and $L_{XY}$. 


\medskip

Now we give a proof of Lemma \ref{LEM-assignmatch}. 

\begin{proof}[\textbf{Proof of Lemma \ref{LEM-assignmatch}}]

By Theorem \ref{THM-extremal}, $V(G)$ has a $\delta$-extremal partition $(W^\prime,X^\prime,Y^\prime,Z^\prime)$. Note that $W^\prime$ and $Y^\prime$ are symmetric by  Definition \ref{DEF-extremal}, moreover, $X^\prime$ and $Z^\prime$ are symmetric if \ref{EP5} is disregarded. In this proof since we do not invoke \ref{EP5}, we may w.l.o.g assume that $|X^\prime|\geqslant |Z^\prime|$ by reversing all edges of $G$ and swapping the labels of $W^\prime$ and $Y^\prime$ as well as those of $X^\prime$ and $Z^\prime$ if necessary. Then $|X^\prime|-|Z^\prime|=O(\delta n)$ by \ref{EP1}.

Let $M^\prime$ be a maximum matching in $E(X^\prime\cup Y^\prime,W^\prime\cup X^\prime)$. If $e(M^{\prime})\geqslant |X^\prime|-|Z^\prime|$, then there is nothing to prove as $(W^\prime,X^\prime,Y^\prime,Z^\prime)$ is the desired partition.  Thus we may assume that $e(M^{\prime})<|X^\prime|-|Z^\prime|$. On the other hand, as $|X^\prime|-|Z^\prime|=O(\delta n)\ll n/200$,  we have $e(M^\prime)\geqslant \min\{|L_{W^\prime X ^\prime}\cup L_{X^\prime Y^\prime}|,|X^\prime|-|Z^\prime|\}$ by (\ref{EQ:eM-L}). Therefore, $$|L_{W^\prime X^\prime}\cup L_{X^\prime Y^\prime}|<|X^\prime|-|Z^\prime|=O(\delta n).$$

\begin{claim}\label{CLM-decrease1}
	After reassigning vertices in $L_{W^\prime X^\prime}\cup L_{X^\prime Y^\prime}$, there is a new $\delta$-extremal partition $(W,X,Y,Z)$ such that either $|X|=|Z|$ or, $|X|> |Z|$ and  $L_{WX}\cup L_{XY}=\emptyset$.
\end{claim}

\begin{proof}
	Let $v$ be any vertex in $L_{W^{\prime}X^{\prime}}$. It follows by the definition that $v\in W^{\prime}\cup X^{\prime}$ and $d^-(v,X^{\prime}\cup Y^{\prime})\geqslant n/50-O(\delta n)$. First we assume that $d^+(v,W^{\prime}\cup X^{\prime})<n/50-O(\delta n)$. By the degree condition, we have $d^+(v,Y^{\prime}\cup Z^{\prime})\geqslant n/50 -O(\delta n)$ and then we reassign $v$ into $Y^{\prime}$. The resulting partition $(W,X,Y,Z)$ of $V(G)$ clearly meets \ref{EP1}-\ref{EP7}. Since $v\in Y$ and $d^+(v,W\cup X)=d^+(v,W^{\prime}\cup X^{\prime})<n/50-O(\delta n)$, the vertex $v$ does not belong to $L_{W X}\cup L_{X Y}$. Moreover,  if $v\in W^\prime$, then $|X^\prime|-|Z^\prime|=|X|-|Z|$ and,  $|X^\prime|-|Z^\prime|=|X|-|Z|+1$ when $v\in X^\prime$. This means that in this case  reassigning  $v$ decreases $|L_{W^\prime X^\prime}\cup L_{X^\prime Y^\prime}|$ by exactly one and the difference between $X^\prime$ and $Z^\prime$ is reduced by at most one.
	
	Now we consider the case that  $d^+(v,W^\prime\cup X^\prime)\geqslant n/50-O(\delta n)$.  If $v\in W^\prime$ or $v\in X^\prime$ and $|X^\prime|-|Z^\prime|\geqslant 2$, then we remove $v$ to the set $Z^{\prime}$. For the case  $v\in X^\prime$ and $|X^\prime|-|Z^\prime|=1$, we have $d^-(v,W^\prime\cup Z^\prime)\geqslant n/50-O(\delta n)$ by \ref{EP7} and then we remove $v$ to $W^{\prime}$. In both cases, let  $(W,X,Y,Z)$ be the resulting partition of $V(G)$.  Clearly, $(W,X,Y,Z)$ is $\delta$-extremal for $V(G)$. Moreover,  in the former case the size of $L_{W^\prime X^\prime}\cup L_{X^\prime Y^\prime}$ decreases by exactly one and, in the latter case the difference between $X^\prime$ and $Z^\prime$ reduces to zero.
	
	The similar operation described above can be applied to the vertices in $L_{X^\prime Y^\prime}$. Recall that  $|L_{W^\prime X^\prime}\cup L_{X^\prime Y^\prime}|<|X^\prime|-|Z^\prime|$. Then we may get the desired partition by reassigning vertices in $L_{W^\prime X^\prime}\cup L_{X^\prime Y^\prime}$. 
\end{proof}

Next we claim that there is a matching of size $|X|-|Z|$ in $E(X \cup Y, W \cup X )$. If $|X|=|Z|$, then there is nothing to prove. Thus we may assume that  $|X|>|Z|$ and $L_{W X }\cup L_{X Y }=\emptyset$. Let $M $ be a maximum  matching in $E(X \cup Y, W \cup X )$. Observe that the edges of $M $ can be  partitioned into four types:  from $X$ to $W$, from $Y$ to $W$, from $Y$ to $X$ and the edges in $G[X]$. For simplicity, we say that the matching from $S$ to $T$ matches $S_T\subseteq S$ to $T_S\subseteq T$ for $S,T\in\{W,X,Y\}$ and, the vertex set of $G[X]$ induced by $M$ is $X_X$. Clearly, $|S_T|=|T_S|$ for $S,T\in\{W,X,Y\}$ and $e(M ) = |W_X| + |W_Y| + |X_X|/2 + |X_Y|$. 

Since $M $ is a maximum matching, every edge in $E(Y ,W )$ must be incident to a vertex in  $W_Y\cup W_X\cup Y_W\cup Y_X$. Moreover, as $L_{W  X }\cup L_{X  Y } =\emptyset$, every vertex in $Y_W\cup Y_X$ has at most $n/50$ outneighbors in $W $ and every vertex in $W_Y\cup W_X$ has at most $n/50$ inneighbors in $Y $. Therefore, 
\begin{align*}
	e(Y ,W ) &\leqslant e(Y_W\cup Y_X, W ) + e(Y ,W_X\cup W_Y) \\       &\leqslant (|W_Y\cup X_Y|+|W_X \cup W_Y|)n/50 \\  
	&\leqslant e(M )|W |/5.
\end{align*}
The last inequality holds from $|W |>n/5$ (due to \ref{EP1}) and $|W_Y\cup X_Y|+|W_X \cup W_Y|\leqslant 2e(M )$. Similarly, we have $e(X ,W ) \leqslant e(M )|W |/5$. Considering the sum of indegrees of vertices in $W $, we have $\sum_{w\in W }d^-(w)\leqslant |W |(|W |-1)/2+|W ||Z |+2e(M )|W |/5$. Then there exists a vertex $w \in W $ with $d^-(w) \leqslant (|W |-1)/2 + |Z | + 2e(M )/5$. In the same way, there is a vertex $y\in Y $ with $d^+(y) \leqslant (|Y |-1) / 2 + |Z | +  2e(M)/5$. 

It follows by $L(W X)\cup L(X Y) =\emptyset$ that $\delta^0(G[X])\leqslant n/50$. Note that every edge of $G[X ]$ must be incident to a vertex in $X_W\cup X_X\cup X_Y$. Then
$e(X ) \leqslant |X_W\cup X_X\cup X_Y| n/50 \leqslant e(M ) |X |/5.$
Thus, there exists a vertex $x\in X $ such that $d(x) \leqslant |W | + |Y | + |Z | + e(M )/5$. Therefore, $$d^-(w) +d^+(y) + d(x) \leqslant \frac{3}{2}(|W |+|Y |+2|Z |)-1 + e(M ).$$ By the degree condition and  $n = |W | + |X | + |Y | + |Z |$, we have  $(3n-4)/2 \leqslant d^-(w) +d^+(y) + d(x)\leqslant  3(n+|Z |-|X |)/2-1 + e(M )$. Thus $|X | - |Z | \leqslant  2e(M )/3 + 2/3$. If  $e(M )<|X | - |Z |$, then  $e(M)+1 \leqslant |X |-|Z | \leqslant 2e(M)/3 + 2/3$, which is impossible. Therefore, $e(M )\geqslant |X | - |Z |$ and then $(W ,X ,Y ,Z )$ is the desired partition.
\end{proof}

\subsection{Proof of Lemma \ref{LEM-balbancecycle}}

Before proving Lemma \ref{LEM-balbancecycle}, we first present a path partition of the cycle $C$, which shows that there is a small number of short subpaths of $C$ with nice properties such that deleting all   \emph{internal} vertices of these  paths produces a family of directed paths whose lengths are modulo 4. 

\begin{lemma}\label{LEM-cyclepartationfew}
Let $0 < 1/n \ll \xi \ll 1$ and let $C$ be any $n$-vertex  oriented cycle with $\sigma(C)\leqslant \xi n$.   There is a path partition of $C=(L_1R_1L_2R_2\cdots L_lR_l)$ satisfying the following:
\begin{enumerate}[label = \upshape \textbf{(\roman*)}, ref = \upshape (\roman*)]
	\setlength{\itemindent}{1.5em}
	\item  $\sqrt{\xi} n\leqslant l\leqslant 2\sqrt{\xi} n$ and $|L_i|,|R_i|\geqslant 3$ for each $i\in[l]$;\label{i}
	
	\item $\sum_{i=1}^l |L_i|\leqslant 13\sqrt{\xi} n+20\xi n$ and there are exactly $\sqrt{\xi} n$ directed 13-paths $L_i$; \label{ii}
	
	\item Each  $R_i$ is a directed path of order 3 modulo 4, and all of its vertices are  normal vertices of $C$. \label{iii}
\end{enumerate}
\end{lemma}
By (iii), deleting all internal vertices of $L_i$ and $L_{i+1}$ produces a directed path that contains $R_i$ and whose length is modulo 4. 

\medskip

Equipped with Lemma \ref{LEM-cyclepartationfew}, we now give a proof of  Lemma \ref{LEM-balbancecycle}.

\begin{proof}[\textbf{Proof of Lemma \ref{LEM-balbancecycle}}]

Set $d=1/2,\Delta=2,c=1/2,k=4$, there exist $\varepsilon_0, \alpha$ such that Lemma \ref{LEM-blowup} holds for all $\varepsilon < \varepsilon_0$  and $\alpha$. Let $0<1/n\ll \delta \ll \xi \ll \alpha,\varepsilon\ll 1$. 

Let $(W,X,Y,Z)$ be a $\delta$-extremal partition of $V(G)$ and let $\mathcal{P}$ be its $\delta$-balanced path system. By \ref{PS1}-\ref{PS2}, we have $|V(\mathcal{P})|=13|\mathcal{P}|=O(\delta n)$. We use $B$ to denote the set of bad vertices in $V(G)\backslash V(\mathcal{P})$. Clearly, $|B|=O(\delta n)$ by \ref{EP3}-\ref{EP5}.  Let $C$ be any $n$-vertex oriented cycle with $\sigma(C)\leqslant \xi n$ and let $(L_1R_1L_2R_2\cdots L_lR_l)$ be its path partition given by Lemma \ref{LEM-cyclepartationfew}. 

Note that there are exactly $\sqrt{\xi} n$ directed 13-paths among $\{L_1,L_2,\ldots, L_l\}$. We first embed $|\mathcal{P}|=O(\delta n)$ of them onto the paths in $\mathcal{P}$. It follows by \ref{PS1} that $$|X\backslash V(\mathcal{P})|=|Z\backslash V(\mathcal{P})|.$$ Set $r=\sqrt{\xi} n-|\mathcal{P}|$. 
Assume w.l.o.g that $L_1,L_2,\ldots,L_r$ are the remaining directed 13-paths.  For those paths $L_i$ that are not directed 13-paths, we greedily embed them into $G[W\backslash (V(\mathcal{P})\cup B)]$. This is possible since the sum of orders of those paths is at most $20\xi n$ (by Lemma \ref{LEM-cyclepartationfew} (ii)) and $G[W\backslash (V(\mathcal{P})\cup B)]$ still has a large semidegree by \ref{EP4} and the fact that $|V(\mathcal{P})\cup B|=O(\delta n)$.

Let $U$ be the set of used vertices in $V(G)$. In the following, the set $U$  will be updated each time we embed a vertex of $C$ onto a vertex in $V(G)$. At this stage, we have $|U|\leqslant |V(\mathcal{P})|+20\xi n=O(\xi n)$. To embed $C$,  it suffices to further embed $L_1,L_2,\ldots, L_r$ and $R_1,R_2,\ldots,R_l$. Indeed, we will proceed in the following three steps. In Step 1, some of  paths $L_i$ with $i\in[r]$ will be used to cover the vertices in $B$ and the embedding will not disrupt the balance between $X$ and $Z$. In other words, the sizes of $X$ and $Z$ are equal after deleting all internal vertices of the embedded paths. The remaining paths in $\{L_1,L_2,\ldots, L_r\}$ will be used to further balance the sizes of $W,X,Y$ and $Z$ in Step 2. Finally, the paths $R_i$ with $i\in[l]$ will be embedded into $G$ by using the Blow-up Lemma in Step 3.

\medskip
\noindent\textbf{Step 1: Embed $|B|=O(\delta n)$ paths from  $\{L_1,L_2,\ldots, L_r\}$ to cover vertices in $B$.}
\medskip

Let $v$ be any vertex of $B$. We first consider the case that $v\in B\cap W$. Recall that $|B|=O(\delta n)$ and $|U|=O(\xi n)$. Then by \ref{EP6}, we can choose a vertex $v_1$ in $N^+(v,W\cup X)\backslash  (B\cup U)$. Note that $v_1$ is good as $v_1\notin B$. In the same way, if $v_1\in W$,   we can choose a vertex $v_2$ in $N^+(v_1,X) \backslash  (B\cup U)$ by \ref{EP3} and, choose $v_2\in N^+(v_1,Y) \backslash  (B\cup U)$ if $v_1\in X$.  We continue in this way until we get a directed path $P_v$ with form $W^2XYZW^2$ or $WXYZW^3$ which starts from $v$.   By symmetry, there is a directed path $P_v^{\prime}$ avoiding all used vertices and all bad vertices except $v$ with form $W^2XYZW^2$ or $W^3XYZW$ which ends at $v$.  Thus, $P_v^{\prime}$ and $P_v$  form a directed 13-path in $G$ that contains $v$ and whose endvertices lie in 
$W$.  Moreover, each of the 13-paths uses an equal number of vertices from sets $X$ and $Z$. 

The proof is analogous for the case $v\in X\cup Y\cup Z$ so we omit the details here.  Choose $|B|$ paths from $\{L_1,L_2,\ldots,L_r\}$ and embed them onto the 13-paths  we found above. By the above arguments, we have  $$\mbox{the number of remaining vertices in } X \mbox{ and } Z \mbox{ is still equal.}$$  

Set $s=r-|B|=\sqrt{\xi}n-|\mathcal{P}|-|B|$. Clearly, $s\geqslant \sqrt{\xi}n/2$ as $|\mathcal{P}|+|B|=O(\delta n)$ and $\delta\ll\xi$.  For simplicity, let $L_1,L_2,\ldots, L_s$ be the remaining $s$ directed 13-paths. 






\medskip
\noindent\textbf{Step 2: Embed paths $L_1,L_2,\ldots, L_s$  to balance the sizes of $W,X,Y$ and $Z$.}
\medskip


So far, we have successfully embedded $l-s$ paths, where all endvertices of such paths belong to $W$. Let $W^{\prime},X^{\prime},Y^{\prime},Z^{\prime}$ be the remaining sets obtained from $W,X,Y,Z$ by deleting all internal vertices of the $l-s$ paths. The argument in Step 1 shows that $|X^{\prime}|=|Z^{\prime}|$. 

Recall that each $L_i$ with $i\in[s]$ is a directed path of order 13 and each $R_i$ with $i\in[l]$ is a directed path of order 3 modulo 4 by Lemma \ref{LEM-cyclepartationfew} (iii).  Then there are $13s+3l+4h$ vertices of $C$  that have not been embedded for some $h\in \mathbb{N}$. On the other hand, as $W^{\prime}$ contains all endvertices of the $l-s$ embedde paths (of order at least 3), $G$ has  $|X^{\prime}|+|Y^{\prime}|+|Z^{\prime}|+(|W^{\prime}|-2(l-s))$ unused vertices.  Therefore, we have  $16s+4(l-s)+4h=|X^{\prime}|+|Z^{\prime}|+|Y^{\prime}|+(|W^{\prime}|-(l-s))$.  It follows by $|X^{\prime}|=|Z^{\prime}|$ that $|Y^{\prime}|+(|W^{\prime}|-(l-s))$ must be even. 

Let $2t$ be the difference between $|Y^{\prime}|$ and $|W^{\prime}|-(l-s)$. At this stage, the number of used vertices in $G$ is $|U|\leqslant |V(\mathcal{P})|+20\xi n+O(\delta n)=O(\xi n)$. Then by \ref{EP1} and the fact $\delta\ll \xi$, we have $t=O(\xi n)$. If $|W^{\prime}|-(l-s)$ is larger than $|Y^{\prime}|$, then we find $t$ directed 13-paths in the remaining $G$ with form $WXYZWXY^{2}ZW^{4}$ and, otherwise we find $t$ directed 13-paths with form  $WXYZWXY^{4}ZW^{2}$. This can be done easily by using \ref{EP3}-\ref{EP4} (recall that all bad vertices of $G$ have already been covered in Step 1) and the fact that $t=O(\xi n)$. We embed $t$ paths among $\{L_1,L_2,\ldots,L_s\}$ onto those 13-paths in $G$.

Let $W^{\prime\prime}, X^{\prime\prime}, Y^{\prime\prime}, Z^{\prime\prime}$ be the sets obtained from $W^{\prime}, X^{\prime}, Y^{\prime}$ and $Z^{\prime}$ by deleting all internal vertices of the embedded paths. Note that we already embedded $l-s+t$ paths. Next we claim that $$|X^{\prime\prime}|=|Z^{\prime\prime}| \mbox{ and } |W^{\prime\prime}|-(l-s+t)=|Y^{\prime\prime}|.$$  Here we only consider the case that $|W^{\prime}|-(l-s)=|Y^{\prime}|+2t$.  Recall that every 13-path we found in $G$ has form $WXYZWXY^{2}ZW^{4}$ and we only delete their internal vertices. Therefore, $|W^{\prime\prime}|=|W^{\prime}|-4t$, $|X^{\prime\prime}|=|X^{\prime}|-2t$, $|Y^{\prime\prime}|=|Y^{\prime}|-3t$ and $|Z^{\prime\prime}|=|Z^{\prime}|-2t$. Thus the  statements hold immediately by the fact that  $|X^{\prime}|=|Z^{\prime}|$ and $|W^{\prime}|-(l-s)=|Y^{\prime}|+2t$.







For simplicity, let $L_1,L_2,\ldots, L_{s-t}$ be the paths that have not yet been embedded. Clearly, $s-t\geqslant \sqrt{\xi}n/3$. By a similar argument, the difference between $|X^{\prime\prime}|$ and $|Y^{\prime\prime}|$ must be even. Then  one can embed $t^{\prime}=O(\xi n)$ paths from $L_1,L_2,\ldots, L_{s-t}$ onto paths in $G$ with form $(WXZ)^2(XYZ)^2W$ or $WXY^{3}ZWXYZW^{3}$ to balance  $|X^{\prime\prime}|$ and $|Y^{\prime\prime}|$. Finally, for the remaining directed 13-paths in $\{L_1,L_2,\ldots, L_{s-t}\}$, we greedily embed them onto paths in $G$ with form $(WXYZ)^3W$ by using \ref{EP3}. In the same way, we have $$|W^{\ast}|-l=|X^{\ast}|=|Y^{\ast}|=|Z^{\ast}|,$$ where $W^{\ast},X^{\ast},Y^{\ast},Z^{\ast}$ are the remaining sets of $W,X,Y,Z$ by deleting all internal vertices of $L_1,L_2,\ldots, L_l$. Figure \ref{FIG-embedfewsink} should help the reader to better understand the embedding of the cycle $C$. The two black paths in (b) correspond to the embedding of $L_1,L_2,\ldots,L_l$ and, the red and blue paths correspond to the the embedding of  $R_1,R_2,\ldots,R_l$, which will be performed in the next step.

\medskip
\noindent\textbf{Step 3: Embed paths $R_1,R_2,\ldots, R_l$ by using the Blow-up Lemma}
\medskip

Let $G^{\ast}$ be the underlying graph of $G$ with vertex set $W^\ast\cup X^\ast\cup Y^\ast\cup Z^\ast$ and edges in the ``4-cycle direction'', that is, $E(G^{\ast})=\{uv: uv\in E_G(J^{\ast},(J^+)^\ast), J\in\{W,X,Y,Z\}\}$. By Lemma \ref{LEM-cyclepartationfew} (ii), the number of vertices in $G$ that have been used at this stage is $\sum_{i=1}^l |L_i|= O(\sqrt{\xi}n)$. Moreover, by \ref{EP1} and \ref{EP3},  the following statements hold for each $J\in \{W,X,Y,Z\}$.

\medskip
(i)   $n/5\leqslant |J^\ast|\leqslant n/3$; 

(ii) For each $\sigma\in \{+,-\}$, we have $d_{G^{\ast}}(v,(J^\sigma)^\ast)\geqslant |(J^\sigma)^\ast|-O(\delta n)$ for every $v\in J^{\ast}$.

To see (ii), first recall that all bad vertices are covered in Step 1 and, there are at most $O(\delta n)$ missing edges between $v$ and $J^-\cup J^+$ in $G$ by \ref{EP3}. 
\medskip

In order to find a copy of  $C$ in $G$,  the adjacency between the two endvertices of $L_i$ and $R_i$ must be preserved when embedding $R_1,R_2,\ldots, R_l$. In other words, the endvertices of $R_i$ should be regarded as special vertices when applying the Blow-up Lemma (Lemma \ref{LEM-blowup}). By  the above property (i) and the fact $\xi\ll\alpha$, the number of these special vertices is $2l=O(\sqrt{\xi} n)\leqslant \alpha |J^{\ast}|$ for each $J\in\{W,X,Y,Z\}$. Moreover, (ii) implies that every special  $u$ has at least $|J^{\ast}|/2$ potential candidates for embedding, where $J^{\ast}$ is the set into which $u$ is supposed to be embedded.  By the fact $\delta\ll \varepsilon$ and (ii)  again,  it is not difficult to check that $(J^{\ast},(J^+)^{\ast})$ is $(\varepsilon,1/2)$-superregular for each $J\in\{W,X,Y,Z\}$.

Recall that Lemma \ref{LEM-cyclepartationfew} (iii) shows that deleting all internal vertices of $L_i$ and $L_{i+1}$ produces a directed path that contains $R_i$ and whose lengths are modulo 4. Thus the paths  $R_1,R_2,\ldots,R_l$ can be embedded into $G$ by applying the Blow-up Lemma with  $d=1/2,\Delta=2,c=1/2,k=4$ such that their endvertices (i.e., the special vertices) are properly embedded.  More precisely, the paths can be embedded as follows: if the direction of $R_i$
coincides with that of $C$, then the initial vertex of $R_i$ is embedded in $X^\ast$ and $R_i$ is embedded onto a path in $G$ with form $XYZ(WXYZ)^{j}$ for some $j$. See the blue path in Figure \ref{FIG-embedfewsink} (b) for an illustration. For the case that the two directions are opposite, then $R_i$ is embedded  onto a path with form $ZYX(WZYX)^{j}$, see the red path in Figure \ref{FIG-embedfewsink} (b) for an illustration. 
\end{proof}

We now end this section by giving the proof of Lemma \ref{LEM-cyclepartationfew}.

\begin{proof}[\textbf{Proof of Lemma \ref{LEM-cyclepartationfew}}]

First observe that the number of sinks and sources of the cycle $C$ are equal, that is, $C$ has $2\sigma(C)$ sinks and sources. Let $s_1,s_2,\ldots,s_{2\sigma(C)}$ be a list of sinks and sources of $C$ along the clockwise order. As $\sigma(C)\leqslant \xi n$ and $\xi\ll 1$, we may w.l.o.g assume that $dist_C(s_{2\sigma(C)},s_1)\geqslant 10$.

For the vertex $s_1$, let $i\geqslant 1$ be the smallest index so that $dist_C(s_i,s_{i+1})\geqslant 10$. Choose a vertex $w$ in $C[s_i,s_{i+1}]$ such that $1< dist_C(s_i,w)\leqslant 5$ and $dist_C(w,s_{i+1})\equiv 0 \pmod 4$.  Let $P_1=C[s_1,w]$ and $Q_1=C[w^+,s_{i+1}^-]$.  Clearly, $|P_1|\geqslant 3$ and $|Q_1|\geqslant 3$. In the same way, for the vertex $s_{i+1}$, choosing the smallest index $j\geqslant i+1$ such that $dist_C(s_{j},s_{j+1})\geqslant 10$, we can  obtain a path partition $(P_2,Q_2)$ of $C[s_{i+1},s_{j+1}^-]$.  We continue in this way until we obtain a path partition  $C=(P_1Q_1P_2Q_2\cdots P_rQ_r)$ for some $r$. It should be noted that  we have $Q_r=C[u^+,s_{1}^-]$ for some vertex $u^+$ as $dist_C(s_{2\sigma(C)},s_1)\geqslant 10$. To put it concisely, we split $C$ between every pair of vertices $s_i,s_{i+1}$ with $dist_C(s_{i},s_{i+1})\geqslant 10$. 

By the construction of $P_i$ and $Q_i$, we have the following:

(a) $|P_i|,|Q_i|\geqslant 3$ for each $i$ and, $\sum_{i=1}^r|P_i|\leqslant 10 \cdot 2\sigma(C)\leqslant 20\xi n$;

(b) $dist_C(P_i,P_{i+1})\equiv 0 \pmod 4$ for each $i$, where $r+1=1$;

(c) $\bigcup_{i=1}^rP_i$ covers all sinks and sources of $C$ and thus all vertices of $Q_i,i\in[r]$ are normal.

\medskip
It follows by (c) and $\sigma(C)\leqslant \xi n$ that $r=O(\xi n)$. Moreover,  (b)  implies that $|Q_i|=4t_i-1$, $i\in[t]$ for some $t_i\geqslant 1$. Then every $Q_i$ can be written as $(R_{i,1}L_{i,1}\cdots R_{i,(s_i-1)}L_{i,(s_i-1)}R_{i,s_i})$, where $s_i\geqslant 1$, $|R_{i,j}|=3,|L_{i,j}|=13$ for $j\in[s_i-1]$ and $0<|R_{i,s_i}|< 16$. Next we claim that $(P_1Q_1P_2Q_2\cdots P_rQ_r)$ can be converted into the desired partition $(L_1R_1L_2R_2\cdots L_lR_l)$ by choosing exactly $\sqrt{\xi}n$ directed 13-paths $L_{i,j}$. Note that this is possible because $\sum_{i=1}^r|Q_i|$ is sufficiently large due to  (a).  More precisely, we relabel $P_1,P_2,\ldots, P_r$ and those $\sqrt{\xi} n$  13-paths $L_{i,j}$ by $L_1,L_2,\ldots,L_{l}$ and relabel the path between $L_i$ and $L_{i+1}$ by $R_i$. Note that  $l=\sqrt{\xi} n+r\leqslant 2\sqrt{\xi}n$ as $r=O(\xi n)$. Then $(L_1R_1L_2R_2\cdots L_lR_l)$ is the desired path partition of $C$ due to (a)-(c).\end{proof}

\section{$C$ has many sinks}\label{SEC-manysink}

This section is devoted to the situation when the Hamilton cycle $C$ contains many sinks. Indeed, we will prove the following result in this section.

\begin{theorem}\label{THM-manysink}
Let $0< 1/n  \ll \xi \ll 1$. Suppose $G$ is an $n$-vertex oriented graph with $\delta^0(G)\geqslant (3n-1)/8$, then $G$ contains a copy of any $n$-vertex oriented cycle $C$ with $\sigma(C)\geqslant \xi n$.  In particular, $G$ has an antidirected Hamilton cycle. 
\end{theorem}

To prove Theorem \ref{THM-manysink}, it suffices to consider the case that $G$ has no expansion property by Theorem \ref{THM-expanderanyori}. Then by Theorem \ref{THM-extremal}, $V(G)$ has a $\delta$-extremal partition $(W,X,Y,Z)$ satisfying \ref{EP1}-\ref{EP7}.  It follows by \ref{EP1} that $|X|+|Z|\approx n/2$ and $|W|+|Y|\approx n/2$. The highly preliminary proof strategy for Theorem \ref{THM-manysink} is outlined as follows. We would like to embed half of cycle $C$ into $G[X\cup Z]$, a quarter into  $G[W]$, and another quarter into $G[Y]$, making use of the fact that these subdigraphs are nearly regular (bipartite) tournaments.  The first challenge we will face is finding
appropriate edges to connect the three parts of the embedding. Such edges are called special edges. More precisely,  an edge $e$ is called \emph{special} for $(W,X,Y,Z)$ if it belongs to $E(W\cup Z,Y\cup Z)\cup E(X\cup Y,W\cup X)$.

The next lemma shows that the degree condition $\delta^0(G)\geqslant (3n-1)/8$ ensures two disjoint special edges for $(W,X,Y,Z)$ unless $G$ has a specific structure.

\begin{lemma}\label{LEM-2matching}
Let $G$ be an $n$-vertex oriented graph  with $\delta^0(G)\geqslant (3n-1)/8$. Then every $\delta$-extremal partition  $(W,X,Y,Z)$ of $V(G)$ has a special edge. Moreover, if it has no two disjoint special edges, then there exists $v\in V(G)$ such that $G-v$ is isomorphic to the oriented graph of order $8s+2$ for some $s$ depicted in Figure \ref{FIG-degreesharp}.
\end{lemma}

By Lemma \ref{LEM-2matching}, if $G$ does not contain two disjoint special edges, then $|G|$ is odd and thus every oriented Hamiltonian cycle $C$ of $G$ is not antidirected. This is due to the fact that every antidirected cycle must be of even order. Then Theorem \ref{THM-manysink} follows immediately from Lemma \ref{LEM-2matching} and the following result.

\begin{theorem}\label{THM-specialmanysink}
Let $0< 1/n\ll\delta  \ll \xi \ll 1$ and let $G$ be an $n$-vertex oriented graph with $\delta^0(G)\geqslant 3n/8 -O(\xi n)$. Suppose $(W,X,Y,Z)$ is  a $\delta$-extremal partition of $V(G)$. If $G$ contains two disjoint special edges, then $G$ contains a copy of any $n$-vertex oriented cycle with $\sigma(C)\geqslant \xi n$. Moreover, one special edge suffices if $C$ is not antidirected.
\end{theorem}

The rest of the section is organized as follows: In Section \ref{SEC:2matching}, we first give a proof of Lemma \ref{LEM-2matching} and we show that every special edge can be extended to a well-behaved path that can be used for connection. The two key ingredients in the proof of Theorem \ref{THM-specialmanysink}, ``Lemma for $G$'' and ``Lemma for $C$'', are given in Sections \ref{SEC:lemmaG} and \ref{SEC:lemmaC} respectively, which give us a further partition of $(W,X,Y,Z)$ and a path partition of $C$. The role of the two lemmas are to obtain some special structures within $G$ and of $C$ so that
$G$ will be suitable for embedding $C$ into. Finally, the proof of Theorem \ref{THM-specialmanysink} is present in Section \ref{SEC:embedC}.

\subsection{Special edges and proper paths}\label{SEC:2matching}
First we give a proof of  Lemma \ref{LEM-2matching}.
\begin{proof}[\textbf{Proof of Lemma \ref{LEM-2matching}}]
Since $(W,X,Y,Z)$ does not contain two disjoint  special edges, either $G$ has exactly three special edges and these edges form a (not necessary directed) triangle or, all special edges for the partition are incident to the same vertex of $G$. We first consider the former case and let $T$ be the triangle. By the definition of special edges, we have $E(T)=E(W\cup Z,Y\cup Z)\cup E(X\cup Y,W\cup X)$. Let $w,x,y,z\notin V(T)$ be vertices of $W,X,Y$ and $Z$, respectively. The vertices can be found by \ref{EP1} and the fact $|V(T)|=3$. Recall that $G$ is an oriented graph. It follows by $E(X)\subseteq E(T)$ and $x\notin V(T)$ that $d(x)\leqslant|W\cup Y\cup Z|$. Similarly, we have  $d(w)\leqslant |(W\backslash\{w\})\cup X\cup Z|$, 
$d(y)\leqslant |(Y\backslash \{y\})\cup X\cup Z|$ and $d(z)\leqslant |W\cup X\cup Y|$. Then the degree condition shows  $3n-1\leqslant d(w)+d(x)+d(y)+d(z)\leqslant 3n-2$, a contradiction.

Thus it suffices to consider the case that all special edges for $(W,X,Y,Z)$ are incident to the same vertex $v$. Let $W^{\prime},X^{\prime},Y^{\prime},Z^{\prime}$ be the sets obtained from $W,X,Y$ and $Z$ by removing the vertex $v$. Observe that $d(z)\leqslant  |W^{\prime}\cup X^{\prime}\cup Y^{\prime}\cup\{v\}|$ for each $z\in Z^{\prime}$. Similarly, by considering the degree of every $w\in W^{\prime},x\in X^{\prime}$ and every $y\in Y^{\prime}$, we have  
$d(w)\leqslant  |(W^{\prime}\backslash\{w\})\cup X^{\prime}\cup Z^{\prime}\cup\{v\}|$, 
$d(x)\leqslant  |W^{\prime}\cup Z^{\prime}\cup Y^{\prime}\cup\{v\}|$ and 
$d(y)\leqslant  |Z^{\prime}\cup X^{\prime}\cup (Y^{\prime}\backslash\{y\})\cup\{v\}|$. 
Then the degree condition and these four inequalities  gives that $$8\cdot \lceil(3n-1)/8\rceil\leqslant d(z)+d(w)+d(x)+d(y)\leqslant  3n-1.$$
This means that all inequalities above must be equal, in particular, $3n-1\equiv 0\pmod 8$. Then $n=8s+3$ for some integer $s$ and $\delta^0(G)=(3n-1)/8=3s+1$. Moreover, we have $|X^{\prime}|=|Z^{\prime}|=2s+1$ and $|W^{\prime}|=|Y^{\prime}|=2s$. Then the degree condition shows that $G-v$ is isomorphic to the oriented graph of order $8s+2$ depicted in Figure \ref{FIG-degreesharp}.

Now we claim that $(W,X,Y,Z)$ has at least one special edge. Since $G$ is oriented and $|W^{\prime}|=2s$, there exist vertices $w_1,w_2\in W^{\prime}$  with $d^-(w_1,W^{\prime})\leqslant s-1$ and $d^+(w_2,W^{\prime})\leqslant s-1$. Moreover, as $3s+1\leqslant d^-(w_1)\leqslant d^-(w_1,W^{\prime})+|Z^{\prime}\cup\{v\}|$ and  $|Z^{\prime}|=2s+1$, the vertex $v$ must be an inneighbor of $w_1$. Similarly, $v$ must be an outneighbor of $w_2$  and there exist $y_1,y_2\in Y^{\prime}$ with $vy_1,y_2v\in E(G)$. Considering the original partition $(W,X,Y,Z)$ of $V(G)$, no matter which set $v$ belongs to, one of the edges in $\{vw_1,w_2v,vy_1,y_2v\}$ is special for $(W,X,Y,Z)$. 
\end{proof}

The degree condition in Lemma \ref{LEM-2matching} is best possible. Indeed, let $G$ be the oriented graph in Figure \ref{FIG-degreesharp} with $n=8s+6$. Clearly, $G$ is $(3n-2)/8$-regular and its vertex partition $(W,X,Y,Z)$ has no special edges.


\begin{definition} \emph{(Proper path)}\label{DEF-goodpath}
Let $(W,X,Y,Z)$ be a $\delta$-extremal partition of $V(G)$.  A path $P=v_1v_2\cdots v_{l}$ is called a $(W,Y)$-proper path if 
\begin{itemize}
	
	\item  $P$ is an antidirected out-path of $G$, i.e., $v_1v_2\in E(G)$;
	
	\item $v_1$ is a good vertex in $W$ and $v_{l}$ is a good vertex in $Y$.
\end{itemize}
\end{definition}

Next we show that every special edge can be extended to a proper path that can be used for connection. 

\begin{lemma}\label{LEM-8specialpath}
Let $0<1/n \ll \delta \ll 1$ and let $G$ be an $n$-vertex oriented graph with a $\delta$-extremal partition $(W,X,Y,Z)$.  For each special edge $uv$ for $(W,X,Y,Z)$, there is a $(W,Y)$-proper 13-path  
containing $uv$ in $G-S$ for any $S\subseteq V(G)\backslash \{x,y\}$ with $|S|=O(\delta n)$.
\end{lemma}

\begin{proof} 
Let $B$ be the set of bad vertices in $V(G)\backslash\{u,v\}$. It follows by \ref{EP3}-\ref{EP5} that $|B|=O(\delta n)$.  By the definition of special edge, we have $uv\in E(W\cup Z,Y\cup Z)\cup E(X\cup Y,W\cup X)$. Here, we only consider the case where $uv \in E(X\cup Y, W\cup X)$,  the case of $uv\in E(W\cup Z, Y\cup Z)$ can be handled similarly, and we omit the details. 

By \ref{EP6} and \ref{EP7}, we have $d^+(u, Y \cup Z) \geqslant n/50-O(\delta n)$ and $d^-(v, W \cup Z) \geqslant n/50-O(\delta n)$. 
As $|S\cup B|=O (\delta n)$, we can choose vertices $u_1\in N^+(u,Y\cup Z)\backslash (S\cup B)$ and $v_1\in N^-(v,W \cup Z)\backslash (S\cup B)$. Note that both $u_1$ and $v_1$ are good as they do not belong to $B$. Thus  \ref{EP3}-\ref{EP5} show that we can choose  $u_2\in N^-(u_1,Y)\backslash (S\cup B)$ and $v_2\in N^+(v_1,W)\backslash (S\cup B)$. We continue in this way until we get the desired $(W,Y)$-proper 13-path with form $W^3v_2v_1vuu_1u_2Y^4$.
\end{proof}

\subsection{Lemma for $G$: A further partition of $V(G)$}\label{SEC:lemmaG}

\begin{lemma}[Lemma for $G$]\label{LEM-splitgraph}
Let $0<1/n\ll \delta\ll1/k \leqslant \varepsilon \ll d\ll \mu\ll 1$. Suppose $G$ is an $n$-vertex oriented graph with a $\delta$-extremal partition $(W,X,Y,Z)$. Then there exist partitions $(J_0,J_1,\ldots,J_k)$ of $J\in\{W,X,Y,Z\}$ satisfying  the following. 
\begin{enumerate}[label =\upshape \textbf{(G\arabic{enumi})},ref=\upshape (G\arabic{enumi})]
	\setlength{\itemindent}{1.5em}
	\item $\mu n\leqslant |J_0|\leqslant 2\mu n$ and $J_0$ contains all bad vertices in $J$;\label{G1}

	\item $|J_i|=m$ for each $i\in[k]$, where $n/(5k)\leqslant m\leqslant n/(4k)$;\label{G2}
	
	\item $d^\pm(v,J_0) = (|J_0|/|J|)d^\pm(v,J) \pm O(dn)$  for each $v\in V(G)$;\label{G3}
	
	\item $d^\pm(v,K_i) = (m/|K|)d^\pm(v,K) \pm O(\delta n)$  for each $v\in V(G)$, $K\in\{W,Y\}$ and $i\in[k]$;\label{G4}

	\item there exists $k^{\ast}$ with $k/4\leqslant k^{\ast}\leqslant 3k/4$ such that $(Z_i,X_i)$ is $(\varepsilon,d)$-superregular for each $i\in [k^{\ast}]$ and $(X_i,Z_i)$ is $(\varepsilon,d)$-superregular for each $i\in[k]\backslash[k^{\ast}]$.\label{G5}
\end{enumerate}
\end{lemma}
\begin{proof} Let $M$ be the constant given in the Diregularity Lemma with parameters $\varepsilon^2,M^\prime=1/\varepsilon, M^{\prime\prime}=1$.  Let $0<1/n\ll  \delta \ll 1/M \ll \varepsilon \ll d\ll \mu\ll 1$.

Let $B$ denote the set of bad vertices in $G$. Clearly, $|B|=O(\delta n)$ due to \ref{EP3}-\ref{EP5}. To ensure that the final partition satisfies \ref{G1}, we first assign all vertices    in $B\cap J$ to $J_0$. However, the degrees of  $v\in V(G)$ in $G$ cannot be proportionally preserved in $G[B]$. Thus to ensure that \ref{G3} holds, we randomly select a linear number of vertices in $J$ and assign them to $J_0$. Indeed, applying Lemma \ref{LEM-randompartition} with parameter $\delta$, there exists $S_J\subseteq J\backslash B$ of size $\mu n$ such that the following holds for each  $v\in V(G)$ as $|B|=O(\delta n)$.
\begin{equation}\label{EQU-Sandomwxyz}
	\begin{aligned}
		&d^\pm(v,S_J)=\frac{|S_J|}{|J|}d^\pm (v,J)\pm O(\delta n) \mbox{ and,} \\
		&d^\pm(v,J\backslash S_J)=\frac{|J\backslash S_J|}{|J|}d^\pm (v,J)\pm O(\delta n).
	\end{aligned}
\end{equation}
Let $J^{\prime}$ be the set obtained from $J$ by removing the vertices in $S_J\cup (B\cap J)$. It follows by \ref{EP1}, $|B|=O(\delta n)$ and $|S_J|=\mu n$ that $|J^{\prime}|=n/4\pm O(\mu n)$ and $||X^{\prime}|-|Z^{\prime}||=O(\delta n)$.  Next we show how to find the wanted partitions of $W,X,Y,Z$. Roughly speaking, the partitions of $X^\prime$ and $Z^\prime$ are derived from the Diregularity Lemma and, the partitions of $W^\prime$ and $Y^\prime$ are obtained by repeatedly applying  Lemma \ref{LEM-randompartition}. 

Firstly, we give the partitions of $X$ and $Z$.  Applying the Diregularity Lemma to $G[X^\prime\cup Z^\prime]$ with $\varepsilon^2,2d,M^\prime=1/\varepsilon$ and $M^{\prime\prime}=1$, there is  a partition $V_0,V_1,\ldots,V_{p+q}$ of $X^\prime\cup Z^\prime$ with $p+q \leqslant M$ and  $|V_0|\leqslant \varepsilon^2 n$. Moreover, $X^\prime$ or $Z^\prime$ contains entire $V_i$ and   $|V_i|=m^{\prime}$ for each $i\in[p+q]$, where $m^\prime \geqslant n/(3M)$ as $|X^{\prime}\cup Z^{\prime}|=n/2\pm O(\mu n)$ and $|V_0|\leqslant \varepsilon^2 n$.  
For simplicity, we rename the sets as $X^{\prime}\cap V_0,X_1,\ldots, X_p, Z^{\prime}\cap V_0,Z_1,\ldots,Z_q$ depending on which set $V_i$ belongs to. The vertices of $K^{\prime}\cap V_0$ are also placed into $K_0$ for each $K\in\{X,Z\}$. Let $R$ be the reduced oriented graph of $G[X^\prime\cup Z^\prime]$ and let $R_X, R_Z$ be the vertex sets of $R$ which corresponding to the clusters in $X$ and in $Z$, respectively. Set $r=\min\{p,q\}$. Clearly, $rm^{\prime}=\min\{|X^{\prime}\backslash V_0|,|Z^{\prime}\backslash V_0|\}$ and  $|R|=|R_X|+|R_Z|=p+q\geqslant 2r$.

Next we claim  that $R$ has no edges between any two vertices of $R_Z$. Suppose to the contrary that $Z_1Z_2\in E(R)$ with $Z_1,Z_2\in R_Z$. Then $(Z_1,Z_2)$ is an $(\varepsilon^2,2d)$-regular pair in $G[X\cup Z]$ and thus $e(W\cup Z,Y\cup Z)\geqslant e(Z)\geqslant2d (m^{\prime})^2$. As $\delta\ll 1/M\ll d$ and $m^\prime \geqslant  n/(3M)$, we have  $e(W\cup Z,Y\cup Z)\gg \delta n^2$, which contradicts the property \ref{EP2}. Hence, there is no edge in $R[R_Z]$ and thus $d^+_R(Z_i,R_X)=d^+_R(Z_i), d^-_R(Z_i,R_X)=d^-_R(Z_i)$  for any  $Z_i\in R_Z$. Since all vertices in $X^\prime\cup Z^\prime$ are good, \ref{EP5} gives that $d^+(v,X)\geqslant |X|/2-O(\delta n)$ for each $v\in Z^{\prime}$. Thus $$d_G^+(v,X^{\prime})\geqslant d_G^+(v,X\backslash S_X)-|B|\overset{\scriptstyle\text{(\ref{EQU-Sandomwxyz})}}{\geqslant} \frac{|X\backslash S_X|}{2}-O(\delta n) \geqslant |X^\prime|/2-O(\delta n) \mbox{ for any }  v\in Z^{\prime}.$$  Then by  (\ref{EQU-orientreduced}) and  $||X^{\prime}|-|Z^{\prime}||=O(\delta n)$, for any $Z_i \in R_Z$, we have 
\begin{align*}
	d^+_R(Z_i,R_X)=d^+_R(Z_i)&\geqslant \left(\frac{e_G(Z_i,X^\prime)}{|X^\prime\cup Z^\prime||Z_i|} -(2d+3\varepsilon^2)\right)(p+q)\\
	&\geqslant \left(\frac{|X^\prime| -O(\delta n)}{2|X^\prime|+O(\delta n)} - (2d+3\varepsilon^2)\right)(p+q)\\
	&\geqslant \frac{r}{2}-5dr.
\end{align*}

By symmetry, we have  $d^-_R(Z_i,R_X)=d^-_R(Z_i)\geqslant r/2-5dr$ and thus  $d_R(Z_i,R_X)\geqslant r-10dr$ for each $Z_i \in R_Z$. Then we can greedily choose a matching $M\subseteq E(R)$ of size $k=r - 10dr$ between $R_X$ and $R_Z$ such that each of the two directions has more than $r/4 -5dr/2$ edges.  For simplicity, we write the edges in $M$ from $R_Z$ to $R_X$ as $\{Z_1X_1,\ldots,Z_{k^\ast}X_{k^\ast}\}$ and write the edges from $R_X$ to $R_Z$ as $\{X_{k^\ast+1}Z_{k^\ast+1},\ldots,X_kZ_k\}$. Clearly,  $k/4\leqslant k^\ast\leqslant 3k/4$. Put all vertices in $X_i$ with $k<i\leqslant p$ and $Z_j$ with $k<j\leqslant q$ into $X_0$ and $Z_0$, respectively. Moreover, we remove $\varepsilon^2 m^\prime$ vertices from each of the two clusters associated with an edge of $M$ and place them into $X_0$ and $Z_0$ so that the cluster pair becomes $(\varepsilon,d)$-superregular. Note that this can be done by Lemma \ref{LEM-regularsuper}.  Set $m=(1-\varepsilon^2)m^\prime$. Then $(Z_0,Z_1,\ldots,Z_k)$ and $(X_0,X_1,\ldots,X_k)$ are two partitions of $Z$ and $X$, respectively, and $|Z_i|=|X_i|=m$ for each $i\in[k]$.

We now construct the partitions of $W$ and $Y$ as follows. Let $J\in\{W,Y\}$. For each $i\in[k]$,  Lemma \ref{LEM-randompartition} shows that there is a set $J_i$ of size $m$ in $J^\prime\backslash( \bigcup_{j=1}^{i-1}J_j)=:S_i$ such that $d^{\pm}(v, J_i)\geqslant (m/|S_i|)d^{\pm}(v,S_i)-\delta n$ and $d^{\pm}(v, S_i\backslash J_i)\geqslant (1-m/|S_i|))d^{\pm}(v,S_i)-\delta n$ for each $v\in V(G)$. Note that the process above can continue $k$ times. This is because $|J^\prime| \geqslant|J|-|B|-|S_J|\geqslant n/4 -O(\delta n) -\mu n$,  $\delta \ll d$ and 
\begin{align*}
	km=(r-10dr)m&\leqslant (1-10d)\max\{|X^\prime|,|Z^\prime|\}\\
	&\leqslant (1-10d)(n/4+O(\delta n) -\mu n) \\
	&\leqslant n/4 -\mu n -2dn.
\end{align*}
Remove all remaining vertices of $J^{\prime}$ into $J_0$. Then $|J_i|=m$ for each $i\in[k]$ and $(J_0,J_1,\ldots,J_k)$ is a partition of $J$ for  $J\in \{W,Y\}$.

Next we show that for each $J\in\{W,X,Y,Z\}$, the partition  $(J_0,J_1,\ldots,J_k)$ we constructed  earlier satisfies the conditions \ref{G1}-\ref{G5}. Clearly, the conditions \ref{G2}, \ref{G4} and \ref{G5} are satisfied by the construction. Recall that $J_0$ contains all vertices in $S_J\cup (B\cap J)$ and $|S_J|= \mu n$. Thus to ensure \ref{G1} holds, it suffices to show $|J_0|\leqslant 2\mu n$. To see this, we need to compute the lower bound of $|J_1\cup J_2\cup \cdots \cup J_k|$, i.e., $km$. Recall that $r=\min\{p,q\}$, $rm^{\prime}=\min\{|X^{\prime}\backslash V_0|,|Z^{\prime}\backslash V_0|\}$ and $m=(1-\varepsilon^2)m^{\prime}$, where $V_0$ is the exceptional set given in the Diregularity Lemma. Therefore, 
\begin{align*}
	km=(r-10dr)m&\geqslant (1-10d)(1-\varepsilon^2) (\min\{|X^\prime|,|Z^\prime|\}-|V_0|)\\
	&\geqslant (1-10d)(n/4-O(\delta n) -|B|-\mu n-O(\varepsilon^2 n))\\
	&\geqslant n/4-\mu n-3dn,
\end{align*}
which implies that $|J_0|=|J|-km\leqslant O(\delta n)+\mu n+3dn$. This proves \ref{G1} due to $\delta \ll d\ll \mu$.  Moreover, the condition \ref{G3} is satisfied by the inequality (\ref{EQU-Sandomwxyz}) and the fact $|J_0\backslash S_J|\leqslant O(\delta n)+3dn=O(dn)$, which completes the proof. 
\end{proof}

The following observations follow immediately from the Lemma for $G$. 

\begin{observation}\label{OBS-degree}
Let $S\subseteq V(G)$ be a set with $|S|=O(dn)$ that contains all bad vertices of $G$.

\indent (i) Both $J_0\cup S$ and $J_0\backslash S$ satisfy \ref{G3} for each  $J\in\{W,X,Y,Z\}$.

(ii) Let $J^{\prime}$ be obtained from a collection of several sets from $\{J_0,J_1,\ldots, J_k\}$ by adding or deleting some vertices in $S$. We have $d^\pm(v,J^{\prime}) = (|J^{\prime}|/|J|)d^\pm(v,J)\pm O(dn)$ for each $J\in\{W,Y\}$ and $v\in V(G)$. In particular, if $J^\prime$ has no bad vertices, then $\delta^0(G[J^{\prime}])\geqslant |J^{\prime}|/2-O(d n)$ by \ref{EP4}.
\end{observation}

\begin{observation}\label{OBS-manypath} 
Let $0<\varepsilon\ll \xi \leqslant 1$. Suppose $J^{\prime}$ is a set of $V(G)$ with $\delta^0(G[J^{\prime}])\geqslant (3/8+\xi)|J^{\prime}|$. Let $\mathcal{P}$ be a collection of oriented paths with order at least $1/\varepsilon$ and $|V(\mathcal{P})|\leqslant \xi|J^{\prime}|/2$.   Then every path in  $\mathcal{P}$  can be embedded into $G[J^{\prime}]$ such that its copy starts and ends at the prescribed vertices. 
\end{observation}

To see Observation \ref{OBS-manypath},  we note that $G[J^{\prime}]$ is a robust $(\nu,\tau)$-outexpander due to Lemma \ref{LEM-semitoexpander}, where $\varepsilon\ll\nu\ll \tau\ll\xi$. Moreover, Lemma \ref{LEM-Connect} show that the paths in $\mathcal{P}$ can be embedded sequentially with all previously used vertices avoided.

\subsection{Lemma for $C$: A path partition of $C$}\label{SEC:lemmaC}

In this subsection, we split the cycle $C$ into a collection of subpaths, where every three consecutive subpaths belongs to one of the specifically defined types below.

\begin{definition}\label{DEF-Ptype} Let $\varepsilon$ be a constant and let $Q_{i-1},P_i,Q_{i}$ be three disjoint subpaths of an $n$-vertex oriented cycle $C$. The 3-tuple $(Q_{i-1},P_i,Q_{i})$ is said to be of 

type I,  if both $Q_{i-1}$ and $Q_{i}$ are directed and $1/\varepsilon \leqslant |P_i|\leqslant n^{1/3}$;

type II, if $Q_{i-1}$ is directed,  $v_{i}Q_{i}u_{i+1}$ is an antidirected path whose vertices are all non-normal vertices of $C$, where $v_{i},u_{i+1}$ are the vertices before and after $Q_{i}$ on $C$ respectively, and $|P_i| = 2/\varepsilon$;

type III, if $v_{i-1}Q_{i-1}u_i$ is an antidirected path whose vertices are all non-normal vertices of $C$, where $v_{i-1},u_{i}$ are the vertices before and after $Q_{i-1}$ on $C$, respectively. Moreover, $P_i$ is antidirected with $n^{1/6}\leqslant |P_i|\leqslant 2n^{1/6}$.
\end{definition}
For simplicity, we will sometimes say $P_i$ is of type I, II or III when no ambiguity arises.

\begin{lemma}[Lemma for $C$]\label{LEM-splitcyclemany}
Let $0<1/n\ll \delta  \ll 1/k \leqslant \varepsilon \ll \xi \ll 1$.  Suppose $C$ is an $n$-vertex oriented cycle with $\sigma(C)\geqslant \xi n$. Let $k^{\ast},m$ be integers with $k/2\leqslant k^{\ast}\leqslant k$ and $n/(5k)\leqslant m\leqslant n/(4k)$. Then $C$ has a path partition $(L_YQ_0P_1Q_1\cdots P_tQ_t P)$ such that 
\begin{enumerate}[label =\upshape \textbf{(C\arabic{enumi})},ref=\upshape (C\arabic{enumi}),series=mylist]
	\setlength{\itemindent}{1.5em}
	\item  $|L_Y|= \sqrt{\xi}n$, $\sigma(L_Y)\geqslant \xi^2 n$ and $|Q_i|=3$ for each $i\in [t]$;\label{C1}
	
	\item for each $i\in[t]$, the 3-tuple $(Q_{i-1},P_i,Q_i)$ is of type I, II or III;\label{C2}
	
	\item $Q_0$ is either an antidirected $13$-path whose initial vertex is a source of $C$ or it is a directed $3$-path. Moreover, $Q_0$ is antidirected if and only if $C$ is antidirected.\label{C3}
\end{enumerate}
Moreover, the paths $P_1,P_2,\ldots,P_t$ can be divided into $k+1$ families $\mathcal{P}_0,\mathcal{P}_1,\ldots,\mathcal{P}_{k}$ such that there is a partition $I_1,I_3$  of $[k]$ satisfying the following.
\begin{enumerate}[label =\upshape \textbf{(C\arabic{enumi})},ref=\upshape (C\arabic{enumi}),resume=mylist]
	\setlength{\itemindent}{1.5em}
	
	\item $\mathcal{P}_i$ consists of paths $P_j$ of type I for $i \in I_1$  and of type III for $i \in I_3$, respectively;\label{C4}

	\item $\mathcal{P}_0$ contains all paths $P_j$ of type II and $|V(\mathcal{P}_0)|\leqslant 4m$;\label{C5}
	
	\item for each $i\in I_1$, we have $|V(\mathcal{P}_i)|\in[\,3m-2\delta  n,3m-\delta  n\,]$;\label{C6}
	
	\item for each $i\in I_3\cap[k^{\ast}]$ we have $|V(\mathcal{P}_i)|\in[\,2m-2\delta  n, 2m-\delta  n\,]$ and,\\ for each $i\in I_3\backslash [k^{\ast}]$, we have  $|V(\mathcal{P}_i)|\in [\,4m-2\delta  n, 4m-\delta  n\,]$.\label{C7}
\end{enumerate}
\end{lemma}

It should be noted that $t\leqslant \varepsilon n$ holds by \ref{C2} and Definition \ref{DEF-Ptype}. Indeed, this bound follows directly from the fact that each path $P_i$ has order at least $1/\varepsilon$. Moreover, \ref{C1} and \ref{C3} imply that  $\sum_{i=0}^t|Q_i|\leqslant 3t + 13 \leqslant 4\varepsilon  n$.

\begin{proof} Since $C$ has at least $\xi n$ sinks and $\beta \ll \xi \ll 1$, a simple calculation shows that there exists a subpath $L_Y$ of $C$ with $|L_Y|=\sqrt\xi n$ containing at least $(\xi n) \sqrt\xi n/|C| \geqslant \xi^2 n$ sinks. Note that we can find such a path $L_Y$ so that the segment after $L_Y$ (along the direction of $C$) is a directed $3$-path when $C$ is not antidirected. Let $Q_0$ be the directed 3-path when $C$ is not antidirected and otherwise let $Q_0$ be the antidirected $13$-path after $L_Y$ such that its first vertex is a source of $C$. Clearly, $L_Y$ satisfies the conditions \ref{C1} and \ref{C3}.

Before presenting a path partition that satisfies \ref{C2} and \ref{C4}-\ref{C7}, we first introduce a primary partition $(L_YQ_0L_1R_1\cdots L_sR_s)$ of $C$ with some $s\geqslant 1$ as follows. Set $R_0=Q_0$. For each $i\geqslant 1$, we assign the vertices of $C$  that lie after $R_{i-1}$ to $L_i$, continuing this assignment until $|L_i|\geqslant 1/\varepsilon$ and   either a directed 3-path is encountered or the initial vertex of $L_Y$ is reached. Let $R_{i}$ denote the directed 3-path following $L_i$ if such a path exists, otherwise set $R_i=\emptyset$. Note that $|R_i|=3$ for each $i\in[s-1]$ and $|R_s|\in\{0,3\}$. 

Now we construct a new path  partition $(L_YQ_0P_1Q_1\ldots P_lQ_l)$ of $C$ by further splitting those long paths $L_j$. Here, we say that $L_j$ is long if  $|L_j|> n^{1/3}$, otherwise it is short. Clearly, $C$ has at most $n^{2/3}$ long paths $L_j$. Moreover,  the construction  of $L_j$ shows that a long $L_j$ is antidirected if we ignore its first $1/\varepsilon +1$ vertices. For each long $L_j$, we further split it into short paths such that $L_j=(L_j^{\prime}Q_j^1P_j^1\cdots Q_j^rP_j^rQ_j^{r+1}L_j^{\prime\prime})$, where $|L_j^{\prime}|=2/\varepsilon $,   $|Q_j^i|=3, |P_j^i|=n^{1/6}$ for each $i$ and $n^{1/6}\leqslant|L_j^{\prime\prime}|\leqslant 2n^{1/6}$. Then we can obtain the desired  path partition  $C=(L_YQ_0P_1Q_1\cdots P_lQ_l)$ by relabeling those paths, where $Q_l=R_s$. It is not difficult to check that for each $i\in[l-1]$, we have $|Q_i|=3$ and $(Q_{i-1},P_i,Q_i)$ is of type I, II or III. Indeed, $(Q_{i-1},P_i,Q_i)$ is of

\begin{itemize}
	\item type I, if $P_i$ is some short $L_j$ in the primary partition;
	\item type II, if $P_i$ is the first segment of some long $L_j$, i.e., $P_i=L_j^{\prime}$;
	\item type III, otherwise.
\end{itemize}
Moreover, there are at most $n^{2/3}$ paths of type II. By Definition \ref{DEF-Ptype}, every path has order at least $1/\varepsilon$ and thus $l\leqslant \varepsilon n$.  Since $|Q_i|=3$ for each $i\in[l-1]$, $|Q_0|=\{3,13\}$ and  $|Q_l|=|R_s|=\{0,3\}$, we have  $\sum_{i=0}^l|Q_i|\leqslant 3l + 13 \leqslant 4\varepsilon n.$ Observe that if we can show there exists $t \in [l]$ such that $P_1, P_2, \ldots, P_t$ meet \ref{C4}-\ref{C7}, then by setting  $P = P_{t+1}Q_{t+1}\cdots P_lQ_l$, the partition $(L_YQ_0P_1Q_1\cdots P_tQ_tP)$ will satisfy \ref{C1}-\ref{C7}. 

Next we assign paths $P_i$ into  $\mathcal{P}_0,\mathcal P_1,\ldots ,\mathcal P_{k+1}$ and find the integer $t$.  Initially, set $\mathcal{P}_i=\emptyset$ for each $i$ and set $\alpha =1$, $\beta  =2, \gamma=3$.
While $\gamma \leqslant k+2$, process $i=1,2,\ldots,l$ sequentially, and do the following.
\begin{enumerate}
	\item If $P_i$ is of type~I, then assign $P_i$ to $\mathcal P_{\alpha }$. Furthermore, after this assignment, if  $|V(\mathcal P_{\alpha })|\in[\,3m-2\delta  n,\;3m-\delta  n\,]$, then reset $\alpha=\gamma$ and $\gamma =\gamma +1$;  
	
	\item If $P_i$ is of type~II, then assign $P_i$ to $\mathcal{P}_0$;
	
	\item If $P_i$ is of type~III, then assign $P_i$ to $\mathcal P_{\beta}$. Moreover, set $\alpha =\gamma$ and $\gamma=\gamma+1$ if one of the following holds after the  assignment.
	
	$\bullet$ If $\beta  \in[k^\ast]\ \text{ and }\ |V(\mathcal P_{\beta  })|\in[\,2m-2\delta  n,\;2m-\delta  n\,]$;
	
	$\bullet$ If $\beta  \in [k+1]\backslash [k^\ast]\ \text{ and }\ |V(\mathcal P_{\beta  })|\in[\,4m-2\delta  n,\;4m-\delta  n\,]$.
\end{enumerate}
We say that $\mathcal{P}_j$ is \emph{saturated} if its number of vertices in $\mathcal{P}_j$ lies within the above intervals. By Definition \ref{DEF-Ptype}, the assignment of $P_i$ to $\mathcal{P}_j$ increases $|V(\mathcal{P}_j)|$ by $|P_i|\leqslant n^{1/3}\ll \delta  n$. Moreover, as $m\geqslant n/(5k)$ and $1/k\gg \delta $, we have $2m-2\delta  n\gg 0$.  Recall that there are at most $n^{2/3}$ paths $P_i$ of type II. Let $P_t$ be the last path assigned. We now show that $t$ is much smaller than $l$.  To see this, as $k^{\ast}\geqslant k/2$ and $m\leqslant n/(4k)$, we have   $\sum_{j=0}^{k+1}|V(\mathcal{P}_j)|=\sum_{i=1}^t|P_i|\leqslant (2
/\varepsilon)n^{2/3} +(k+1-k^\ast)4m + 3k^\ast m\leqslant 9n/10$. Since $\sum_{i=0}^l|Q_i| \leqslant 4\varepsilon n$ and $|L_Y|=\sqrt{\xi}n$, we have  $|P|=|P_{t+1}\cup Q_{t+1}\cup\cdots\cup Q_l|\geqslant n -\sqrt{\xi}n- 9n/10-4\varepsilon  n\geqslant n/20$. Recall that  $|P_i|\leqslant n^{1/3}$  and $|Q_i|\leqslant 13$ for each $i$. Hence, $t<l$ and thus the procedure can terminate.

Next we give the desired  partition  $I_1,I_3$ and $\mathcal{P}_0,\mathcal{P}_1,\ldots, \mathcal{P}_k$. Let $a,b$ be the value of $\alpha,\beta$ before the assignment of $P_t$. Clearly, $\max\{a+1,b+1\}=k+2$. Moreover,  the family into which $P_t$ is placed, w.l.o.g say $\mathcal{P}_a$, will become saturated after $P_t$ is assigned, whereas $\mathcal{P}_{b}$ remains unsaturated. Relabeling $\mathcal{P}_1,\mathcal{P}_2,\ldots, \mathcal{P}_{k+1}$, we may assume that every $\mathcal{P}_j$ with $j\in[k]$ is saturated. We remove all paths of $\mathcal{P}_{k+1}$ into $\mathcal{P}_0$.  Let $I_1,I_3\subseteq [k]$ denote the sets of indices of all  $\mathcal{P}_j$ that consist of paths of type I and type III, respectively.  Then $I_1,I_3$ is the desired partition of $[k]$ and $\mathcal{P}_0,\mathcal{P}_1,\ldots, \mathcal{P}_k$ are the wanted families satisfying \ref{C4}-\ref{C7} as $|V(\mathcal{P}_0)|\leqslant (2/\varepsilon)n^{2/3}  +|V(\mathcal{P}_{k+1})| \leqslant 4m$. This completes the proof.
\end{proof}

\subsection{Proof of Theorem \ref{THM-specialmanysink}}\label{SEC:embedC}

For convenience, let us first recall the statement of the theorem. 

\medskip
\noindent\textbf{Theorem \ref{THM-specialmanysink}.} \textit{Let $0< 1/n\ll\delta  \ll \xi \ll 1$ and let $G$ be an $n$-vertex oriented graph with $\delta^0(G)\geqslant 3n/8 -O(\xi n)$. Suppose $(W,X,Y,Z)$ is  a $\delta$-extremal partition of $V(G)$. If $G$ contains two disjoint special edges, then $G$ contains a copy of any $n$-vertex oriented cycle with $\sigma(C)\geqslant \xi n$. Moreover, one special edge suffices if $C$ is not antidirected.}
\medskip

In what follows, we establish the proof of this theorem. Before delving into the proof, we first provide some basic analysis. All parameters, with the following order, will be used throughout the section.  $$0<1/n\ll \delta \ll  1/k\leqslant\varepsilon\ll d\ll \mu\ll \xi\ll 1.$$

Let $G$ be an $n$-vertex oriented graph with $\delta^0(G)\geqslant 3n/8 -O(\xi n)$ and let $C$ be any $n$-vertex oriented cycle with $\sigma(C)\geqslant \xi n$. Suppose $(W,X,Y,Z)$ is a $\delta$-extremal partition of $V(G)$ such that it has a special edge and moreover, it has two disjoint special edges if $C$ is antidirected. From Lemma \ref{LEM-8specialpath}, every special edge is contained in a $(W,Y)$-proper 13-path of $G$. Moreover,  any two disjoint special edges can be extended to two disjoint $(W,Y)$-proper 13-paths. Indeed, these two paths can be constructed sequentially so that the second path avoids all vertices of the first one. This means that $G$ has a  proper 13-path $P^{\ast}$, moreover, it has two disjoint proper 13-paths $P^{\ast}, P^{\ast\ast}$ if $C$  is antidirected. Let $\hat{W},\hat{X},\hat{Y},\hat{Z}$ be the resulting sets obtained from $W,X,Y$ and $Z$ by deleting all vertices of $P^{\ast}$  and $P^{\ast\ast}$. Clearly, $(\hat{W},\hat{X},\hat{Y},\hat{Z})$ is a vertex partition of the resulting graph $\hat{G}$ which satisfies \ref{EP1}-\ref{EP7}.  Applying  Lemma \ref{LEM-splitgraph} to $\hat{G}$ with $n=\hat{n}:=|\hat{G}|$, there is a partition $(\hat{J_0},\hat{J_1},\ldots,\hat{J_k})$ of  $\hat{J}$ satisfying \ref{G1}-\ref{G5} for each $J\in  \{W,X,Y,Z\}$.

Set $G^{\ast}=G-P^{\ast}$. For each $J\in\{W,X,Y,Z\}$, set $J^{\ast}=J\backslash V(P^{\ast})$ and $J_i=\hat{J_i}$ for $i\in[k]$ and, let  $J_0=\hat{J_0}\cup (V(P^{\ast\ast})\cap J)$ if $C$ is antidirected  and otherwise let $J_0=\hat{J_0}$.  In other words, if $C$ is antidirected, then we incorporate all vertices of $V(P^{\ast\ast})\cap J$ into $J_0$. It is not difficult to check that  $(W^{\ast},X^{\ast},Y^{\ast},Z^{\ast})$ is a $\delta$-extremal partition of $V(G^{\ast})$ and  $(J_0,J_1,\ldots, J_k)$ is a partition of $J^{\ast}$  satisfying \ref{G1}-\ref{G5} as each set only incorporates constant vertices.  In particular, there exists $k^{\ast}$ with $k/4\leqslant k^{\ast}\leqslant 3k/4$ such that $(Z_i,X_i)$ is $(\varepsilon,d)$-superregular for each $i\in [k^{\ast}]$ and $(X_i,Z_i)$ is $(\varepsilon,d)$-superregular for each $i\in[k]\backslash[k^{\ast}]$. We may further assume $k^{\ast}\geqslant k-k^\ast$, i.e., $k^{\ast}\geqslant k/2$, by reversing all edges of $G^{\ast}$ and swapping the labels of $W^{\ast}$ and $Y^{\ast}$.

Due to \ref{G2} and the fact that $\sigma(C)\geqslant \xi n$,  Lemma \ref{LEM-splitcyclemany} implies that there is a path partition $C=(L_YQ_0P_1Q_1\cdots P_tQ_t P)$ satisfying \ref{C1}-\ref{C7}. In the following, we divide the proof of Theorem \ref{THM-specialmanysink} into the following two claims. The first one shows that we can embed the subpath $Q_0P_1Q_1\cdots P_tQ_t$ of $C$ into $G^\ast$ very effectively and the second one shows that there is a copy of $PL_Y$ in the remaining $G$ and thus there is a copy of $C$ in the whole $G$.

Suppose $w^{\ast}$ is the vertex after $Q_t$ and $y^{\ast}$ is the vertex before $Q_0$  on the cycle $C$, respectively. In other words,  $w^{\ast}$ is the initial vertex of $P$ and $y^{\ast}$ is the terminal vertex of $L_Y$.

\begin{claim}\label{CLM-embedpath}
There is a copy $H^{\ast}$ of $y^{\ast} Q_0P_1Q_1\cdots P_tQ_tw^{\ast}$ in $G^{\ast}$. Suppose the vertices $w^{\ast}, y^{\ast}$ are embedded onto $w^{\ast}_G$ and $y^{\ast}_G$, respectively. For each $J\in\{W,X,Y,Z\}$, let $J^\prime$ be the set obtained from $J$ by removing all internal vertices of $H^{\ast}$. Then the following holds.
\begin{enumerate}[label =\upshape \textbf{(H\arabic{enumi})},ref=\upshape (H\arabic{enumi})]
	\setlength{\itemindent}{1.5em}
	\item  $w^{\ast}_G\in W_0,y^{\ast}_G\in Y_0$ are two good vertices;\label{H1}

	\item   $|W^{\prime}|\geqslant \mu n/2$, $3\mu n\geqslant |X^{\prime}|,|Z^{\prime}|\geqslant \mu n/2$ and $|Y^{\prime}|\geqslant n/20$;\label{H2}

	\item $d^\pm(v,J^\prime)=(|J^\prime|/|J|)d^\pm(v,J) \pm O(d n)$ for any $v\in V(G)$.\label{H3}

\end{enumerate}
\end{claim}

It should be noted that $w^{\ast}_G\in W^{\prime}, y^{\ast}_G\in Y^{\prime}$ since we only remove all internal vertices of $H^{\ast}$. The property \ref{H2} shows that the embedding of $y^{\ast} Q_0P_1Q_1\cdots P_tQ_tw^{\ast}$ covers nearly  all vertices of $X^{\ast},Z^{\ast}$. On the other hand,  the remaining sets of $W,X$ and $Z$  still contain a linear number of vertices, while  the remaining set of $Y$ is of significant size. Moreover, the degrees are preserved proportionally within these resulting subsets by \ref{H3}.

\begin{claim}\label{CLM-connectcycle}
Let $G^{\prime}=G[W^{\prime}\cup X^{\prime}\cup Y^{\prime}\cup Z^{\prime}\cup V(P^{\ast})]$. Then $G^\prime$ contains a copy $H^{\prime}$ of $PL_Y$  which starts from $w^{\ast}_G$ and ends at $y^{\ast}_G$.
\end{claim}

\begin{proof}[\textbf{Proof of Claim \ref{CLM-embedpath}}]
Recall that the paths $P_1,P_2,\ldots,P_t$ can be divided into $k+1$ families $\mathcal{P}_0,\mathcal{P}_1,\ldots, \mathcal{P}_k$ by Lemma \ref{LEM-splitcyclemany}. Moreover, by \ref{G5} and the arguments before Claim \ref{CLM-embedpath}, there exists $k^{\ast}$ with $k/2\leqslant k^{\ast}\leqslant 3k/4$ such that $(Z_i,X_i)$ is $(\varepsilon,d)$-superregular for each $i\in [k^{\ast}]$ and $(X_i,Z_i)$ is $(\varepsilon,d)$-superregular for each $i\in[k]\backslash[k^{\ast}]$. It follows by \ref{EP3}-\ref{EP5} that $WXYZ$ is close to a blowup of the directed 4-cycle. Thus for each $i,j\in[k]$, it is not difficult to check that each of $(W_i,X_j)$, $(X_i,Y_j)$, $(Y_i,Z_j)$ and $(Z_i,W_j)$ is $(\varepsilon,d)$-superregular as every vertex is good by \ref{G1}. It is natural to consider that, by using Lemma \ref{LEM-randomwind} or the Blow-up Lemma, the paths in $\mathcal{P}_i$ can be embedded at random by ``winding'' around  one of ``triangles'' $Z_iX_iY_iZ_i,X_iZ_iW_iX_i$ or one of ``edges'' $Z_iX_i, Z_iW_i$ and $W_iX_i$ in $G^{\ast}$. Based on this, we divide the proof into the following three steps. In Step 1, we prepare for the idea stated above. Specifically, we introduce a strategy for paths in 
$\mathcal{P}_i$ to wind around a triangle or an edge. Then we first embed the connecting paths $Q_0,Q_1,\ldots,Q_t$ in Step 2 and, in Step 3 we give an embedding of paths $P_1,P_2,\ldots, P_t$.

\medskip
\noindent\textbf{Step 1: The assignment of the paths $P_1,P_2,\ldots,P_t$.}
\medskip

It is worth noting that in this step, we do not embed  $P_1,P_2,\ldots,P_t$ into $G^{\ast}$, instead,  we perform the winding operation on these paths in a virtual graph whose vertex set is $\{J_i: 0\leqslant i\leqslant k, J\in\{W,X,Y,Z\}\}$, and where  every edge $J_iK_i$ corresponds to a superregular pair $(J_i,K_i)$ in $G$ with $J\neq K\in\{W,X,Y,Z\}$.

Recall that $\mathcal{P}_i$ with $i\in I_1$ (resp., with $i\in I_3$) consists of paths of type I (resp., type III) due to Lemma \ref{LEM-splitcyclemany}. Moreover,  Definition \ref{DEF-Ptype} shows that every path of type III is an antidirected path. Now we formally present the winding strategy,  see Figure  \ref{FIG-wind}  for an illustration.

\begin{enumerate}[label =\upshape \textbf{(S\arabic{enumi})},ref=\upshape (S\arabic{enumi}),series=mylist1]
	\setlength{\itemindent}{1.5em}
	\item Assign all paths of $\mathcal{P}_0$ to $W_0$;\label{S1}
	
	\item For each $i\in I_1\cap [k^\ast]$, assign all paths of $\mathcal{P}_i$ to $Z_iX_iY_iZ_i$ by applying Lemma \ref{LEM-randomwind};\label{S2}
	
	\item For each $i\in I_1\backslash [k^\ast]$, assign all paths of $\mathcal{P}_i$ to $X_iZ_iW_iX_i$ by applying Lemma \ref{LEM-randomwind};\label{S3}
	
	\item For each $i\in I_3\cap [k^\ast]$, assign all paths in $\mathcal{P}_i$ to $Z_iX_i$ in such a way that every source  of $C$ on these paths lies in  $Z_i$; \label{S4}
	
	\item  For each $i\in I_3\backslash [k^\ast]$, choose two unused $W_{i_1},W_{i_2}$ with $i_1\neq i_2\in [k]$ and, assign paths in $\mathcal{P}_i$ to one of $Z_iW_{i_1}$ and $W_{i_2}X_i$ so that  every source  of $C$ on these paths lies in $Z_i$ or $W_{i_2}$.\label{S5}
\end{enumerate}

It should be noted that only the assignments in \ref{S3} and \ref{S5} use vertices $W_j$. As $k^{\ast}\geqslant k/2$, the assignments above use at most $2(k- k^{\ast})\leqslant k$ vertices $W_j$  in the virtual graph.

Next we estimate how many times each vertex of the virtual graph is traversed. Clearly, \ref{S1} and \ref{C5} show that the vertex $W_0$ of the virtual graph is used $|V(\mathcal{P}_0)|\leqslant 4m$ times. For each $\mathcal{P}_i$ with $i\in I_1$,  every path $P_j\in \mathcal{P}_i$ is of type I with  $1/\varepsilon \leqslant|P_j|\leqslant n^{1/3}$ by Definition \ref{DEF-Ptype}.  Furthermore, as indicated by \ref{C6}, we have $|V(\mathcal{P}_i)|\in [3m-2\delta  n,3m-\delta  n]$. Applying Lemma \ref{LEM-randomwind} with $\varepsilon = \delta ^2$, each of vertices $W_i,X_i,Y_i,Z_i$ in the virtual graph is traversed $|V(\mathcal{P}_i)|/3\pm \delta ^2 n$ times after the assignment in \ref{S2} or \ref{S3}. More precisely,  the number of times each of those vertices is traversed lies between $m-2\delta n$ and $m-\delta   n/5$.

For each $\mathcal{P}_i$ with $i\in I_3$, we get that every path $P_j\in \mathcal{P}_i$ is an antidirected path of type III with $n^{1/6}\leqslant |P_j|\leqslant 2n^{1/6}$ by \ref{C4} and Definition \ref{DEF-Ptype}. Moreover, by \ref{C7}, for each $i\in I_3\cap[k^{\ast}]$ we have $|V(\mathcal{P}_i)|\in[\,2m-2\delta  n, 2m-\delta  n\,]$  and thus there are at most $2m/n^{1/6}\leqslant 2\varepsilon n^{5/6}\ll \delta ^2 n$ paths in $\mathcal{P}_i$.  On the other hand,  the assignment in \ref{S4} shows that the number of times each path in $\mathcal{P}_i$ with $i\in I_3\cap [k^{\ast}]$ passes through the vertex $Z_i$ and the vertex $X_i$ differs by at most 1. Therefore, after the assignment in \ref{S4}, the number of times each of $Z_i$ and $X_i$ is traversed lies between $m-2\delta   n$ and $m-\delta   n/5$. Similarly, for each $\mathcal{P}_i$ with $i\in I_3\backslash [k^{\ast}]$, the number of times each of $Z_i, W_{i_1},W_{i_2}$ and $X_i$ is traversed lies between $m-2\delta   n$ and $m-\delta   n/5$.

\medskip
\noindent \textbf{Step 2: An embedding of paths  $Q_0,Q_1,\ldots,Q_t$.}
\medskip

Note that the adjacency between endvertices of $P_{j},P_{j+1}$ and $Q_{j}$ must be preserved when embedding those  paths.  This implies that we need to embed $Q_j$ in accordance with the winding strategy of $P_j,P_{j+1}$ mentioned in Step 1. Let $u_j,v_j$ be the initial and terminal vertices of $P_j$, respectively. Set $v_0=y^{\ast}$ and $u_{t+1}=w^{\ast}$, where $w^{\ast},y^{\ast}$ are given in the statement of the claim. Then $Q_j$ is a $(v_j^+,u_{j+1}^-)$-path for each $0\leqslant j\leqslant t$. 

Let $B$ be the set of bad vertices in $G$. (Note that the endvertices of special edges in $P^{\ast},P^{\ast\ast}$ may be bad vertices, however, these endvertices are excluded from our consideration of bad vertices in the proof.) It follows by \ref{EP3}-\ref{EP5} that $|B|=O(\delta n)$. A rough idea is to embed $Q_0,Q_1,\ldots,Q_t$ into $G^{\ast}[(W_0\cup X_0\cup Y_0\cup Z_0)\backslash B]$ and, moreover, embed $Q_0$ onto $P^{\ast\ast}$ if $P^{\ast\ast}$ exists. Note that this is possible as $\sum_{i=0}^t|Q_i|\leqslant 4\varepsilon n$ by Lemma \ref{LEM-splitcyclemany} and the ``4-cycle direction'' degrees of $G^{\ast}[(W_0\cup X_0\cup Y_0\cup Z_0)\backslash B]$, as well as the degrees in $G^{\ast}[W_0\backslash B]$ and $G^{\ast}[Y_0\backslash B]$ are large enough by Observation \ref{OBS-degree} (i) and $\delta\ll\varepsilon\ll d$. Before giving an embedding of $Q_0,Q_1,\ldots, Q_t$, we first assign the vertices $w^{\ast}$ and $y^{\ast}$ as follows.
\begin{enumerate}[label =\upshape \textbf{(S\arabic{enumi})},ref=\upshape (S\arabic{enumi}),resume=mylist1]
	\setlength{\itemindent}{1.5em}
	\item Assign the vertex $y^{\ast}$ to $Y_0\backslash B$ and assign the vertex $w^{\ast}$ to $W_0\backslash B$.\label{S6}
\end{enumerate}

Next we greedily embed $Q_j$ based on the assignments of the vertices $v_j$ and $u_{j+1}$ in Step 1. For simplicity, let $Q_j=v_j^+q_ju_{j+1}^-$ if $|Q_j|=3$. It follows by \ref{C1} and \ref{C3} that $|Q_i|=3$ for each $i\in[t]$ and $|Q_0|\in\{3,13\}$. Moreover, if $|Q_0|=13$, then it is antidirected.

\medskip
\textbf{Case 1.} $Q_j$ is a directed path with $|Q_j|=3$.
\medskip

It is not difficult to check that  $v_j^+,u_{j+1}^-$ can be embedded into $(W_0\cup Y_0)\backslash B$ no matter which sets $v_j$ and $u_{j+1}$ are assigned to due to \ref{EP3} and \ref{G3}.  Then the vertex $q_j$ can be embedded properly depending on the sets $v_j^+$ and $u_{j+1}^-$ belong to. For example, if $v_j$ is assigned to some vertex $Z_i$ in Step 1, then we embed $v_j^+$ into $W_0\backslash B$ if the edge between $v_j$ and $v_j^+$ is oriented from $v_j$ to $v_j^+$ and otherwise embed $v_j^+$ into  $Y_0\backslash B$. Moreover, assume that $v_j^+$ and $u_{j+1}^-$ are embedded to $W_0\backslash B$ and $Y_0 \backslash B$, respectively, then we embed $q_j$ into  $X_0\backslash B$ if the directed path $Q_j$ is oriented from $v_j^+$ to $u_{j+1}^-$ and otherwise embed $q_j$  into  $Z_0\backslash B$.

\medskip
\textbf{Case 2.} $Q_j$ is an antidirected path.
\medskip

If $j=0$, i.e., $Q_0$ is antidirected, then \ref{C3} shows that $Q_0$ is an antidirected 13-path, where both of its endvertices are sources of $C$. Moreover,  $C$ is antidirected and thus the proper 13-path $P^{\ast\ast}$ exists by the arguments before Claim \ref{CLM-embedpath}. Then $Q_0$ can be embed onto the path $P^{\ast\ast}$ so that the copy of $Q_0$ in $G$ starts from the  endvertex of $P^{\ast\ast}$ in $Y_0\backslash B$ and ends at the other  endvertex of  $P^{\ast\ast}$ in $W_0\backslash B$.  By \ref{S6}, $y^{\ast}$ is assigned to a vertex in $Y_0\backslash B$.  Then the adjacency between $Q_0$ and $y^{\ast}$ is preserved due to Observation \ref{OBS-degree} (ii).  Next we claim that the adjacency between $Q_0$ and $ P_1$ is preserved.  By \ref{C2} and Definition \ref{DEF-Ptype}, the tuple $(Q_0,P_1,Q_1)$ is of type III. Moreover, as the endvertices of $Q_0$ are sources of $C$, the initial vertex $u_1$ of $P_1$ is a sink of $C$.  It follows by \ref{S5} that $u_1$ is assigned into $W_{i_1}\cup X_i$ for some $i,i_1\in[k]$. Recall that the terminal vertex of $Q_0$ is embedded onto a vertex in $W_0\backslash B$ and every vertex of  $W_0\backslash B$ has many outneighbors in both $W_{i_1}$ and $X_i$ by \ref{G4} and \ref{EP3}-\ref{EP4}.  Hence there are many choices for the embedding of $u_1$ (in Step 3) based on its assignment mentioned in Step 1.

For the case $j\geqslant 1$, we have $|Q_j|=3$ and thus  $Q_j=v_j^+q_ju_{j+1}^-$.  Again, as $Q_j$ is antidirected, \ref{C2} and Definition \ref{DEF-Ptype} imply that $P_{j+1}$ is of  type III (or $P_{j+1}=u_{j+1}=w^{\ast}$)  and $v_jv_j^+q_ju_{j+1}^-u_{j+1}$ is an antidirected subpath of $C$. Moreover, as $Q_j$ is not directed, $(Q_{j-1},P_j,Q_j)$, and consequently $P_j$, is of type II or III. By \ref{S1} and \ref{S4}-\ref{S6}, the vertices $v_{j}$ and $u_{j+1}$ are assigned to $(W_0\backslash B) \cup W_i\cup X_{i}\cup Z_{i}$ for some $i\in[k]$. Moreover, for each of $v_j$ and $u_{j+1}$, if it is assigned to $X_{i}$ (resp., to $Z_{i}$), then it is a sink (resp., source of $C$).

We embed $Q_j=v_j^+q_ju_{j+1}^-$ into $G^{\ast}[W_0\backslash B]$ and next we show that the embedding preserves the adjacency. Here, we only consider the case that both $P_j$ and $P_{j+1}$ are of type III, and we further assume that all$v_j,q_j$ and $u_{j+1}$ are sinks of $C$.  This implies that $v_j^+,u_{j+1}^-$ are sources of $C$. By \ref{S4}-\ref{S5}, the vertex $v_j$ is assigned to $W_{i_1}\cup X_{i}$. It follows by  $v_j^+$ is a source that $v_j^+v_j$ is an edge of $C$. Then the embedding of $v_j^{+}$ is feasible for the edge $v_j^+v_j$ as every vertex in $W_0\backslash B$  has many outneighbors in $W_{i_1} \cup X_i$. In the same way, the embedding of $u_{j+1}^-$ preserves the adjacency between $u_{j+1}^-$ and $u_{j+1}$ and thus we can embed $Q_j=v_j^+q_ju_{j+1}^-$ into $G^{\ast}[W_0\backslash B]$.

\medskip
\noindent\textbf{Step 3: An embedding of paths  $P_1,P_2,\ldots,P_t$.}
\medskip

Let $U\subseteq V(G)$ denote the set of all bad vertices and vertices used in Step 2. Then $|U|= |B|+\sum_{i=0}^t|Q_i|=O(\varepsilon n)$ by Lemma \ref{LEM-splitcyclemany}.  Recall that the embedding in Step 2 already ensures the adjacency between $P_i$ and $Q_i$. Thus to embed  $P_i$ properly, it suffices to show its  endvertices $u_i,v_i$ have many possible positions for embedding.  

Set $P_0:=y^{\ast}$ and $v_0:=y^{\ast}$. By symmetry, we only estimate the number of all possible positions of  terminal vertices $v_i$ of $P_i$ with $0\leqslant i\leqslant t$.  Assume that $v_i$ is assigned to $J_j$ in Step 1 and $v_i^+$ is embedded to $K_0\backslash B$ in Step 2 for some $J,K\in\{W,X,Y,Z\}$ and $j\in[t]$. Recall that the adjacency between $v_i^+$ and $v_i$ is preserved.  This means that, for example, if $v_i$ is assigned to $W_j$ and $v_iv_i^+$ is oriented from $v_i$ to $v_i^+$, then in Step 2 we embed $v_i^+$ into $W_0\backslash B$ or $X_0\backslash B$. In both cases, \ref{G4} and $|U|= O(\varepsilon n)$ imply that the copy of $v_i^+$ in $G$ has at least $|W_j|/3$ unused inneighbors in $W_j$ and thus $v_i^+$  has $|W_j|/3$ possible positions for embedding.

Next we present a rigorous embedding for $P_1,P_2\ldots, P_t$. Recall that all paths in $\mathcal{P}_0$ are assigned to $W_0$  by \ref{S1}. For each path $P_i\in \mathcal{P}_0$, we choose an unused possible position in $W_0\backslash U$ for each of $u_i$ and $v_i$. Then we may greedily embed all paths of $\mathcal{P}_0$ in $G^{\ast}[W_0\backslash U]$ by applying Observation \ref{OBS-degree} (ii) and Observation \ref{OBS-manypath} with $J^{\prime}=W_0\backslash U$.  Furthermore, based on the assignments in \ref{S2}-\ref{S5}, we can embed paths into $\mathcal{P}_i$ with $i\in[k]$ by using the Blow-up Lemma and the arguments before Step 1. It should be noted that $t\leqslant \varepsilon n$ by Lemma \ref{LEM-splitcyclemany}  and the endvertices $u_i,v_i$ with $i\in[t]$ are $2t$ ``special vertex'' when applying the Blow-up Lemma.  Therefore, there is a copy $H^{\ast}$ of $y^{\ast} Q_0P_1Q_1\cdots P_tQ_tw^{\ast}$ in $G^{\ast}$.

Let $J^{\prime} = J\backslash V(H^{\ast})$ for $J\in\{W,X,Y,Z\}$ and let $w^{\ast}_G, y^{\ast}_G$ be the copies of $w^{\ast}, y^{\ast}$ in $G$.  Then \ref{H1} clearly holds by \ref{S6}. It follows by \ref{G1} and \ref{C5}  that $\mu n\leqslant |J_0|\leqslant 2\mu n$ and $|V(\mathcal{P}_0)|\leqslant 4m\leqslant 4\varepsilon n$. Moreover, \ref{C1} and \ref{C3} imply that  $\sum_{i=0}^t|Q_i|\leqslant 3t + 13 \leqslant 4\varepsilon  n$.  Recall that we only embed some paths in $\mathcal{P}_0$ and $\cup_{i=0}^t Q_i$ into $G^{\ast}[J_0]$. Hence $|J^{\prime}|\geqslant |J_0|-4\varepsilon n-4\varepsilon  n\geqslant \mu n/2$. On the other hand, by the arguments below \ref{S1}-\ref{S5} in Step 1, for each $K\in\{X,Z\}$ and $i\in[k]$, there are at most $2\delta  n$ unused vertices in $K_i$ and thus $|K^{\prime}|\leqslant |K_0|+2k\delta  n\leqslant 3\mu n$.  Note that only the assignment in \ref{S2} uses sets $Y_i$ of $Y$. Since $k^{\ast}\leqslant 3k/4$ and $km\leqslant |Y|$, we have that $|Y^{\prime}|\geqslant |Y|-|Y_0|-k^{\ast}m\geqslant n/4-3n/16-O(\mu n)\geqslant n/20$, which proves \ref{H2}. Note that for each $J\in\{W,X,Y,Z\}$,  every set $J_i$ with $i\in[k]$ has at least $m-2\delta  n$ vertices that either all used or all unused in the assignment/embedding process in Step 1/3. Moreover, we use at most $|V(\mathcal{P}_0)|+\sum_{i=0}^t|Q_i| \leqslant 8\varepsilon n$ vertices of $W_0\cup X_0\cup Y_0\cup Z_0$ in the embedding process. Thus \ref{H3} follows by Observation \ref{OBS-degree} with $|S|\leqslant 4k\times (2\delta  n)+5\varepsilon n=O(dn)$, which completes the proof.
\end{proof}

We next end this section by giving the proof of  Claim \ref{CLM-connectcycle}.

\begin{proof}[\textbf{Proof of Claim \ref{CLM-connectcycle}}]

Recall that $C$ has a partition $(L_YQ_0P_1Q_1\cdots P_tQ_t P)$ due to  Lemma \ref{LEM-splitcyclemany}. Moreover, by Claim \ref{CLM-embedpath}, there is a copy $H^{\ast}$ of $y^{\ast} Q_0P_1Q_1\cdots P_tQ_tw^{\ast}$ in $G^{\ast}$ which starts from $y^{\ast}_G\in Y^{\prime}$ and ends at $w^{\ast}_G\in W^{\prime}$, where $w^{\ast}$ is the initial vertex of $P$ and $y^{\ast}$ is the terminal vertex of $L_Y$. Therefore, to embed $C$ into $G$, it suffices to embed $PL_Y$ in the remaining oriented graph $G^{\prime}$ such that the copy of $PL_Y$ starts from $w^{\ast}_G$ and ends at $y^{\ast}_G$.

We first give a further partition $P=L_WP_WLP_Y$ of $P$ as follows: Let $L_W$ be the first $dn$ vertices of $P$ and let $P_W$ be the subpath of $P$ after $L_W$ such that $|L_W|+|P_W|=|W^\prime|+\sqrt{\delta} n$. Note that it is possible as $|P|+|L_Y|=|G^{\prime}|\geqslant |W^{\prime}|+n/20$ by \ref{H2} and \ref{C1}. Moreover, if the 12-path after $P_W$ contains a normal vertex of $C$, say $x$, then we extend $P_W$ to $x^-$ and let $L=x$. Otherwise,  every vertex of this 12-path is either a sink or a source of $C$, then we extend $P_W$
by at most one vertex to yield an antidirected 11-path $L$, whose initial vertex is a sink of $C$. Let $P_Y$ denote the remaining subpath of $P$. 

Recall that $w^{\ast}$ is the initial vertex of $L_W$ and $y^{\ast}$ is the terminal vertex of $L_Y$. For simplicity, we assume that $P_W$ starts from $w^2$ and ends at $w^1$ and, $P_Y$ starts from $y^1$ and ends at $y^2$.  By the above construction, we have the following:
\begin{enumerate}[label=\upshape\textbf{(\alph*)},ref=\upshape(\alph*)]
	\item $|W^{\prime}|+\sqrt{\delta}n\leqslant|L_W|+|P_W|\leqslant |W^{\prime}|+\sqrt{\delta}n+12$ and $|L_W|=dn$;\label{a}
	
	\item $w^1Ly^1$ is either a directed 3-path or, an antidirected out-path of order 13;\label{b}
	
	\item $L_W$ contains either $\varepsilon  n$ disjoint directed $13$-paths or $\varepsilon  n$ sinks of $C$.\label{c}
\end{enumerate}

To see (c),  note that the numbers of sinks and sources  in any oriented path differ by at most 2. Moreover, removing all sinks and sources of $C$ from  $L_W$ results in a collection of  directed paths.  If $L_W$ contains neither $\varepsilon n$ disjoint  directed $13$-paths nor  $\varepsilon n$ sinks,  we have $|L_W| \leqslant (2\varepsilon  n +2)+ 13\varepsilon  n + 12\times (2\varepsilon  n +3)< d n$, a contradiction.

Recall that $B$ is the set of bad vertices of $G$. We first update the set $B$ and embed the path $L$ into $G^{\prime}$ as follows.  If $w^1Ly^1$ is antidirected, we regard $P^\ast$ as a copy of $w^1Ly^1$.  For the case $w^1Ly^1$ is a directed 3-path, we put all vertices of $P^{\ast}$ into $B$ and embed $w^1Ly^1$ into $G^{\prime}$ with form $WXY$ or $WZY$ such that its initial and terminal vertices belong to $W^{\prime}\backslash B, Y^{\prime}\backslash B$, respectively. This is possible due to \ref{EP3}, \ref{H3} and the fact $|B|=O(\delta n)$. We use $w^1_G$ and $y^1_G$ to denote the copies of $w^1$ and $y^1$ in $G^{\prime}$, respectively.

Let $U\subseteq V(G^{\prime})$ denote the set of vertices used in the embedding of $L$. Let  $B^{\sigma}=\{x\in B:d^{\sigma}(x,W^\prime)\geqslant \sqrt{d}n\}$, where $\sigma\in\{+,-\}$. Clearly, $|U\cup B^+\cup B^-|=O(\delta n)$ and $\{w^1_G,y^1_G\}\cap U=\emptyset$. Next we divide the proof into the following three steps.

\medskip
\noindent\textbf{Step 1: Embed $L_W$ to cover  bad vertices in $B^+\cup B^-$.}
\medskip

Recall that $w^{\ast}$ is embedded onto the vertex $w^{\ast}_G\in G$. We further embed $w^2$, the initial vertex of $P_W$, into $G[W^{\prime}\backslash (B\cup U)]$ and say $w^2_G$ is the corresponding copy of $w^2$. It follows by \ref{H2}-\ref{H3} that each of $W^{\prime},X^{\prime},Y^{\prime}$ and $Z^{\prime}$ is of linear size and $G^\prime$ inherits the degree properties of all vertices in $G$. Moreover,  \ref{H3}, \ref{EP4} and  $| B\cup U|=O(\delta n)\ll dn$  imply that $$\delta^0(G[W^\prime\backslash  (B\cup U)])\geqslant |W^\prime\backslash(B\cup U)|/2-O(dn).$$

By \ref{c}, $L_W$ contains either $\varepsilon n$  directed 13-paths or  $\varepsilon n$ sinks. If the former case holds, we construct an embedding of $L_W$ in two steps. First we use directed 13-paths of $L_W$ to cover all bad vertices in $B^+\cup B^-$ by applying the same arguments mentioned in Step 1 of the proof of Lemma \ref{LEM-balbancecycle}. Note that the  endvertices of those 13-paths belong to the set $W^{\prime}\backslash (B\cup U)$. This is possible as $|B^+\cup B^-|=O(\delta n)$ and $L_W$ has at least $\sqrt{\delta}n$ disjoint directed 13-paths. The second step is to embed the remaining subpaths of $L_W$ into $G[W^\prime\backslash (B\cup U)]$ by applying Observation \ref{OBS-manypath}. Note that Observation \ref{OBS-manypath} requires $|L_W|=dn\leqslant \xi|W^{\prime}\backslash (B\cup U)|/2$ and, every remaining subpath to have order at least $1/\varepsilon$ and this can be satisfied by \ref{H1} and by properly selecting the $\sqrt{\delta}n$ directed 13-paths embedded in the first step as $\delta \ll \varepsilon$. 

For the case that $L_W$ contains $\varepsilon n$ sinks of $C$, it clearly contains $\varepsilon n-2$ sources of $C$. Hence, there are $\sqrt{\delta}n$ sinks (resp., sources) of $C$, where the distance between any two sinks (resp., sources) on $L_W$ is at least $2/\varepsilon$. As $|B\cup U|=O(\delta n)$,  each $x\in B^+$ has at least $\sqrt{d}n/2$ outneighbors in $W^{\prime}\backslash (B\cup U)$. Then we can greedily choose two of such outneighbors and embed one of the $\sqrt{\delta}n$ sources of $L_W$ onto the vertex $x$. In the same way, the vertices in $B^-$ can be covered by using the $\sqrt{\delta}n$ sinks of $L_W$. Therefore, all vertices of $B^+\cup B^-$ can be covered by embedding $|B^+\cup B^-|$ disjoint 3-paths with form $W(B^+\cup B^-)W$.  Since the distance between any two sinks (resp., sources) we used is at least $2/\varepsilon$, each of the remaining subpaths of $L_W$ has order at least $1/ \varepsilon$ and thus they can be embeded into the remaining $G[W^{\prime}\backslash B]$ by Observation \ref{OBS-manypath} again. More precisely, the path $L_Ww^2$ can be embedded into $G[W^{\prime}\cup B^+\cup B^-]$ with form
$w^{\ast}_G W\cdots W (B^+W\cdots W)^{|B^+|} (B^-W\cdots W)^{|B^-|}w^2_G$. It should be noted that in both cases, there is a copy of $L_Ww^2$ in $G^{\prime}$ which  starts from $w^{\ast}_G$ and ends at $w^2_G$ due to Observation \ref{OBS-manypath}.

\medskip
\noindent\textbf{Step 2: Embed $P_W$ to cover vertices in $W^\prime\backslash( B\cup U)$.}
\medskip

Update the set $U$ such that it now consists of all vertices on the copies of $L$ and $L_W$ in $G^{\prime}$. Clearly, $dn+1\leqslant |U|=|L\cup L_W|\leqslant dn+11$ by \ref{b}. Moreover, by the embedding of $L_W$, all but at most $13|B^+\cup B^-|$ vertices of $L_W$ are embedded into $G[W^{\prime}]$. This gives that   $dn -O(\delta n)\leqslant |W^\prime\cap U|\leqslant dn+11$. It follows by \ref{a} that $|W^{\prime}|-dn+\sqrt{\delta}n\leqslant|P_W|\leqslant |W^{\prime}|-dn+\sqrt{\delta}n+12$. Obviously, $|W^\prime|=|W^\prime\backslash(B\cup U)| +|W^{\prime}\cap (B\cup U)|$. Then by $|B|=O(\delta n)$ we have $$|W^\prime\backslash(B\cup U)|+\sqrt{\delta}n/2\leqslant |P_W|\leqslant |W^\prime\backslash(B\cup U)|+2\sqrt{\delta}n.$$
Let $l=|P_W|-|W^{\prime}\backslash (B\cup U)|$. Clearly, $\sqrt{\delta}n/2\leqslant l\leqslant 2\sqrt{\delta}n$.

Since $l\leqslant2\sqrt{\delta}n$ and $\delta \ll \mu$, we get a conclusion similar to \ref{c} in Step 1: Upon removing the endvertices $w^1,w^2$, $P_W$ contains either $l$ sinks of $C$ or $l$ disjoint directed $13$-paths.  For the former case, let $P_W^{\ast}$ be an auxiliary path obtained from $P_W$ by contracting $l$ sinks. In other words, $P_W^{\ast}$ is obtained from $P_W$ by replacing $s^-_is_is^+_i$ with an edge $s_i^-s_i^+$ for each $i\in[l]$, where $s_i$ is a sink. Clearly, $|P^\ast_W|=|W^\prime\backslash (B\cup U)|$. Moreover, by $|W^\prime|\geqslant \mu n/2 \gg dn$ and $|B\cup U|\leqslant 2dn$, we have $\delta^0(G[W^\prime\backslash (B\cup U)])\geqslant |W^\prime\backslash (B\cup U)|/2-O(dn)$. Thus $G[W^\prime\backslash (B\cup U)]$ contains a copy of $P^\ast_W$ which starts from $w^2_G$ and ends at $w^1_G$ by Lemmas \ref{LEM-semitoexpander} and \ref{LEM-Connect}.  Note that the copy of every vertex of $P^\ast_W$ is a good vertex in $W^{\prime}$ and thus it has at most $O(\delta n)$ missing edges to $X^{\prime}$ by \ref{EP3} and \ref{H3}. Thus for each $i\in[l]$, the copies of $s_i^-$ and $s_i^+$ have at least $|X^{\prime}\backslash (B\cup U)|-O(\delta n)$ common outneighbors in $X^{\prime}\backslash (B\cup U)$, which means that $s_i$ has many possible position to embed. Therefore, the embedding of $P_W^{\ast}$ can be extended to a copy of $P_W$ in $G^{\prime}$ due to \ref{H2} and  $|B\cup U|\leqslant 2dn\ll \mu n$. See Figure  \ref{FIG-auxiliary}  for an illustration.

\begin{figure}[H]
	\begin{center}
		\begin{tikzpicture}[black,line width=1pt,scale=0.4]
			\path (-12,6) (12,-6); 
			\draw (0,5) ellipse (6 and 1);
			\draw (7,0) ellipse (1 and 3);
			\draw (-7,0) ellipse (1 and 3);
			\draw (0,-5) ellipse (5 and 1);
			
			\foreach \i/\j in {(5.4,5)/w1,(3.6,5)/w2,(2.4,5)/w3, (0.6,5)/w4,(-0.6,5)/w5,(-2.4,5)/w6,(-3.6,5)/w7,(-5.4,5)/w8,(7,2.5)/x1,(7,-1)/x2,(-7,1)/z1,(7,1)/x3,(0,-5)/y1,(-7,-1)/z2}{\filldraw[black]\i circle (1.5pt)coordinate(\j);}

			\foreach \i/\j in {w2/w3,w4/w5,w6/w7}{\draw[line width=1.5pt,red,dashed] (\i) -- (\j);}

			\foreach \i/\j in {w2/x1,w3/x1,w4/x2,x2/z2,z2/w5,x3/y1,z1/w7,y1/z1,w6/x3} {\draw[leftlearrow={latex},line width=1.5pt,red](\i) -- (\j);}
			
			\foreach \i/\j in {w1/w2,w3/w4,w5/w6,w7/w8}{\draw[decoration={aspect=-0.05, segment length=2mm, amplitude=1mm,coil},decorate,line width=0.05cm] (\i) -- (\j);}
			
			\draw (0,5) ellipse (6 and 1);
			\draw (7,0) ellipse (1 and 3);
			\draw (-7,0) ellipse (1 and 3);
			\draw (0,-5) ellipse (5 and 1);
			\foreach \i/\j in {(5.4,5)/w1,(3.6,5)/w2,(2.4,5)/w3, (0.6,5)/w4,(-0.6,5)/w5,(-2.4,5)/w6,(-3.6,5)/w7,(-5.4,5)/w8,(7,2.5)/x1,(7,1)/x2,(-7,-1)/z1,(7,-1)/x3,(0,-5)/y1,(-7,1)/z2}{\filldraw[black]\i circle (1.5pt)coordinate(\j);}
			
			\node at (9,5) [black] {$W^\prime\backslash(B\cup U)$};
		\end{tikzpicture}
	\end{center}
	\caption{An illustration of how to cover all vertices in $W^\prime\backslash(B\cup U)$ by embedding $P_W$. The path in $G[W^\prime\backslash(B\cup U)]$ is the auxiliary path $P_W^{\ast}$ obtained from $P_W$ by replacing several short subpaths with edges.}
	\label{FIG-auxiliary}
\end{figure}

Thus we may assume that $P_W$ contains $l$ disjoint directed $13$-paths, and consequently, it has $l$ directed 4-paths and 5-paths.  Let $\alpha,\beta>0$ be two integers such that $l = 2\alpha+3\beta$ and let $P_W^{\ast}$ be the path obtained from $P_W$ by contracting $\alpha$ directed 4-paths and $\beta$ directed 5-paths. In the same way, $G[W^\prime\backslash (B\cup U)]$ has a copy of $P_W^{\ast}$  starting from $w^2_G$ and ending at $w^1_G$, and this copy can be extended to an embedding of $P_W$ in $G^{\prime}$.  More precisely, assume that $wxyzw^{\prime}$ is a directed 5-path of $P_W$ and thus $ww^{\prime}$ belongs to $P_W^{\ast}$ by the construction. By \ref{EP3} and \ref{H3},  every vertex has $O(\delta n)$ missing edges in the ``4-cycle direction''. Thus $w$ has $|X^{\prime}|-O(\delta n)\geqslant \mu n/4$ outneighbors in $X^{\prime}$ and then $x$ has $\mu n/4-|B\cup U|\geqslant \mu n/8$ possible positions. Similarly, there are $\mu n/8$ possible positions for $z$. Since every pair of vertices in $X$ and $Z$ has at least $|Y^{\prime}|-O(\delta n)$ common neighbors in $Y^{\prime}$, the vertex $y$ has many positions for embedding.  By the same argument, assume $wxzw^{\prime}$ is a 4-path of $P_W$, then $x$ has $\mu n/8$ possible positions. Moreover,  by \ref{EP5} and \ref{H3}, every vertex in $X^{\prime}\backslash B$ has more than $\mu n/8$ outneighbors in $Z^{\prime}\backslash (B\cup U)$ and thus $z$ has $\mu n/8-O(\delta n)$ possible positions for embedding. As $l\ll \mu n$,  we can embed $P_W$ properly  using the exact method described above.

\medskip
\noindent\textbf{Step 3: Embed $L_Y$ and $P_Y$ to cover all remaining vertices.}
\medskip

By the definitions of $B^+$ and $B^-$,  we have $d(v,W^\prime)\leqslant 2\sqrt{d} n$ for each $v\in B\backslash (B^+\cup B^-)$. It follows by \ref{H3} and $|W^\prime|\geqslant \mu n/2$ that $d(v,W)\leqslant (|W|/|W^\prime| )d(v,W^\prime)+O(dn)\leqslant O(\xi n)$.  On the other hand,  \ref{EP1} and the degree condition imply $d(v,W\cup Y)\geqslant n/4 -O(\xi n)$ and thus $d(v,Y)\geqslant |Y| -O(\xi n)$. By \ref{H2}-\ref{H3} again, we have that $d(v,Y^{\prime})\geqslant |Y^{\prime}| -O(\xi n)\geqslant n/25$. 

Recall that in Step 1, all vertices in $B^+\cup B^-$ are covered by the embedding of $L_W$ and thus  $B^+\cup B^- \subseteq U$. We further embed the vertex $y^2$ into $G[Y^{\prime}\backslash(B\cup U)]$ and let $y^2_G$ be its copy. Let $U^\sigma=\{v\in (B\cup  X^\prime\cup Z^{\prime})\backslash U: d^\sigma(v,Y^\prime\backslash U)\geqslant n/200\}$. Observe that $U^+,U^-$ is a partition of $(B\cup X^{\prime}\cup Z^{\prime})\backslash U$ by \ref{EP3}-\ref{EP4},  \ref{H2}-\ref{H3} and the fact $|U|=O(dn)$. Moreover, \ref{H2} and $|B|=O(\delta n)\ll \mu n$ imply that $|U^+\cup U^-|\leqslant |B| + |X^\prime \cup Z^\prime|\leqslant 7\mu n$.  Recall that $|L_Y|=\sqrt{\xi}n $ and $L_Y$ contains at least $\xi^2 n$ sinks by \ref{C1}. By the same argument as in Step 2,  there is a copy of  $y^2L_Y$ in the remaining $G^{\prime}$ with form $y^2_G(Y\cdots YU^+)^{|U^+|} (Y\cdots YU^-)^{|U^-|}Y\cdots Yy^{\ast}_G$. 
Let $Y^{\ast\ast}$ be the set of remaining vertices in $Y^\prime$. Clearly, $|Y^{\prime}\backslash Y^{\ast\ast}|\leqslant O(\sqrt{\xi}n)$ and $|Y^{\ast\ast}|\geqslant n/30$. By a similar argument, \ref{EP4} and \ref{H3} yield that $\delta^0(G[Y^{\ast\ast}])\geqslant |Y^{\ast\ast}|/2-O(\sqrt{\xi}n)$ and thus  $G[Y^{\ast\ast}]$ contains a copy of $P_Y$ which starts from $y^1_G$ and ends at $y^2_G$ by Lemmas \ref{LEM-semitoexpander} and \ref{LEM-Connect}.

Therefore, we get an embedding of $C$, i.e., $L_YQ_0P_1Q_1\cdots P_tQ_tL_WP_WLP_Y$, which completes the proof.
\end{proof}

\section{Pancyclicity}\label{SEC-pancyclicity}

In this section, we are going to prove Theorem \ref{THM-anyorianylength}.  We first mention some results on short cycles in digraphs. Recall that a central problem in digraph theory, Caccetta-H\"aggkvist Conjecture,  states that every oriented graph on $n$ vertices with minimum outdegree $d$ contains a directed cycle of length at most $n/d$. For the existence of a directed 3-cycle, Shen \cite{shenJCTB74} proved minimum outdegree $0.355 n$ suffices and, Hamburger,  Haxell and  Kostochka \cite{hamburgeEJC14} showed that if one consider minimum semidegree instead of  minimum outdegree, then the constant can be improved slightly, they showed that $\delta^0(G)\geqslant 0.346 n$ ensures a directed 3-cycle. For transitive 3-cycles, an $n/3$-blowup of the directed 3-cycle shows that minimum semidegree $\delta^0(G)\geqslant n/3$ is not sufficient  and it is not difficult to check that $\delta^0(G)\geqslant n/3+1$ suffices.

Theorem \ref{THM-kellyanyorianyl} shows that minimum semidegree $(3/8+\varepsilon)n$ ensures a cycle of every possible orientation and of every possible length in an oriented graph. Kelly, K\"uhn and Osthus  proved that for cycles of constant lengths the bound can be improved, but the improvement depends on the \emph{cycle-type}, where the cycle-type is the number of forward edges minus the number of backward edges. 

\begin{theorem}[\cite{kellyJCTB100}]\label{THM-shortcycle}
(i) Let $l\geqslant 4$ and  $\varepsilon >0$. There exists $n_1=n_1(l,\varepsilon)$ such that if $G$ is an oriented graph on $n\geqslant n_1$ vertices with  $\delta^0(G)\geqslant (1/3+\varepsilon)n$ then $G$ contains every possible orientation of an $l$-cycle.

(ii) Let $\gamma>0$ and let $l$ be some positive constant. There exists $n_2=n_2(\gamma,l)$ such that every oriented graph $G$ on $n\geqslant n_2$  vertices with $\delta^0(G)\geqslant \gamma n$ contains every oriented cycle of length at most $l$ and cycle-type 0.
\end{theorem}

Note that antidirected cycles are special cases of cycles with cycle-type 0. In \cite{kuhnEJC34} K\"uhn,  Osthus and  Piguet  considered the case that the cycle-type is sufficient large with respect to $k$. 
\begin{theorem}[\cite{kuhnEJC34}]
Let $k \geqslant 3$ and $t \geqslant 10^8k^6$. Suppose that $k$ is the minimal integer greater than 2 that does not divide $t$. Then for all $\varepsilon>0$ and all $l$ there exists an integer $n_0 = n_0(\varepsilon, l)$ such that every oriented graph $G$ on $n \geqslant n_0$ vertices with  $\delta^0(G)\geqslant (1 + \varepsilon)n/k$ contains any oriented $l$-cycle of cycle-type $t$.
\end{theorem}

It is worth noticing that for some cycle-types and lengths, a semidegree exceeding $n/3$ is necessary, for example, an $n/3$-blowup of the directed 3-cycle has no directed $k$-cycle if 3 does not divide $k$. Thus the degree condition in Lemma \ref{THM-shortcycle} (i) is the best possible. For cycles of cycle-type 0, Lemma \ref{THM-shortcycle} (ii) implies that we can find short antidirected cycles with a minimum semidegree $\gamma n$. Further, for antidirected cycles of every possible length, a minimum semidegree $(3/8+\varepsilon)n$ suffices due to Kelly, K\"uhn, and Osthus. Recently, Stein and Z\'arate-Guer\'en \cite{steinCPC33} gave an intermediate bound on minimum semidegree for antidirected cycles of medium length. 

\begin{theorem}[\cite{steinCPC33}]
For all $\varepsilon\in(0,1)$ there is $n_0$ such that for all $n\geqslant n_0$ and $k\geqslant \eta n$, every oriented graph $G$ on $n$ vertices with $\delta^0(G)>(1+\varepsilon)k/2$ contains any antidirected cycle of length at most $k$.
\end{theorem}    

\begin{theorem}\label{THM-expanderanylength}
Let  $0<1/n \ll \nu \ll \tau  < 1$. 
Let $G$ be an oriented graph on $n$ vertices which is a robust $(\nu,\tau)$-outexpander. If $\delta^0(G)\geqslant 0.346 n$, then $G$ contains a cycle of every possible orientation and of every possible length.
\end{theorem}
\begin{proof}

Choose $\varepsilon$ with $0<1/n\ll \varepsilon \ll \nu$.  By the arguments in the beginning of this section, Hamburger, Haxell and  Kostochka showed that $\delta^0(G)\geqslant 0.346 n$ ensures a directed 3-cycle and for a transitive 3-cycle, $\delta^0(G)\geqslant n/3+1$ suffices.  For the case $4\leqslant l \leqslant 1/\varepsilon+1$, by Lemma \ref{THM-shortcycle} (i) and the fact $\delta^0(G)\geqslant 0.346 n$, one may find $C$ in the oriented graph $G$. When $1/\varepsilon +1\leqslant l \leqslant n$,  Lemma \ref{LEM-Connect} shows that for any edge $xy$, there exists a path of every possible orientation with length $l-1$ from $x$ to $y$ and there is an embedding of $C$ in $G$.
\end{proof}

Now we are ready to prove
our pancyclicity result, i.e., Theorem \ref{THM-anyorianylength}. 

\begin{proof}[\textbf{Proof of Theorem \ref{THM-anyorianylength}.}]
Let $G$ be an oriented graph on $n $ vertices with $\delta^0(G)\geqslant  (3n-1)/8$  and let $C$ be a cycle of every possible orientation and every possible length. 

By Theorem \ref{THM-expanderanylength}, it suffices to consider the case that $G$ is not a robust outexpander. Then $V(G)$ has a partition $(W,X,Y,Z)$ satisfying \ref{EP1}-\ref{EP7} due to Theorem \ref{THM-extremal}. In particular, by \ref{EP3}-\ref{EP5}, there are $O(\delta n)$ bad vertices in $G$. Denote the set of these bad vertices by $B$. It follows by \ref{EP1} and \ref{EP4}  that $|W\backslash B|\geqslant n/5$ and $\delta^0(G[W\backslash B])\geqslant 4|W\backslash B|/9$. Then Lemma \ref{LEM-semitoexpander} and Theorem \ref{THM-expanderanylength} imply that there is a copy of $C$ in $G[W\backslash B]$ if $|C|\leqslant n/5$. 

Thus it suffices to consider the case that  $|C|\geqslant n/5$. We first assume that $C$ contains only a few sinks and, by symmetry, we further assume $|X|\geqslant |Z|$. By Lemma \ref{LEM-assignmatch}, $V(G)$ has a $\delta$-extremal partition $(W^\prime,X^\prime,Y^\prime,Z^\prime)$ with $|X^\prime|\geqslant |Z^\prime|$ such that $E(X^\prime\cup Y^\prime,W^\prime\cup X^\prime)$ contains a matching $M$ of size $|X^\prime|-|Z^\prime|$. By the definition of $\delta$-extremal partition, we have $|X^\prime|-|Z^\prime|=O(\delta n)$. Fix the vertices in $V(M)$ and apply Lemma \ref{LEM-randompartition} to choose a subset in each of $W^\prime,X^\prime,Y^\prime,Z^\prime$ such that the resulting  $W^\ast,X^\ast,Y^\ast,Z^\ast$ and $G^{\ast}$ satisfying the following.

(i)   $|G^{\ast}|=|C|$ and $(W^\ast,X^\ast,Y^\ast,Z^\ast)$ is a $\delta$-extremal partition of $V(G^{\ast})$;

(ii) $J^{\ast}\subseteq J^\prime$ for each $J\in\{W,X,Y,Z\}$;

(iii) $0\leqslant |X^\ast|-|Z^\ast|\leqslant |X^\prime|-|Z^\prime|$;

It should be noted that $E(X^\ast\cup Y^\ast,W^\ast\cup X^\ast)$ still contains the matching $M$  as we fixed them first. Moreover, Lemma \ref{LEM-randompartition} shows that $G^{\ast}$ inherits the degree condition of $G$, that is, $\delta^0(G^{\ast}) \geqslant 3 |G^\ast|/8-O(\varepsilon |G^\ast|)$. Then there is a copy of $C$ in $G^\ast$ by Lemmas  \ref{LEM-matchlarge} and \ref{LEM-balbancecycle}.

A similar proof can be applied to the case that $C$ has many sinks. More precisely, we first apply Lemma \ref{LEM-2matching} to the partition $(W,X,Y,Z)$ of $V(G)$. If $G-v$ is isomorphic to the graph in Figure \ref{FIG-degreesharp}, then Theorem \ref{THM-anyorianylength} (ii) holds clearly. Thus we may assume that there are two disjoint special edges for $(W,X,Y,Z)$. Fixing the endvertices of these two edges and applying Lemma \ref{LEM-randompartition} again, the resulting oriented graph $G^{\ast}$ inherits the degree condition of $G$ and, the resulting partition $(W^{\ast},X^{\ast},Y^{\ast},Z^{\ast})$ is  $\delta$-extremal for $V(G^{\ast})$. Then the problem of embedding $C$ into $G$ is equivalent to the problem of finding a Hamilton cycle $C$ in $G^{\ast}$ and  we are done by applying Theorem \ref{THM-specialmanysink}.
\end{proof}

\section{Concluding remarks}\label{SEC-remark}

In this paper,  we determine the exact minimum semidegree threshold for a cycle of every possible orientation and of every possible length in oriented graphs,  which improves an approximate result of Kelly in \cite{kellyEJC18} and  solves a problem of H\"aggkvist and Thomason in \cite{haggkvist1997}  (also see page 2 in \cite{haggkvistJGT19}). Next we mention a few open questions and future directions for research.

Given a digraph $G$, let $\delta(G)$ denote the minimum degree of $G$ (i.e. the minimum number of edges
incident to a vertex) and set $\delta^{\ast}(G) := \delta(G) + \delta^+(G) + \delta^-(G)$. H\"aggkvist \cite{haggkvistCPC2} conjectured that every $n$-vertex oriented graph $G$ with $\delta^{\ast}(G)> (3n-3)/2$ contains a directed Hamilton cycle.  In \cite{kellyCPC17}, the  conjecture  was verified approximately by Kelly, K\"uhn and Osthus. Moreover, Woodall \cite{woodallPLMS24} proved the following digraph version of Ore's theorem: if $d^+(x)+d^-(y) \geqslant |G|$  for every two distinct vertices $x,y$  for which there is no edge from $x$ to $y$, then the given digraph $G$ has a directed Hamilton cycle. It seems to be an interesting problem to determine the exact $\delta^{\ast}(G)$ or Ore-type condition threshold for every possible orientation  of a Hamilton cycle in  oriented graphs and in general digraphs.

One may further consider the transversal generalization of Theorem \ref{THM-arbitraryori}.  For a collection
$\mathcal{G} = \{G_1,G_2,\ldots,G_m\}$  of not-necessarily distinct digraphs on the common vertex set $V$, an $m$-edge digraph $G$ on vertex set $V$ is called a \emph{transversal} of  $\mathcal{G}$ if the edges of $G$ intersect each $G_i$ exactly once. For results on transversal, we refer the reader to the survey paper \cite{sunarXiv2024}. It is well-known that every tournament contains a directed Hamilton path, and every strongly connected tournament contains a directed Hamilton cycle. Chakraborti, Kim, Lee and Seo \cite{chakrabortiC44}  
established transversal generalizations of these results. Further, they \cite{chakrabortiarXiv2024} established the transversal generalization of a result of Thomason in \cite{thomasonTAMS296}, providing the existence of every possible
transversal cycles in a collection of tournaments.  They showed that every system of $n$ tournaments with order $n$ contains a transversal of every orientation of a  Hamilton cycle except possibly the directed one.  In  \cite{chengarXiv2025}, Cheng, Li, Sun and Wang generalized Ghouila-Houri's theorem to transversal version and they posed a corresponding problem in oriented graphs.  Determining the exact degree threshold that ensures a transversal of every orientation of a  Hamilton cycle in a digraph or an oriented graph system is a natural  question. Indeed, we have the following  problems.

\begin{problem}
Let $\mathcal{G} = \{G_1,G_2,\ldots,G_n\}$ be a collection of oriented graphs with common vertex set
$V$ of size $n$. If $\delta^0(G_i) \geqslant(3n-1)/8$  for all $i\in[n]$, does $\mathcal{G}$ contain a transversal of every orientation of a  Hamilton cycle?
\end{problem}

\begin{problem}
Let $\mathcal{G} = \{G_1,G_2,\ldots,G_n\}$ be a collection of digraphs with common vertex set
$V$ of size $n$. If $\delta^0(G_i) \geqslant n/2+1$ for all $i\in[n]$, does $\mathcal{G}$ contain a transversal of every orientation of a Hamilton cycle?
\end{problem}

\bibliographystyle{plain}

\newpage    
\appendix

\section{A construction of $G$ in Figure \ref{FIG-degreesharp}}
\label{APPSEC-table1}

Now we give a construction of $G$ as follows:  
Let $W,X,Y,Z$ be a partition of $V(G)$. We first add all edges from $W$ to $X$, from $X$ to $Y$, from $Y$ to $Z$ and from $Z$ to $W$. Then we choose the tournament inside $W$ and $Y$ to be as regular as possible. 

Next we construct the bipartite tournament between $X$ and $Z$.  For the case $n=8s+1$, we have $|X|=|Z|=2s+1$. Let $X=\{x_1,x_2,\ldots,x_{2s+1}\}$ and $Z=\{z_1,z_2,\ldots,z_{2s+1}\}$. In this case we choose the bipartite tournament between $X$ and $Z$ such that $d^+(x_i,Z)= d^-(z_j,X)=s+1$ and $d^-(x_i,Z)= d^+(z_j,X)=s$ for $i,j\in[2s+1]$. This can be achieved by letting $N^+(x_i,Z)=\{z_i,z_{i+1},\ldots,z_{i+s}\}$ for each $x_i\in X$, where the subscript is taken mod $2s+1$. For the case that $n\neq 8s+1$, let $X=X_1\cup X_2$, $Z=Z_1\cup Z_2$ with $|X_1|=|Z_1|=s+1$ and  $|X_2|=|Z_2|\in\{s,s+1\}$ and let $X_1Z_1X_2Z_2$ be a blowup of the directed cycle of length 4.

Table \ref{TAB-degreesharp} will help readers to better check every vertex of $G$  has the correct indegree and outdegree by the construction.

\begin{table}[H]

\centering
\begin{tabular}{m{1.5cm}<{\centering}m{1.5cm}<{\centering}m{1.5cm}<{\centering}m{2cm}<{\centering}m{1.5cm}<{\centering}}
	\hline
	$n$  &  $\delta^0(G)$ & $|W|$ & $|X|=|Z|$ & $|Y|$ \\
	\hline
	$8s+1$ & $3s$ & $2s$ & $2s+1$ & $2s-1$ \\
	$8s+2$ & $3s$ & $2s$ & $2s+1$ & $2s$ \\
	$8s+3$ & $3s$ & $2s$ & $2s+1$ & $2s+1$ \\
	$8s+4$ & $3s+1$ & $2s+1$ & $2s+1$ & $2s+1$ \\
	$8s+5$ & $3s+1$ & $2s+1$ & $2s+2$ & $2s$ \\
	$8s+6$ & $3s+2$ & $2s+1$ & $2s+2$ & $2s+1$ \\
	$8s+7$ & $3s+2$ & $2s+1$ & $2s+2$ & $2s+2$ \\
	$8s+8$ & $3s+2$ & $2s+2$ & $2s+2$ & $2s+2$ \\
	\hline
\end{tabular}
\caption{ \small The sizes of $W,X,Y,Z$ and the minimum semidegree of $G$ according to the value of $n$ mod 8.}
\label{TAB-degreesharp}
\end{table}

\section{Proof of Lemma \ref{LEM-orientedreduced}}\label{APPSEC-proforien}

Let $R$ be a spanning oriented subgraph obtained from $R^\prime$ by deleting edges as follows. For every unordered pair $V_i,V_j$ of clusters, we delete edge $V_iV_j$ (if it lies in $R^\prime$) with probability $e_{G^\prime}(V_j,V_i)/(e_{G^\prime}(V_i,V_j)+e_{G^\prime}(V_j,V_i)).$ Otherwise, we delete $V_jV_i$ (if it lies in $R^\prime$). We interpret the probability $0$ if $V_iV_j,V_jV_i \notin E(R^\prime)$. Therefore, if $R^\prime$ contains at most one of the edges $V_iV_j,V_jV_i$ then we do nothing. We do this for all unordered pairs of clusters independently. For any fixed $V_i\in R$, we have 
\begin{align*}
\mathbb{E}[d_R^+(V_i)] &= \sum_{V_j\in V(R)\backslash \{V_i\}} \frac{e_{G^\prime}(V_i,V_j)}{e_{G^\prime}(V_i,V_j)+e_{G^\prime}(V_j,V_i)} \geqslant \sum_{V_j\in V(R)\backslash \{V_i\}}\frac{e_{G^\prime}(V_i,V_j)}{|V_i||V_j|}\\
& \geqslant\frac{|R|}{|G||V_i|} \left(\sum_{v\in V_i}(d_{G^\prime}(v) - |V_0|-|V_i|)\right)\\
& \geqslant \left(\frac{e_G(V_i,V(G))}{|G||V_i|}- (d+2\varepsilon)\right)|R|.
\end{align*}
Then Chernoff Bound 1 implies that $\mathbb{P}[|d^+_R(V_i)-\mathbb{E}[d^+_R(V_i)]|\leqslant \varepsilon |R|/2 ]\leqslant e^{-\varepsilon |R|/12}$. By symmetry, we have $$\mathbb{E}[d_R^-(V_i)]\geqslant \left(\frac{e_G(V(G),V_i)}{|G||V_i|}- (d+2\varepsilon)\right)|R|,$$ and $\mathbb{P}[|d^-_R(V_i)-\mathbb{E}[d^-_R(V_i)]|\leqslant \varepsilon |R|/2 ]\leqslant e^{-\varepsilon |R|/12}$. Using the union bound,  there is a spanning oriented subgraph $R\subseteq R^\prime$ satisfying (\ref{EQU-orientreduced}) with positive probability. Moreover, the ``in particular'' part follows immediately  from (\ref{EQU-orientreduced}).

\end{document}